\documentclass[12pt]{amsart}
\usepackage{amsmath,amssymb,amsbsy,amsfonts, times, amsthm}
\usepackage[foot]{amsaddr}
\usepackage{lipsum}
\usepackage{color}
\usepackage{graphics}
\usepackage{epsfig,psfrag,graphicx}
\usepackage{verbatim}
\usepackage{cancel}
\usepackage[normalem]{ulem}
\usepackage{colonequals}
\usepackage{soul}
\usepackage{mathtools}
\usepackage{enumitem}

\usepackage{hyperref}
\usepackage{parskip}
\usepackage{eepic,epic}
\usepackage{pgf,tikz}
\usetikzlibrary{arrows}
\usetikzlibrary{shapes,decorations}
\usepackage{multirow}

\newcommand{\cred}{\color{black}}
\newcommand{\cblue}{\color{black}}

\newcommand{\mk}{\mbox{$m\big(k_{\mu_t}\big)\!(x)$}}
\newcommand{\mhk}{\mbox{$m^h\big(k_{\mu_t}\big)\!(x)$}}

\newcommand{\vk}{\mbox{$v\big(k_{\mu_t}\big)\!(x)$}}
\newcommand{\vhk}{\mbox{$v^h\big(k_{\mu_t}\big)\!(x)$}}

\newcommand{\RR}{\mathbb{R}}
\newcommand{\RRD}{{\mathbb{R}^d}}

\newcommand{\CC}{\mathcal{C}}
\newcommand{\CCA}{\mathcal{C}^{1+\alpha}}

\newcommand{\ww}{\overline{w}}

\newtheorem{definition}{Definition}[section]
\newtheorem{theorem}{Theorem}[section]
\newtheorem{lemma}[theorem]{Lemma}
\newtheorem{proposition}[theorem]{Proposition}
\newtheorem{corollary}[theorem]{Corollary}

\newtheorem{example}{Example}

\theoremstyle{remark}
\newtheorem*{remark}{Remark}

\newcommand{\nc}{\newcommand}
\nc{\PP}{{\cal P}} 
\nc{\V}{{\cal V}} 
\nc{\M}{{\cal M}} 
\nc{\T}{{{\mathbb R}^n}}  
\nc{\D}{{\cal D}} 
\nc{\R}{{\mathbb R}} 
\nc{\N}{{\mathbb N}}
\nc{\Thh}{T_h(\mu_t^h)}
\nc{\Th}{T_h(\mu_t)}

\newcommand{\dd}{{\,\mathrm{d}}}

\newcommand{\ZZZ}{\mathcal{Z}}
\newcommand{\NN}{\mathbb{N}}

\newcommand{\MM}{\mathcal{M}}
\newcommand{\MMR}{\mathcal{M}^+(\mathbb{R}^d)}

\newcommand{\CCAS}{\left(\mathcal{C}^{1+\alpha}(\mathbb{R}^d)\right)^*}

\newcommand{\TV}{\mathrm{TV}}

\newcommand\footnoteref[1]{\protected@xdef\@thefnmark{\ref{#1}}\@footnotemark}

\nc{\weak}{\rightharpoonup}
\nc{\weakstar}{\stackrel{\ast}{\rightharpoonup}} 

\renewcommand{\div}{{\mathrm{div}_x}\,}

\def\vec#1{\boldsymbol{#1}}

\def\dd#1#2{\frac{\d #1}{\d #2}}

\def\b0{\vec{0}}

\renewcommand{\div}{{\rm div}}

\newcommand{\Eehz}{e^{\int_0^t \overline{w}(s, y; 0)\dd s}}

\def\XXint#1#2#3{{\setbox0=\hbox{$#1{#2#3}{\int}$ }
\vcenter{\hbox{$#2#3$ }}\kern-.6\wd0}}

\def\dd{{\rm d}}

\title[Differentiability of measure solutions to perturbed transport equation]{Existence and differentiability in parameter of the measure solution to a~perturbed non-linear transport equation 
}

\author{Piotr Gwiazda$^1$}
\author{Sander C. Hille$^2$}
\author{Kamila \L yczek$^3$}

\address{$^1$Institute of Mathematics, Polish Academy of Sciences, \'Sniadeckich 8, 00-656 Warszawa, Poland} 

\address{$^2$Mathematical Institute, Leiden University P.O. Box 9512, 2300 RA Leiden, The Netherlands} 

\address{$^3$Institute of Applied Mathematics and Mechanics, University of Warsaw, Banacha 2,\newline 02-097 Warszawa, Poland} 

\begin{document}
\maketitle

\begin{abstract}
We consider a perturbation in the non-linear transport equation on measures i.e. both initial condition $\mu_0$ and the solution $\mu_t^h$ are bounded Radon measures $\MM(\RRD)$. The perturbations occur in the velocity field and also in the right-hand side scalar function. It is shown that the solution is differentiable with respect to the perturbation parameter $h$ i.e. that derivative is an element of a proper  Banach space. This result extends~\cite{lycz:2018} which considered the linear transport equation.  The proof exploits approximation of the~non-linear problem which is based on the study of  the linear equation. 
\end{abstract}

\section{Introduction}
Dependence of  the solution of a partial differential equation (PDE) on coefficients of the equation is a natural question to consider which is commonly investigated when a~solution remains unknown.  
Lipschitz dependence on initial conditions in the space of measures were investigated in {\it structured population models}~\cite{Carillo:2015, Czochra:2010,Czochra:2010_2,carillo:2012,Rosi:2016}, {\it dynamics of
system}~\cite{Carillo:2011,Gango:2013,MHerty:2014,Fornasier:2018a,Fornasier:2018} and {\it vehicular traffic 
flow}~\cite{Degond:2008,Evers:2016,Goatin:2016,Goatin:2017}.  All of mentioned applications involve a transport equation on measures, and have been widely investigated. 

 Differentiability with respect to perturbation parameters in PDE systems is important in the context of e.g. application in
optimal control theory or linear stability -- Lipschitz dependence is not enough.
Previous considerations concerning the transport equation in the space of measures did not allow to analyze the differentiability of solutions with respect to a perturbation of the~system~\cite{Ambrosio_lect,Thieme:2003,Picc:2014,Picc:2017}, see Example~\ref{example}.

In this paper we consider the initial value problem for the non-linear transport equation in conservative form
\begin{equation}\label{def:prob_init}
\left\{ \begin{array}{l}
\partial_t\mu_t + \textrm{div}_x\Big(v(\mu_t)\: \mu_t\Big) = m(\mu_t)\:\mu_t \quad \textrm{in }(\CC_c^1([0,+\infty) \times \mathbb{R}^d))^{\ast}, \\
\mu_{t=0}= \mu_0 \in \mathcal{M}(\mathbb{R}^d), 
\end{array} \right.
\end{equation}
where $t\mapsto\mu_t$ is unknown. 

{\bf Notation:} By $\mathcal{M}(\RRD)$ we denote the vector space of bounded signed Radon measures on $\RRD$ and by $\MM^+(\RRD)$ the space of non-negative  bounded Radon measures.
By $(\cdot)^{\ast}$ we denote the topological dual to~$(\cdot)$, when the latter is equipped with a~suitable locally convex or norm topology. The space~$\mathcal{C}_c^1$ consists of continuously differentiable functions with compact support.

Cristiani, Piccoli and Tosin considered in~\cite{MR2769993} the homogeneous version of~\eqref{def:prob_init} ($m(\mu_t)\equiv 0$) on measures   to model {\it crowd dynamics}, which can be applied to evacuation simulation. Measure $\mu_t(x)$ was interpreted as a number of individuals at moment $t$ at position $x$.

In~\cite{MR2769993} the authors consider the following form of the velocity field
\begin{equation}\label{eq:crowd}
v(\mu_t)(x)=\int_{\RRD\setminus \{x\}} K(|y-x|)\:g(\alpha_{xy})\:\frac{y-x}{|y-x|}\: \dd\mu_t(y),
\end{equation}
 where $K:\RR_+\to\RR$. It was assumed that $K$ is compactly supported and describes how the individuals at $x$ position reacts to his neighbours, which are at $y$ position. The value $\alpha_{xy}$ is an angle between vector $y-x$ and a vector describing the direction of individual at $x$. A fact that when a human does not have eyes around the head and reacts only on surroundings in the field of vision, is described by the function $g:[-\pi,\pi]\to[0,1]$.
It is a known phenomenon that a strategically placed obstacle near an exit can speed evacuations~\cite{PhysRevE}. It~means that wide-open exits are not always the most efficient at speeding pedestrians through. This is because a wide-open space in front of exit is that individuals can approach from all sides, allowing the maximum number of pedestrians to enter into conflict at the exit. Reducing exit access with an~obstacle can pare down the severity of those conflicts. 
  One can introduce perturbation into the velocity field and consider to minimize time needed to evacuate all pedestrians. For studying such optimization problem it is highly convenient to have differentiable dependence of the (measure) solution on the perturbation parameter. It is one of the main results of this paper, next to existence of solution, that such differentiability holds indeed, in a~suitable H\"older space, see Theorem~\ref{main_hille}.
 
The list of references on optimal control problems concerning transport equation is steadily growing \cite{Gango:2013,Bongini:2017,Fornaster:2017,Degong:2017,Herty:2018a,Bonnet:2018a}.
 Solving optimal control problem when a solution is not differentiable with respect
to the control parameter, we can proceed twofold. The first approach is to
use non-smooth analysis. The other way to go is to concentrate on regularity of
solutions and discuss what regularity of coefficients can provide sufficient regularity
of solutions to use smooth analysis methods anyway. In particular, in this case
there are developed significantly more tools, which are less numerically complex
than non-smooth methods.

At the problem~\eqref{def:prob_init} we can expect that $\mu_t$ (solution at fixed moment $t$) is a measure. This suspicion will prove to be true; see Corollary~\ref{cry:measure}.
By the velocity field  $v$ and scalar right-hand side function $m$ we mean, similar to~\eqref{eq:crowd}:
\begin{equation}\label{eq:v_K_form}
x\mapsto v\!\left(\int_\RRD K_v(y,x)\:\dd \mu_t(y)\right), \qquad  x\mapsto m\!\left(\int_\RRD K_m(y,x)\:\dd \mu_t(y)\right),
\end{equation}
where $K_v(y, x), K_m(y, x):\RRD\times\RRD\to\RR$ and $v:\RR\to\RRD$, $m:\RR\to\RR$. Let us denote
\begin{equation}\label{eq:kernel_k}
k_{\mu}(x):=\int_\RRD K_{u}(y,x)\:\dd\mu(y), \qquad \textrm{for }u\in\{v,m\}.
\end{equation}
Since assumptions on $K_v$, $K_m$ will be the same, we will not distinguish them in the notation~$k_{\mu}$. Instead of $v(\mu)$ and $m(\mu)$ we will write $v(k_\mu)(x)$ and $m(k_\mu)(x)$, if the structure~\eqref{eq:v_K_form} need to be emphasized.


{\bf Notation:}  Space $\CC_b$ means continuous and bounded functions. 
 Space
$\CC^{m+\alpha}$ consists of bounded functions whose partial derivatives up to order~$m$ are bounded and continuous and additionally the derivative of order $m$ is also H\"older continuous with exponent $\alpha$, where $0 < \alpha \leq 1$.
 
Recall that a function $f:\RRD\to\RR$ is H\"older continuous if the H\"older supremum
\[|f|_\alpha:=\sup_{\substack{x_1 \neq x_2\\ x_1, x_2 \in \RRD}}
 \frac{|f(x_1) - f(x_2)|}{|x_1-x_2|^\alpha}<\infty.\]
The H\"older space $\CCA(\RRD)$ is a~Banach space with the norm
\begin{equation*}
\|f\|_{\CCA(\RRD)} := \sup_{x\in\RRD}|f(x)| +\ \sup_{x\in\RRD} |\nabla_x f(x)| \ +\
|\nabla_x f|_\alpha.
\end{equation*}
Moreover, one has
\begin{align}
\label{eqn:CCA_fg}
\textrm{if } f, g\in\CCA(\RRD),&\textrm{ then }\|f\cdot g\|_{\CCA(\RRD)}\leq \|f\|_{\CCA(\RRD)}\|g\|_{\CCA(\RRD)};\\
\label{eqn:alpha_infty}
\textrm{if } f, g\in\CC^\alpha(\RRD),&\textrm{ then }|f\cdot g|_\alpha \leq |f|_\alpha \|g\|_\infty + \|f\|_\infty |g|_\alpha;\\
\label{eqn:alpha}
\textrm{if }f\in\CC^\alpha(\RRD), g\in\CC^1(\RRD;\RRD), &\textrm{ then } 
|f\circ g|_\alpha \leq |f|_\alpha \: \|\nabla g\|_\infty^\alpha. 
\end{align}

A~mapping $ t \mapsto \mu_t:[0,+\infty)\to\MM(\RRD)$ is  {\it narrowly continuous} if 
$t \mapsto \int_{\RRD} f \dd \mu_t$ is a~continuous function for all $f\in\CC_b(\RRD)$.

\begin{definition}[Weak solution to non-linear problem]\label{def:weak_solu_22}
 We say that a narrowly continuous curve  $[0,+\infty)\to\MM(\RRD): t \mapsto \mu_t$  is a weak solution to~\eqref{def:prob_init}  if 
\begin{equation*}\label{eqn:weak_solu}
\begin{split}
&\int_0^{\infty} \int_\RRD \partial_t\varphi(t,x) +\vk\:\nabla_x \varphi(t,x)\: \dd\mu_t(x)\dd t  + \int_\RRD\varphi(0,\cdot)\: \dd \mu_0(x)=\\
=&\int_0^{\infty} \int_\RRD \mk\:\varphi(t,x) \:\dd\mu_t(x) \dd t,\\
\end{split}
\end{equation*}
holds for all test functions $\varphi \in \CC_c^1([0,+\infty) \times \RRD)$.
\end{definition}
Let us consider the following perturbation of the coefficients, for the real perturbation parameter~$h$:
\begin{equation}
\begin{split}
\vhk\colonequals& v_0 \left(\int_{\RRD} K_{v_0}(y, x)\:\dd{\mu_t}(y)\right)+h \: v_1\left(\int_{\RRD} K_{v_1}(y, x)\:\dd{\mu_t}(y)\right)\label{pert_non},
\\
\mhk\colonequals& m_0 \left(\int_{\RRD} K_{m_0}(y, x)\:\dd{\mu_t}(y)\right)+h \: m_1\left(\int_{\RRD} K_{m_1}(y, x)\:\dd{\mu_t}(y)\right).
\end{split}
\end{equation}
%
%
The perturbed problem corresponding to~\eqref{def:prob_init} has the form\begin{equation}\label{def:prob_init2}
\left\{ \begin{array}{l}
\partial_t\mu_t^h + \textrm{div}_x\Big( v^h\!\left(k_{\mu_t^h}\right)\!\!(x)\;\mu_t^h \Big) = m^h\!\left(k_{\mu_t^h}\right)\!\!(x)\;\mu_t^h \quad \textrm{in }(\CC_c^1([0,+\infty) \times \RRD))^{\ast}, \\
\mu_{t=0}^h= \mu_0 \in \mathcal{M}(\RRD) .
\end{array} \right.
\end{equation}
In this paper we investigate existence of weak solution~$t\mapsto \mu_t^h$ and then  differentiability of the mapping $h\mapsto\mu_t^h$. Notice that the initial condition in~\eqref{def:prob_init2} is not perturbed.

The duality of $\MM(\RRD)$ with $W^{1,\infty}(\RRD)$ is a~natural setting when considering the transport equation in the space of bounded Radon measures  \cite{Czochra:2010}, \cite{Czochra:2010_2}. The metric $\rho_F$ on $\MM(\RRD)$, called the~{\it flat metric} (or {\it bounded Lipschitz distance}), is defined as follows:
\[\rho_F(\mu, \nu):= \sup_{f \in W^{1,\infty}, \Vert f \Vert_{W^{1, \infty}}\leq 1}\left\lbrace \int_{\RRD} f\:\dd(\mu-\nu)\right\rbrace.\]
Below we present a counterexample, that motivates why differentiability in parameter $h$ of the~weak solution cannot be obtained in $(\MM(\RRD),\rho_F)$. 

\begin{example}\label{example}\cite{Skrz:2018}
Consider one-dimensional linear transport equation, i.e. $x\in \RR:$
\begin{equation}\label{counter_example}
\left\{ \begin{array}{l}
\partial_t \mu_t^h + \partial_x((1+h)\mu_t^h)=0,\\ \mu_0^h=\delta_0\in\MM(\RR),
\end{array}\right.
\end{equation}
where by $\delta_x$ we denote the Dirac measure concentrated at~point~$x$.
It can be easily checked that $\mu_t^h=\delta_{(1+h)t}$ is a weak solution to ~\eqref{counter_example} in the sense of Definition~\ref{def:weak_solu_22}. Note that the map $h\mapsto \mu_t^h$ is Lipschitz continuous i.e.
\[\rho_F(\mu_t^h,\mu_t^{h'})=\rho_F(\delta_{(1+h)t}, \delta_{(1+h')t})\leq |h-h'|t.\]
However,  it is not differentiable for $\rho_F$.
If $\frac{\mu_t^h-\mu_t^0}{h}$ were convergent as $h\to0$, it would satisfy the Cauchy condition with respect to $\rho_F$. One obtains
\begin{equation*}
\begin{split}
&\int_{\RR}f(x)\Big(\frac{\dd\mu_t^{h_1}(x)-\dd\mu_t^0(x)}{h_1} - \frac{\dd\mu_t^{h_2}(x)-\dd\mu_t^0(x)}{h_2} \Big)=\\
&=\frac{f((1+h_1)t)-f(t)}{h_1}-\frac{f((1+h_2)t)-f(t)}{h_2}.
\end{split}
\end{equation*}
If we choose a test function from $W^{1,\infty}(\RR)$ such that
\[f(x)=
\begin{cases}
|x-t|-1, & \textrm{if } |x-t| \leq 1,\\
0,	& \textrm{if } |x-t|>1,
\end{cases}\]
then for $h_1>0$ and $h_2<0$, we get $\rho_F\left(\frac{\mu_t^{h_1}-\mu_t^0}{h_1}, \frac{\mu_t^{h_2}-\mu_t^0}{h_2} \right) \geq 2t$.
Thus $\frac{\mu_t^h-\mu_t^0}{h}$ does not converge.
Therefore, differentiability in parameter requires a topology defined by test functions that must be more regular than $W^{1,\infty}(\RR)$.
\end{example}
A class of functions more regular than $W^{1,\infty}$, that suffices for our needs is $\CCA$. Now, we need to introduce a space which contains the derivative of $\mu_t^h$ at fixed $h$ and $t$.
Following~\cite{lycz:2018} let us define a space $\ZZZ$ as a~predual to $\CCA(\RRD)$. It will turn out that $\partial_h\mu_t^h|_{h=h_0}$ exits in $\ZZZ$. 
The space of Radon measures $\mathcal{M}(\RRD)$ inherits the dual norm of $(\CCA(\RRD))^*$ by means of~embedding the former into the latter, where a measure is identified with the functional defined by  integration against the measure. Throughout we identify the former with the subspace of~$(\CCA(\RRD))^*$. Let then
\begin{equation}\label{def:space_Z}
\ZZZ\colonequals \overline{{\mathcal{M}}(\RRD)}^{(\CCA(\RRD))^*},
\end{equation}
which is a Banach space equipped with the dual norm $\Vert \cdot \Vert_{(\CCA(\RRD))^*}$. Unless stated otherwise we equip $\MM(\RRD)$ with the topology induced by $(\CCA(\RRD))^*$.

The main result of this paper concerns the regularity of mapping $h\mapsto(t\mapsto\mu_t^h)$ and for this we need to introduce a proper class of regularity.
 It will turn out that the solution $t\mapsto\mu_t^h$ is of a class $\CC([0,\infty);\MM(\RRD))$, where $\MM(\RRD)$ is equipped with a relative topology of $(\CCA(\RRD))^*$. The boundedness of the mapping $t\mapsto\mu_t^h$ will be necessary. Unfortunately $t\mapsto\|\mu_t^h\|_\ZZZ$ is not bounded. However, the growth of $\mu_t^h$ can be controlled in time by an exponential function; see Proposition~\ref{prop:mu_0_TV}. It means that $\|\mu_t\|_\ZZZ\leq \exp(gt)\|\mu_0\|_\TV$ for some constant $g>0$, and hence the value $\sup_{t\geq 0}\exp(-gt)\|\mu_t\|_\ZZZ$ is bounded.

For every $f:[0,+\infty)\to X$, where $(X, \|\cdot\|_X)$ is normed vector space, we define the $\omega$-weighted norm~as
\begin{equation}\label{eq:bielec_norm_I}
\Vert f \Vert_{X_\omega}:=\sup_{t\geq 0}\omega(t)\Vert f(t)\Vert_X,
\end{equation}
where $\omega(t)> 0$ is called a {\it weight} (this condition is necessary for~\eqref{eq:bielec_norm_I} to be a norm). By $\CC_{\omega}([0,+\infty);X)$ we denote the space of continuous functions valued in $X$ on $[0,\infty)$, such that the norm $\|f\|_{X_\omega}<\infty$. When $\omega(t)=e^{-gt}$ and $g>0$ is a constant, then~\eqref{eq:bielec_norm_I} is called the {\it Bielecki norm}~\cite{Kwap:1991}. The norm involved in the space $\CC([0,\infty);\ZZZ)$ will not be the Bielecki norm, but it is closely related; see weight~\eqref{eq:final_weig_final_3}.

The appropriate weight (which will be necessary not only for controlling growth of $\mu_t$ in time) is
\begin{equation}\label{eq:final_weig_final}
\tilde{\omega}(t):=\min\big\{\omega_1(t),\;\omega_2(t)\big\},\qquad \textrm{where}
\end{equation}
\begin{equation}\label{eq:final_weig_final_3}
\omega_1(t)=\left\{ \begin{array}{ll}
e^{-g_1 t}& \textrm{for }  0\leq t< 1, \\
e^{-g_2 t}& \textrm{for } 1\leq t< 2,\\
\cdots\\
e^{-g_N t} &\textrm{for } N-1\leq t<N,\\
\cdots\\
\end{array} \right.\qquad\textrm{and}
\qquad
\omega_2(t)=\mathcal{O}\left(\frac{1}{t^2}\exp(-Ct)\right),
\end{equation}
where the constant $g_N$ are positive for all $N\in\NN\setminus\{0\}$. The reason for such a form of $\omega_1$ follows from Proposition~\ref{prop:T_H_0_infty_00} and Proposition~\ref{prop:T_H_0_infty}. The weight $\omega_2$ results from Theorem~
\ref{lem:main_2}.

Now, we are ready to state the main result of this paper.

\begin{theorem}\label{main_hille}
Assume that for $i=0,1$
\begin{description}
\item[A1] $v_i(\cdot)\in\CC^{3+\alpha}(\RR;\RRD)$,
\item[A2] $m_i(\cdot)\in\CC^{3+\alpha}(\RR;\RR)$,
\item[A3] $\left(x \mapsto K_{v_i}(\cdot, x)\right)$, $\left(x \mapsto K_{m_i}(\cdot, x)\right)$ $\in \CCA(\RRD; \CCA(\RRD))$,
\item[A4] $y \mapsto K_{v_i}(y,\cdot)$, $y \mapsto K_{m_i}(y,\cdot): \RRD\to\CCA(\RRD)$ are weakly continuous and bounded.
\end{description} 
Let the coefficients be defined by~\eqref{pert_non}. 
Then problem~\eqref{def:prob_init2} has a unique weak solution $t\mapsto\mu_t^h$
and the mapping 
\[\big(h\mapsto(t\mapsto\mu_t^h)\big)\in\CC^1\big((-\tfrac{1}{2},\tfrac{1}{2});\CC_{\tilde{\omega}}([0,\infty);\ZZZ)\big),
\]
where $\tilde{\omega}(t)$ is defined by~\eqref{eq:final_weig_final}. In particular, for fixed $t$ the derivative $\partial_h \mu_t^{h}|_{h=h_0} \in \ZZZ.
$
\end{theorem}

\begin{remark}
In fact, the parameter $h$ can be considered in any bounded interval and Theorem~\ref{main_hille} remains true. The reason why only bounded intervals can be taken into account is that in the proof of Theorem~\ref{main_hille} we exploit the supremum norm of $m^h(k_{\mu_t^h})$. If the parameter~$h$ is not bounded,
then for fixed $h$ the scalar function $m^h(k_{\mu_t^h})$ is bounded if $m_0,m_1$ satisfy assumption~{\bf A2}, but the family of these functions is not uniformly bounded. This uniformly boundedness is necessary later. 
\end{remark}
Above theorem can be naturally extended for the velocity field explicitly dependent also on $x$ and $t$ i.e. $v(t,x,\mu_t)$, analogously $m(t,x,\mu_t)$. We investigate a dependency only on the solution~$\mu_t$, in order to clarify the line of reasoning and the way estimates are made.

The proof of Theorem~\ref{main_hille} is based on a~scheme approximating the solution $t\mapsto\mu_t^h$. As we will see in Section~\ref{sec:well_00} the approximating scheme involves a linear transport equation. Therefore, in Section~\ref{sec:linear} some results concerning a linear transport equation are recalled. We consider
\begin{equation}\label{eq:def_linear}
\left\{ \begin{array}{l}
\partial_t\nu_t + \textrm{div}_x\Big(b(t,x)\: \nu_t\Big) = w(t,x)\:\nu_t \quad \textrm{in }(\CC_c^1([0,+\infty) \times \mathbb{R}^d))^{\ast}, \\
\nu_{t=0}= \nu_0 \in \mathcal{M}(\mathbb{R}^d),
\end{array} \right.
\end{equation}
where $b:[0,\infty)\times\RRD\to\RRD$ and $w:[0,\infty)\times\RRD\to\RR$.
By $\nu_t$ we always denote the solution to the linear problem, in contrast to $\mu_t$ which refers to the non-linear equation. In~Section~\ref{sec:linear_perturb}, two types of perturbation are studied. The first one is linear $b_0(t,x)+hb_1(t,x)$ and $w_0(t,x)+hw_1(t,x)$. And the second one is non-linear 
\begin{equation}\label{def:pertur_non_lin}
b_0(t,x)+hb_1(t,x)+C(t)\mathcal{O}(h)\quad\textrm{ and }\quad w_0(t,x)+hw_1(t,x)+C(t)\mathcal{O}(h).
\end{equation} Differentiability of the solution $t\mapsto\nu_t$ with respect to the perturbing parameter~$h$ is discussed (in both types of the perturbation). In Section~\ref{sec:ZZZ} some properties of the space $\ZZZ$ are presented.

In Section~\ref{sec:nemy_0}, the regularity of a mapping 
\[(t\mapsto\mu_t) \mapsto\! \left(t\mapsto u\!\left(\int_\RRD K_u(y,x)\:\dd\mu(y)\right)(\cdot)\right)\]
is investigated, where $u\in\lbrace v,m\rbrace$. It is concluded by stating that under assumptions {\bf A1, A2, A3} of Theorem~\ref{main_hille} this mapping is of a~class $\CC^1\left(\CC_\omega([0,\infty);\ZZZ);\: \CC_\omega([0,\infty); \CCA(\RRD;\RRD))\right)$ for $u\in\{v_0,v_1\}$ and analogously for $m_0$, $m_1$. This is necessary further on and allows us treating perturbations~\eqref{pert_non} as a~much simpler case. Using an approximating scheme, the non-linear problem is essentially reduced to the linear equation with perturbation~\eqref{def:pertur_non_lin}.

Later on, in Section~\ref{sec:well_00}, the scheme  approximating the weak solution $\mu_t^h$ is constructed. By contraction argument we conclude that~\eqref{def:prob_init2} has a weak solution.
It is also argued that if $\mu_0\in\MM(\RRD)$, then the solution $\mu_t^h$  is also a signed Radon measure -- this is not obvious, since $\MM(\RRD)$ is not complete with the norm $\|\cdot\|_\ZZZ$. 
In Section~\ref{sec:well} a simple modification of the approximating scheme introduced in Section~\ref{sec:well_00} is described, which is convenient to investigate quotients $\frac{\mu_t^h-\mu_t^0}{h}$.

Finally, in Section~\ref{sec:linear_tr} all elements are combined to complete the proof of Theorem~\ref{main_hille}.

\section{Preparation}\label{sec:linear}
This section is a brief overview of results concerning {\bf differentiability with respect to parameter of the solution to linear transport equation} on measures. By $\nu_t$ we denote the solution to linear transport equation, in contrast to $\mu_t$ which is a solution to non-linear problem.

First, we state some facts about the non-perturbed linear transport equation, in particular the explicit formula for its solution.
Then, a brief review of \cite{lycz:2018} is made i.e. linear perturbation in the linear transport equation. 
In contrast to~\cite{lycz:2018} we present the dependence of the derivative $\partial_h\nu_t^h$ on time. We also discuss some properties of the~space $\ZZZ$ and the~weighted space $\CC_\omega([0,\infty);\ZZZ)$. At the end of this section, a~{\bf non-linear perturbation} in the transport equation is discussed, see Theorem~\ref{lem:main_2}, which has not been studied in~\cite{lycz:2018}.

\subsection{Linear equation with no perturbation}
Consider the linear transport equation~\eqref{eq:def_linear}.
\begin{definition}[Weak solution to linear problem]\label{def:weak_solu}
 We say that the narrowly continuous curve $[0,+\infty)\to\MM(\RRD): t \mapsto \mu_t$ is a weak solution to~\eqref{eq:def_linear}~if 
\begin{equation*}\label{eqn:weak_solu}
\begin{split}
&\int_0^{\infty} \int_{\RRD} \big(\partial_t\varphi(t,x) +b(t,x)\nabla_x \varphi(t,x)\big)\: \dd \nu_t(x) \dd t  + \int_{\RRD} \varphi(0,\cdot) \:\dd \nu_0(x)= \\
=&\int_0^{\infty} \int_{\RRD}w(t,x)\varphi(t,x)\:\dd \nu_t(x) \dd t,\\
\end{split}
\end{equation*}
holds for all test functions $\varphi \in \CC_c^\infty([0,\infty) \times \RRD)$.
\end{definition}

A standard method to solve the transport equation in a classical sense is the method of characteristics. When the initial condition $\nu_0\in\MMR$, and if the velocity field $b$ and right-hand side $w$ are {\it sufficiently regular}, then the method of characteristics  can be applied in the measure setting too. The characteristic system associated to equation~\eqref{eq:def_linear} has the form
\begin{equation}\label{def:diff_init0}
\left\{ \begin{array}{ll}
{\dot X_b}(t,x)&=b\left(t, X_b(t,x)\right),\\
X_b(t_0, x)&=x \in \RRD.
\end{array} \right.
\end{equation}
A solution to (\ref{def:diff_init0}), $X_b$ is called a {\it flow map}. Note that if $b$ is Lipschitz continuous, the flow maps $X_b$ are defined for all $t \in \mathbb{R}$ and thus $x \mapsto X_b(t, x)$ is a one-parameter family of diffeomorphisms on $\RRD$ (dependent on the~variable~$b$). 

Let us define the {\it push-forward operator}~\cite{Ambrosio_lect}, which is needed to formulate an explicit formula for solution to~\eqref{eq:def_linear}. If $Y_1$, $Y_2$ are separable metric spaces, $\nu \in \mathcal{M}^+(Y_1)$, and $r\colon  Y_1 \to Y_2$ is a $\nu$-measurable map, we denote by $\nu \mapsto  r\# \nu \in \mathcal{M}^+(Y_2)$ the~push-forward of $\nu$ through $r$, defined by
\begin{equation}\label{eq:push}
r\# \nu(\textrm{B})\colonequals  \nu(r^{-1}(\textrm{B})), \qquad \textrm{ for all }  \textrm{B}\in \mathcal{B}(Y_2) 
.\end{equation}

 Consequently, one can also consider signed measures. Since Hahn-Jordan Decomposition Theorem push-forward can work separately on positive and negative parts of a signed measure and gives as a result sum of positive and negative measures.

The following lemma yields existence and uniqueness of a weak solution to~\eqref{eq:def_linear}.

\begin{lemma}[A representation formula for the linear equation] \label{lem:repr_form}\cite[Proposition 3.6]{Mani:2007}
 Let the mapping $t\mapsto b(t,\cdot)$ be a Borel velocity field in $L^1([0,T]; W^{1, \infty}(\RRD;\RRD))$, $w(t,x)$ be a Borel measurable function, bounded and for fixed $t$ locally Lipschitz continuous with respect to $x$. Let $\nu_0\in\mathcal{M}(\RRD)$. Then there exists a unique, narrowly continuous family of finite Borel  measures solving (in the distributional sense) the initial value problem~\eqref{eq:def_linear} and it is given by the explicit formula
\begin{equation}\label{repr:form}
\nu_t=X_b(t,\cdot)\#( e^{\int_0^t w(s, X_b(s,\cdot))\:\dd s}  \:\nu_0), \qquad  \textrm{ for all } t \in[0,T].
\end{equation}
Moreover,
\begin{equation}\label{eq:man_2}
\nu_t=X(t,0,\cdot)\#\nu_0 +\int_0^t X(t,s,\cdot)\#(w(s,\cdot)\nu_s)\:\dd s\qquad \textrm{ for all } t \in[0,T].
\end{equation}
provides an implicit formula for $\nu_t$.
\end{lemma}
Lemma~\ref{lem:repr_form} states in particular that the solution $\nu_t$ is an element of $\MM(\RRD)$. Notice that the representation formula yields that sign is preserved i.e. if $\nu_0\in\MM^+(\RRD)$  then $\nu_t\in\MM^+(\RRD)$.

\begin{remark}
Representation formula proposed in~\cite[Proposition 3.6]{Mani:2007} has identical form as~\eqref{repr:form} but it was originally formulated for a problem when initial condition $\nu_0$ is an element of $\mathcal{P}(\RRD)$. The proof of Lemma~\ref{lem:repr_form} remains essentially the same as in~\cite[Proposition 3.6]{Mani:2007}.
\end{remark}

\begin{corollary}\label{coro:conti_time}
If $\nu_\bullet$ is a weak solution to the linear transport equation~\eqref{eq:def_linear} with $b$ and $w$ as in Lemma~\ref{lem:repr_form}, then $\nu_\bullet:[0,\infty)\to\ZZZ$ is  a~locally Lipschitz mapping. In particular, it is continuous.
\end{corollary}
\begin{proof}
The mapping $\nu_\bullet$, as a weak solution to the linear problem, has the explicit representation~\eqref{repr:form}. Let us denote
$\psi(t,\cdot):=\exp(\int_0^tw(s,X_b(s,\cdot))\dd s)$. 
 Let $\varphi\in\CCA(\RRD)$ and take ${\overline{t}}\in[0,\infty)$. If $t\in[0,\infty)$ and $t\neq{\overline{t}}$ then
\begin{equation*}
\begin{split}
&|\langle \nu_t-\nu_{{\overline{t}}} ,\varphi\rangle|=\left|\int_\RRD\varphi(X_b(t,x))\psi(t,x)-\varphi(X_b({\overline{t}},x))\psi({\overline{t}},x)\dd\nu_0(x)\right|\\
&\leq \int_\RRD\!\!\left|\varphi(X_b(t,x))-\varphi(X_b({\overline{t}},x))\right|\psi(t,x)\dd\nu_0(x)+\int_\RRD\!\!|\varphi(X_b({\overline{t}},x))|\: \big|\psi(t,x)-\psi({\overline{t}},x)\big|\dd\nu_0(x).
\end{split}
\end{equation*} 
Without lost of generality we may assume that $|t-{\overline{t}}|<1$. Then one has
\[|\psi(t,x)|\leq \exp\left(\int_0^tw^+(s,X_b(s,x))\dd s\right)\leq \exp\big(\|w^+\|_\infty({\overline{t}}+1)\big)=:C_1
\]
and hence, writing $\int_0^T\|b\|_\infty \dd s$ as a shortcut for $\int_0^T\|b(s,\cdot)\|_\infty\dd s$,
\begin{equation*}
\begin{split}
&|\langle \nu_t-\nu_{{\overline{t}}},\varphi\rangle|\leq\\
&\leq \|\nu_0\|_{TV}\left(C_1\|\varphi(X_b(t,\cdot))-\varphi(X_b({\overline{t}},\cdot))\|_\infty+\|\varphi\|_\infty\|\psi(t,\cdot)-\psi(\overline{t},\cdot)\|_\infty\right)\\
&\leq\|\nu_0\|_{TV}\left(C_1\|\nabla\varphi\|_\infty\|X_b(t,\cdot)-X_b({\overline{t}},\cdot)\|_\infty
+\|\varphi\|_\infty\|\psi(t,\cdot)-\psi({\overline{t}},\cdot)\|_\infty
\right)
\\
&\leq \|\nu_0\|_{TV}\left(C_1\int_0^T\!\!\|b\|_\infty\dd s\:\|\nabla\varphi\|_\infty|t-{\overline{t}}|+\|\varphi\|_\infty\|\psi(t,\cdot)-\psi({\overline{t}},\cdot)\|_\infty\right)\\
&\leq
\|\nu_0\|_{TV}\left(C_1\int_0^T\!\!\|b\|_\infty\dd s\:\|\nabla\varphi\|_\infty|t-{\overline{t}}|+\exp(\|w\|_\infty\max\{t,\overline{t}\})\|w\|_\infty\|\varphi\|_\infty|t-{\overline{t}}|\right).
\end{split}
\end{equation*}
Thus, taking supremum over $\varphi$ in the unit ball of $\CCA(\RRD)$, one obtains
\begin{equation}\label{eq:lipsh_time}
\|\nu_t-\nu_{{\overline{t}}}\|_\ZZZ\leq \Big(C_1\int_0^T\|b\|_\infty\dd s+\exp(\|w\|_\infty\max\{t,\overline{t}\})\|w\|_\infty\Big)\|\nu_0\|_{TV}|t-{\overline{t}}|
\end{equation}
what means that $\nu_\bullet$ is a  locally Lipschitz continuous. In particular, it is continuous mapping.
\end{proof}
Let us assume that
\begin{equation}\label{eq:assum_b}
\begin{split}
\big(t \mapsto b(t, \cdot)\big)& \in \CC_b\left([ 0, +\infty);\; \CCA(\RRD;\RRD)\right),\\
\big(t \mapsto w(t, \cdot)\big)& \in \CC_b\left([ 0, +\infty);\; \CCA(\RRD;\RR)\right).
\end{split}
\end{equation}
These assumptions are necessary to assume that the solution is differentiable with respect to the perturbating parameter (see Theorem~\ref{th:main} and Theorem~\ref{tw:ciagla_pocho}). Under assumption~\eqref{eq:assum_b} the velocity field $b$ is globally Lipschitz and hence there exists a global unique solution to~\eqref{def:diff_init0} what implies that~\eqref{eq:def_linear} also has unique solution.

 Let us notice the following dependence of $X_b$ on $b$. 
\begin{lemma}\label{lem:X_estim}
Let $b(t, \cdot)$, $\overline{b}(t,\cdot)\in \CCA(\RRD;\RRD)$ and let $X_b$ and $X_{\overline{b}}$ be the solution to problem~\eqref{def:diff_init0} corresponding to~$b$ and $\overline{b}$. Then the following estimates holds
\[\left| X_b(t,x) - X_{\overline{b}}(t,x)\right| \leq  \exp(t\: C_t[b, \overline{b}]) \int_0^t \|b(s, \cdot)- \overline{b}(s,\cdot)\|_\infty\dd s,\]
where
\begin{equation}\label{eq:constant_Ct}
C_t[b, \overline{b}]:=  \min \left\lbrace \sup_{0\leq s \leq t} \Vert \nabla_x b(s, \cdot) \Vert_\infty, \sup_{0\leq s \leq t} \Vert \nabla_x \overline{b}(s, \cdot) \Vert_\infty \right\rbrace.\end{equation}
\end{lemma}
\begin{proof}
One has
\begin{equation*}
\begin{split}
&\Big| X_b(t, x)-X_{\overline{b}}(t,x) \Big|=\\
&= \Big| x +\int_0^t b(s, X_b(s, x))\:\dd s - x -\int_0^t \overline{b}(s, X_{\overline{b}}(s, x))\:\dd s\Big| \\
&\leq  \int_0^t| b(s, X_b(s,x))-\overline{b}(s, X_{\overline{b}}(s,x))| \dd s\\
& \leq 
\int_0^t| b(s, X_b(s,x))-\overline{b}(s, X_{b}(s,x))| \dd s + \int_0^t| \overline{b}(s, X_b(s,x))-\overline{b}(s,  X_{\overline{b}}(s,x))| \dd s\\
& \leq 
\int_0^t \|b(s,\cdot)-\overline{b}(s,\cdot)\|_\infty \dd s + \int_0^t \|\nabla_x\overline{b}(s,\cdot)\|_\infty |X_b(s,x)-X_{\overline{b}}(s,x)|\dd s\\
& \leq 
\int_0^t \|b(s,\cdot)-\overline{b}(s,\cdot)\|_\infty \dd s + \sup_{0\leq s \leq t}\|\nabla_x\overline{b}(s,\cdot)\|_\infty\int_0^t  |X_b(s,x)-X_{\overline{b}}(s,x)|\dd s.
\end{split}
\end{equation*}
By Gronwall Lemma and observing that the roles of $b$ and $\overline{b}$
can be interchanged, the proof is
completed.
\end{proof}

\subsection{Space $\ZZZ$ and weighted space}\label{sec:ZZZ}
Recall the definition of the space $\ZZZ$ in~\eqref{def:space_Z}, following~\cite{lycz:2018}.
Throughout we identify $\MM(\RRD)$ with the subspace of~$(\CCA(\RRD))^*$. 
The norm on~$\ZZZ$ (and on $\MM(\RRD)$) is $\Vert \cdot \Vert_{(\CCA(\RRD))^*}$, which will be denoted by $\|\cdot\|_\ZZZ$. 
Below some properties of the space $\ZZZ$ are presented. 
\begin{proposition}\cite[Proposition 4.1]{lycz:2018}\label{prop:separable}
The set $\mathrm{span}_{\mathbb{Q}}\{\delta_x\colon  x\in\mathbb{Q}^d \}$ is dense in $\ZZZ$ with respect to  the $\CCAS$-topology, i.e.
\[\ZZZ=\overline{\mathrm{span}_{\mathbb{Q}}\{\delta_x\colon x\in\mathbb{Q}^d \}}^{\CCAS}.\]
Consequently, $\ZZZ$ is a separable space.
\end{proposition}
\begin{proposition}\cite[Proposition 4.3]{lycz:2018}\label{prop:Z_isom}
The space $\ZZZ^*$ is linearly isomorphic to $\CCA(\RRD)$ under the map $\phi \mapsto T\phi$, where $T\phi(x)\colonequals \phi(\delta_x)$, $T\colon  \ZZZ^* \to \CCA(\RRD)$.
\end{proposition}
The following result was not stated in~\cite{lycz:2018}:
\begin{proposition}\label{prop:Z+}\label{completeZ+}
$\MMR$ is closed in $\CCAS$.
\end{proposition}
\begin{proof}
Let $\nu_n\in\MMR$ such that $\nu_n\to F$ in $\CCAS$. Then for every  $f\in\CC_c^1(\RRD)\subset\CCA(\RRD)$ and $f\geq 0$ one has $\langle f,\nu_n \rangle \geq 0$. This implies that $\langle f, F\rangle\geq 0$. A well-known result concerning distributions says: if a distribution $F$ is non-negative (in the sense that $\langle f,F \rangle\geq 0$ for all $f\in\CC_c^\infty$ such that $f\geq 0$), then $F$ is a non-negative Radon measure  (see \cite[Lemma~37.2]{tart:2007}). 
\end{proof}
%
%
For the investigation of the mapping $h\mapsto(t\mapsto\nu_t^h)$, we need a {\it weighted space}. A  motivation for this is that $t\mapsto\nu_t$ is a continuous mapping i.e. $t\mapsto\nu_t\in\CC([0,\infty);\MM(\RRD))$ with $\|\cdot\|_\ZZZ$, but it is not bounded with respect to time, generally. 
Boundedness of this mapping will play a key role in the result concerning non-linear transport equation. Hence, we introduce a new norm which makes this mapping bounded. Following proposition says about evolution of $\nu_t$ in time.
\begin{proposition}\label{prop:mu_0_TV}
 Let $\nu_0\in\MM(\RRD)$, velocity field $b$ satisfies~\eqref{eq:assum_b} and $w\in\CC_b([0,\infty)\times \RRD)$.
Then for every weak solution $\nu_t$ to~\eqref{eq:def_linear} the following estimate holds 
\[
\Vert \nu_t \Vert_{TV} \leq e^{t\|w^+(\cdot,\cdot)\|_\infty}\Vert \nu_0 \Vert_{TV},\qquad \textrm{for all } t\in[0,\infty),
\] where~$w^+$ is the positive part of $w$.
\end{proposition}
\begin{proof}
By the definition of $TV$-norm and using representation formula~\eqref{repr:form} one has
\begin{equation*}
\begin{split}
\|\nu_t\|_{TV}=&\sup_{f\in\CC_b(\RRD)}\Big\lbrace \int_{\RRD}f \:\dd \nu_t: \|f\|_{\infty}\leq 1\Big\rbrace\\
\leq&\exp(\|w^+(\cdot,\cdot)\|_\infty t )\sup_{f\in\CC_b(\RRD)}\Big\lbrace \int_{\RRD}f(X_b(t, \cdot))\:\dd \nu_0: \|f\|_{\infty}\leq 1\big\rbrace\\
\leq&\exp(\|w^+(\cdot,\cdot)\|_\infty t )\sup_{f\in\CC_b(\RRD)}\Big\lbrace \int_{\RRD}f\:\dd \nu_0: \|f\|_{\infty}\leq 1\big\rbrace\\
  \leq& \exp(\|w^+(\cdot,\cdot)\|_\infty t )\|\nu_0\|_{TV}, 
\end{split}
\end{equation*}
what finishes the proof.
\end{proof}

\begin{corollary}\label{cly:bound_nu}
Since for any $\nu\in\MM(\RRD)$ one has $\|\nu\|_\ZZZ\leq \|\nu\|_{TV}$, therefore
\begin{equation}\label{eqn:TV_Z_nu_0}
\|\nu_t\|_\ZZZ\leq \exp(\|w^+(\cdot,\cdot)\|_\infty t)\|\nu_0\|_{TV},\qquad \textrm{ for all } t\in[0,\infty).
\end{equation}
\end{corollary}
%
%
%
\begin{corollary}\label{prop:exp_omega_1}
Under assumptions of Lemma~\ref{lem:repr_form}, the weak solutions $t\mapsto \nu_t$ of linear problem~\eqref{eq:def_linear} are of a class $\CC_{\omega}([0,\infty);\MM(\RRD))$, where $\omega(t)=\exp(-\|w\|_\infty t)$.
\end{corollary}
\begin{proof}
The fact that $\nu_t$ are bounded Radon measures follows by Lemma~\ref{lem:repr_form}. 
By Corollary~\ref{coro:conti_time} one has that the solution is
locally Lipschitz continuous with respect to time, hence in particular it is continuous. By assumption $\nu_0\in\MM(\RRD)$ and estimate~\eqref{eqn:TV_Z_nu_0} one has, that the value $\|\nu_t\|_{\ZZZ_{\omega}}$ is bounded.
\end{proof}
\begin{remark}
The space $\CC_\omega([0,\infty);\ZZZ)$ is complete for any weight function $\omega(t)$ such that for all $t_0\in[0,\infty)$, there exists $\sigma>0$ and closed interval $[t_1,t_2]$ containing $t_0$ with $t_1<t_2$, such that $w(t)>\sigma$ for all $t\in[t_1,t_2]$.
\end{remark}

\subsection{Linear transport equation with perturbation}\label{sec:linear_perturb}
A brief review of \cite{lycz:2018} is presented i.e. linear perturbation in the linear transport equation. We extend the results of~\cite{lycz:2018} in Theorem~\ref{tw:ciagla_pocho} and Theorem~\ref{lem:main_2}.
Let us introduce a perturbation to the velocity field $b$ and into $w$ as follows:
\begin{equation}\label{def:pert}
 b^h(t, x) \colonequals b_0(t, x) + h \: b_1(t, x),\qquad \textrm{ and }\qquad  w^h(t, x) \colonequals w_0(t, x) + h \: w_1(t, x),
\end{equation}
where $b_i:[0,\infty)\times\RRD\to\RRD$ and $w_i:[0,\infty)\times\RRD\to\RR$ for $i=0,1$ and parameter
$h\in(-\frac{1}{2},\frac{1}{2})$. 
The perturbed problem has the form
\begin{equation}\label{def:prob_init2_lin}
\left\{ \begin{array}{l}
\partial_t\nu_t^h + \textrm{div}_x\Big( b^h( t, x)\; \nu_t^h \Big) = w^h(t,x)\;\nu_t^h \quad \textrm{in }(\CC_c^1([0,+\infty) \times \RRD))^{\ast}, \\
\nu_{t=0}^h= \nu_0 \in \mathcal{M}(\RRD). 
\end{array} \right. 
\end{equation}
Notice that the initial condition in~\eqref{def:prob_init2_lin} is not perturbed. Proposition~\ref{prop:mu_0_TV} remains true for solution $\nu_t^h$. 
The following theorem addresses the differentiability of $t\mapsto\nu_t^h$ with respect to parameter $h$.

\begin{theorem}\label{th:main}\cite[Theorem 1.2]{lycz:2018}
Assume that $\left(t \mapsto b_i(t, \cdot) \right)\in \CC_b\big([0, +\infty); \CCA(\RRD;\RRD)\big)$ and $\left(t \mapsto w_i(t, \cdot) \right)\in \CC_b\big([0, +\infty); \CCA(\RRD;\RR)\big)$ for $i=0,1$.
Let $\nu_t^h$ be the weak solution to problem~\eqref{def:prob_init2_lin} where $b^h$ and $w^h$ are  defined by~\eqref{def:pert}. Then the mapping
\[(-\tfrac{1}{2}, \tfrac{1}{2}) \ni h \mapsto \nu_t^{h} \in \mathcal{M}(\RRD)\]
 is differentiable in $\ZZZ$, i.e. for fixed $t$ the derivative $
\partial_h \nu_t^{h}|_{h=h_0} \in \ZZZ.
$
\end{theorem}
Lemma~\ref{lem:repr_form} provides that~\eqref{def:prob_init2_lin} has a unique global weak solution, which is a measure.
Theorem~\ref{th:main} is an analogue to~\cite[Theorem 1.2]{lycz:2018} -- it was originally formulated for $\nu_0\in\mathcal{P}(\RRD)$, where the linear perturbation occurred only in $b$, not in $w$. However, the proof can be easily extended to our settings.
The proof of~\cite[Theorem 1.2]{lycz:2018} is based on the fact that $\ZZZ$ is a complete space. Thus it is enough to show that a proper sequence of differential quotients is a~Cauchy sequence.
The proof is presented in Appendix~\ref{apdx:lin} on page~\pageref{apdx:lin}. In contrast to~\cite{lycz:2018} we will point out a dependence of constants on time, which helps us to obtain a stronger result, concerning continuous differentiability.

\begin{theorem}\label{tw:ciagla_pocho}
Let $\nu_\bullet^h:t\mapsto\nu_t^h$ be the weak solution to problem~\eqref{def:prob_init2_lin} with coefficients~$b^h$ and~$w^h$ defined by~\eqref{def:pert}. Moreover, let coefficients $b^h$, $w^h$ satisfy assumptions of Theorem~\ref{th:main}. Then 
\[ \Big(h \mapsto \nu_\bullet^{h}\Big) \in \CC^1\left((-\tfrac{1}{2}, \tfrac{1}{2});\;\CC_{\widehat{\omega}}([0,\infty);\ZZZ)\right),\]
where $\widehat{\omega}(t)$ is of order $\mathcal{O}(\tfrac{1}{t}\exp(-gt))$  for some constant $g>0$ independent of $h$.
\end{theorem} 
The constant $g$ depends on the size of the interval $(-\xi,\xi)$ in which $h$ is considered.
The proof is presented in Appendix~\ref{apdx:ciagla_pocho} on page~\pageref{apdx:ciagla_pocho}.
Note that Theorem~\ref{tw:ciagla_pocho} implies Theorem~\ref{th:main} for fixed~$t$.
Explicit weight $\widehat{\omega}(t)$ is given in the Appendix by formula~\eqref{final_weight_oljee}.

Let us investigate a different perturbation than~\eqref{def:pert}: 
\begin{equation}\label{eq:non_linear_b}
\begin{split}
\overline{b}^h(t, x):=&b_0(t,x)+h\: b_1(t,x)+C(t)\mathcal{O}(|h|),
\\
\overline{w}^h(t, x):=&w_0(t,x)+h\: w_1(t,x)+C(t)\mathcal{O}(|h|),
\end{split}
\end{equation}
where $C(t)$ is a constant dependent on time. As before coefficients $b_i:[0,\infty)\times\RRD\to\RRD$ and $w_i:[0,\infty)\times\RRD\to\RRD$ for $i=0,1$ and parameter
$h\in(-\frac{1}{2},\frac{1}{2})$. The perturbed problem  has the form
\begin{equation}\label{def:prob_init3}
\left\{ \begin{array}{l}
\partial_t\nu_t^h + \textrm{div}_x\left( \overline{b}^h( t, x)\;\nu_t^h \right) = \overline{w}^h(t,x)\;\nu_t^h \quad \textrm{in }(\CC_c^1([0,+\infty) \times \RRD))^{\ast}, \\
\nu_{t=0}^h= \nu_0 \in \mathcal{M}(\RRD). 
\end{array} \right. 
\end{equation}
The following extension of Theorem~\ref{tw:ciagla_pocho} remains true.
\begin{theorem}\label{lem:main_2}
Let the mapping $\nu_\bullet^h:t\mapsto\nu_t^h$ be the weak solution to problem~\eqref{def:prob_init3} with coefficients defined by~\eqref{eq:non_linear_b}.
Assume that $\left(t \mapsto b_i(t, \cdot) \right)\in \CC_b\big([0, +\infty); \CCA(\RRD;\RRD)\big)$ and $\left(t \mapsto w_i(t, \cdot) \right)\in \CC_b\big([0, +\infty); \CCA(\RRD;\RR)\big)$ for $i=0,1$ and let $C(t)=\mathcal{O}(te^{gt})$.
 Then the mapping
$(-\frac{1}{2}, \frac{1}{2}) \ni h \mapsto \nu_t^{h} \in \mathcal{M}(\RRD)$
 is differentiable in~$\ZZZ$. Moreover,
 \[ \Big(h \mapsto \nu_\bullet^{h}\Big) \in \CC^1\big((-\tfrac{1}{2}, \tfrac{1}{2});\;\CC_{\omega_2}([0,\infty);\ZZZ)\big),\]
where weight $\omega_2(t)=\mathcal{O}\big(\tfrac{1}{t^2}\exp(-gt)\big)$ and constant $g>0$.
\end{theorem}
The proof  follows exactly the same lines of reasoning as the proof of Theorem~\ref{tw:ciagla_pocho}, and hence is omitted. Theorem~\ref{lem:main_2} will be used in the final part of the proof of the Theorem~\ref{main_hille} (see page~\pageref{eq:final}).

\section{Properties of the superposition operators}\label{sec:nemy_0}
The density dependent velocity field $v(\mu_t)$ and production rate $m(\mu_t)$ are formed from a combination of a kernel operator $\mu_t\mapsto k_{\mu_t}$ (equation~\eqref{eq:kernel_k}) and a superposition operator: $v(\mu_t)=v\circ k_{\mu_t}$. In the literature, few results can be found on such operators in H\"older spaces, even less in relation to measures. This section provides the required results. 
In particular, in the proof of Theorem~\ref{main_hille} we will approximate the weak solution to the non-linear equation by solutions to the linear equation. We will then use properties of the linear equation; see Theorem~\ref{tw:ciagla_pocho} and Theorem~\ref{lem:main_2}. Recall that by $\mu_\bullet$ we mean the mapping $t\mapsto\mu_t$. This section is one of steps in arguing that coefficients $v(k_{\mu_\bullet})$ and $m(k_{\mu_\bullet})$ of equation~\eqref{def:prob_init2} satisfy assumptions of Theorem~\ref{tw:ciagla_pocho}.

First, we will show that the mapping \[\mu\mapsto u\!\left(\int_\RRD K_u(y,x)\:\dd\mu(y)\right)\quad \textrm{is of a class} \quad \CC^1\left(\ZZZ;\; \CCA(\RRD;\RRD)\right),\] for $u\in\{v_0,v_1\}$ and analogously for $u\in\{m_0,m_1\}$.
Then we will prove that \[\mu_\bullet \mapsto\! \left(t\mapsto u\!\left(\int_\RRD K_u(y,x)\:\dd\mu_t(y)\right)(\cdot)\right)\]
is of class $\CC^1\left(\CC_\omega([0,\infty);\ZZZ);\: \CC_\omega([0,\infty); \CCA(\RRD;\RRD))\right)$ for $u\in\{v_0,v_1\}$ and analogously for $u\in\{m_0,m_1\}$. Since we fix $u$, we simply write $K$ instead of $K_u$.

The function $k_{\mu}(x):=\int_{\RRD} K(y, x)\:\dd \mu(y)$ is well-defined for any measure $\mu\in \mathcal{M}(\RRD)$ as an~integral. However, $\int_\RRD K(y,x)\:\dd\mu(y)$ is not well-defined for $\mu\in\ZZZ\setminus\MM(\RRD)$. For fixed $x$, the kernel $K(\cdot,x)\in\CCA(\RRD)$ by assumption {\bf A3} of Theorem~\ref{main_hille}. Hence, $\langle K(\cdot,x),\phi\rangle_{\CCA(\RRD),\:\ZZZ}$ is defined for every $x\in\RRD$ as the value of $\phi\in\ZZZ\subset(\CCA(\RRD))^*$ on $K(\cdot,x)\in\CCA(\RRD)$ .

We will always write integral instead of $\langle \cdot,\cdot \rangle$ also when $\mu\in\ZZZ\setminus \MM(\RRD)$, realizing the slight abuse of notation. One of the reasons why we do not distinguish notation is that in fact solution to~\eqref{def:prob_init2} is a~measure (see  Corollary~\ref{cry:measure}).

\begin{lemma}\label{lem:estim_x_K}
Let $\left( x \mapsto K(\cdot, x) \right)\in \CCA\left(\RRD; \CCA(\RRD)\right)$. Then for any $\mu \in  \ZZZ$ the mapping $\mu \mapsto k_\mu: \ZZZ \to \CCA(\RRD)$ is well-defined and linear. Moreover, this mapping is continuous. 
\end{lemma}

\begin{proof}
Notice that $k_\mu$ is differentiable, since $\nabla_x k_\mu=\int_\RRD \nabla_x K(y,\cdot)\:\dd\mu(y)$.
The linearity of the mapping $\mu \mapsto k_\mu$ is obvious, thus we concentrate on showing boundedness as  a~linear map $\ZZZ\to \CCA(\RRD)$. 
By definition
 \[\Vert k_\mu(\cdot) \Vert_{\CCA(\RRD)} = \sup_{x \in \RRD} |k_\mu(x)| +\sup_{x\in \RRD}|\nabla_x k_\mu(x)| + |\nabla_x k_\mu(x)|_\alpha.\]
We will estimate each term separately
\begin{align}\label{estim:bound1}
 \sup_{x \in \RRD} |k_\mu(x)| &= \sup_{x\in \RRD} \left|\int_{\RRD} K(y, x) \:\dd\mu(y)\right| \leq \sup_{x\in \RRD}\Vert K(\cdot, x)\Vert_{\CCA(\RRD)} \Vert \mu \Vert_\ZZZ,\\
\sup_{x\in \RRD}|\nabla_x k_\mu(x)| & 
= \sup_{x\in \RRD}\left| \int_{\RRD} \nabla_x K(y, x) \:\dd\mu(y)\right| 
\leq \sup_{x \in \RRD} \Vert \nabla_x K(\cdot,x)\Vert_{\CCA(\RRD)}  \Vert \mu \Vert_\ZZZ.\label{estim:bound11}
\end{align}
We need to argue that~$\sup_{x\in\RRD}\Vert \nabla_x K(\cdot,x)\Vert_{\CCA(\RRD)}$ is bounded. By assumption
$\left(x \mapsto\! K(\cdot, x)\right)\! \in \CCA(\RRD; \CCA(\RRD))$ one has that $\nabla_x K\in\CC^{0+\alpha}(\RRD)$, hence in particular $\nabla_x K$  is bounded. One has also, that for any $x$ value $\|\nabla_x K(\cdot,x)\Vert_{\CCA(\RRD)}$ (the norm $\|\cdot\|_{\CCA}$ concerns $y$) is bounded.
And next 
\begin{equation}\label{estim:bound3}
\begin{split}
 \left|\nabla_xk_\mu(x)\right|_\alpha &
=\sup_{\substack{x_1\neq x_2\\ x_1, x_2 \in \RRD }}\frac{\left|\int_{\RRD} \left( \nabla_x K(y, x_1) - \nabla_x K(y, x_2)\right) \:\dd\mu(y)\right|}{|x_1 - x_2|^{\alpha}}\\\
& \leq \sup_{\substack{x_1 \neq x_2\\ x_1, x_2 \in \RRD}}\left(\frac{\Vert  \nabla_x \left(K(\cdot,x_1) - K(\cdot, x_2)\right)\Vert_{\CCA(\RRD)}}{|x_1 - x_2|^{\alpha}}\right)  \Vert \mu \Vert_\ZZZ.
\end{split}
\end{equation}
The last inequality holds by assumption that $(x\mapsto K(\cdot, x))\in \CCA(\RRD; \CCA(\RRD))$, which implies that $\nabla_x K(\cdot, x)\in \CC^{0+\alpha}(\RRD)$.
Using~\eqref{estim:bound1}-\eqref{estim:bound3} one obtains that $k_\mu\in\CCA(\RRD)$ and 
\begin{equation}\label{eq:k_mu_in_CCA}
\|k_\mu\|_{\CCA(\RRD)} \le \|\mu\|_\ZZZ \|K\|_{\CCA( \RRD;\:\CCA(\RRD))}.
\end{equation}
Therefore, the mapping $\mu \mapsto k_\mu$ is continuous linear transformation of $\ZZZ$ into $\CCA(\RRD)$.
\end{proof}
\begin{lemma}\label{lem:1_25}
If $v\in\CCA(\RR)$ and $f,g\in\CCA(\RRD)$, then  $v\circ f$ and $v\circ g$ are in $\CCA(\RRD)$ and moreover
\[|v\circ f-v\circ g|_\alpha\leq \max\Big\{2\|v'\|_\infty \|f-g\|_\infty,\;
\|v'\|_\infty\|\nabla f-\nabla g\|_\infty+|v'|_\alpha \|f-g\|_\infty^\alpha\|\nabla g\|_\infty
\Big\}.
\]
\end{lemma}
\begin{proof}
Compositions $v\circ f$ and $v\circ g$ are $\CCA(\RRD)$ according to~\eqref{eqn:alpha_infty} and~\eqref{eqn:alpha}. Knowing that $\frac{\partial}{\partial x_j}(v\circ f)=(v'\circ f)\frac{\partial f}{\partial x_j}$ and making use of estimations \eqref{eqn:alpha_infty} and~\eqref{eqn:alpha} one obtains
\[
\left\|\frac{\partial}{\partial x_j}(v\circ f)\right\|_\infty \leq \|v'\|_\infty \left\|\frac{\partial f}{\partial x_j}\right\|_\infty<\infty\]
and also
\begin{align*}
\left|\frac{\partial}{\partial x_j}(v\circ f)\right|_\alpha &\leq
|v'\circ f|_\alpha \left\|\frac{\partial f}{\partial x_j}\right\|_\infty
+
\|v'\circ f\|_\infty\left|\frac{\partial f}{\partial x_j}\right|_\alpha\\
&\leq |v'|_\alpha \|\nabla f\|_\infty^\alpha \left\|\frac{\partial f}{\partial x_j}\right\|_\infty +\|v'\|_\infty \left|\frac{\partial f}{\partial x_j}\right|_\alpha<\infty.
\end{align*}
Note that for any $\psi\in\CC^1(\RRD)$ the following holds $|\psi|_\alpha\leq \max \{2\|\psi\|_\infty,\|\nabla\psi\|_\infty\}$ and hence
\begin{align*}
&|v\circ f-v\circ g|_\alpha\leq \max \big\{2\|v\circ f-v\circ g\|_\infty,\quad \|\nabla(v\circ f)-\nabla(v\circ g)\|_\infty\big\}\\
&\leq \max\big\{2\|v'\|_\infty\|f-g\|_\infty, \quad \|(v'\circ f)\nabla f-(v'\circ g)\nabla g\|_\infty\big\}\\
&\leq \max\big\{2\|v'\|_\infty\|f-g\|_\infty, \quad 
\|v'\circ f\|_\infty \|\nabla f-\nabla g\|_\infty +\|v'\circ f-v'\circ g\|_\infty \|\nabla g\|_\infty\big\}\\
&\leq \max\big\{2\|v'\|_\infty\|f-g\|_\infty, \quad 
\|v'\|_\infty \|\nabla f-\nabla g\|_\infty +|v'|_\alpha\|f-g\|_\infty^\alpha \|\nabla g\|_\infty\big\}.
\end{align*}
\end{proof}
\begin{lemma}\label{lem:lem_2_2502}
If $v\in\CC^{2+\alpha}(\RR;\RRD)$ then the mapping $f\mapsto v\circ f:\CCA(\RRD)\to\CCA(\RRD;\RRD)$ is continuous.
\end{lemma}
\begin{proof}
Let $f,g\in\CCA(\RRD)$. Then, with $v(s)=(v_1(s),\dots, v_d(s))$ one obtains
\begin{align*}
&\|v_i\circ f-v_i\circ g\|_\infty\leq \|v_i'\|_\infty\|f-g\|_\infty\\
&\textrm{and}\\
&\left\|\frac{\partial}{\partial  x_j}(v_i\circ f)-\frac{\partial}{\partial x_j}(v_i\circ g)\right\|_\infty
	=\left\|\left(\frac{\partial v_i}{\partial x_j}\circ f\right)\frac{\partial f}{\partial x_j}-\left(\frac{\partial v_i}{\partial x_j}\circ g\right)\frac{\partial g}{\partial x_j}\right\|\infty\\
	&\qquad \leq
	\left\|\frac{\partial v_i}{\partial x_j}\circ f-\frac{\partial v_i}{\partial x_j}\circ g\right\|_\infty \left\|\frac{\partial f}{\partial x_j}\right\|_\infty
	+\left\|\frac{\partial v_i}{\partial x_j}\circ g\right\|_\infty
	\left\|\frac{\partial f}{\partial x_j}-\frac{\partial g}{\partial x_j}\right\|_\infty\\
&\textrm{and also}\\
&\left|\frac{\partial }{\partial x_j}(v_j\circ f)-\frac{\partial }{\partial x_j}(v_i\circ g)\right|_\alpha \!\!\leq 
\underbrace{\left|\left(\frac{\partial v_i}{\partial x_j}\circ f-\frac{\partial v_i}{\partial x_j}\circ g\right)\frac{\partial f}{\partial x_j}\right|_\alpha
\!\!+
\left|\left(\frac{\partial v_i}{\partial x_j}\circ g\right)\!\left(\frac{\partial f}{\partial x_j}-\frac{\partial g}{\partial x_j}\right)\right|_\alpha}_{\textrm{according to~\eqref{eqn:alpha_infty} and~\eqref{eqn:alpha}}}\\
	& \qquad \leq \underbrace{\left|\frac{\partial v_i}{\partial x_j}\circ f -\frac{\partial v_i}{\partial x_j}\circ g \right|_\alpha}_{\textrm{according to Lemma~\ref{lem:1_25}}} \left\|\frac{\partial f}{\partial x_j}\right\|_\infty +\left\|\frac{\partial v_i}{\partial x_j}\circ f -\frac{\partial v_i}{\partial x_j}\circ g\right\|_\infty \left|\frac{\partial f}{\partial x_j}\right|_\alpha\\
	&\qquad\qquad +\left|\frac{\partial v_i}{\partial x_j}\circ g\right|_\alpha\left\|\frac{\partial f}{\partial x_j}-\frac{\partial g}{\partial x_j}\right\|_\infty + \left\|\frac{\partial v_i}{\partial x_j}\right\|_\infty \left|\frac{\partial f}{\partial x_j}-\frac{\partial g}{\partial x_j}\right|_\alpha\\
		& \qquad \leq \max \Big\{2\|v_i''\|_\infty \|f-g\|_\infty,\quad\|v_i''\|_\infty\|\nabla f-\nabla g\|_\infty+|v''|_\alpha\|f-g\|_\infty^\alpha\|\nabla g\|_\infty\Big\}\left\|\frac{\partial f}{\partial x_j}\right\|_\infty \\
		&\qquad \qquad 
		+\|v''\|_\infty \|f-g\|_\infty\left|\frac{\partial f}{\partial x_j}\right|_\alpha +\big(\left|v'\right|_\alpha\|g\|_\infty+\|v'\|_\infty|g|_\alpha\big)\left\|\frac{\partial f}{\partial x_j}-\frac{\partial g}{\partial x_j}\right\|_\infty\\
	&\qquad\qquad  + \left\|v'\right\|_\infty \left|\frac{\partial f}{\partial x_j}-\frac{\partial g}{\partial x_j}\right|_\alpha.
\end{align*}
These estimates yields the continuity of $f\mapsto v\circ f$ on $\CCA(\RRD).$
\end{proof}

\begin{corollary}\label{cry:1_25}
If $v\in\CC^{2+\alpha}(\RR;\RRD)$ and $S\subset\CCA(\RRD)$ is a separable subset, then $ v(S)\subset \CCA(\RRD;\RRD)$ is a separable subset.
\end{corollary}
We use this to observe the following:
\begin{corollary}\label{cry:2_25}
Let $v\in\CC^{2+\alpha}(\RR;\RRD)$, then $\{v(k_\mu(\cdot)):\mu\in\ZZZ\}$ is a~separable subset of $\CCA(\RRD;\RRD)$.
\end{corollary}
\begin{proof}
The mapping $\mu\mapsto k_\mu:\ZZZ\to\CCA(\RRD)$ is continuous (see Lemma~\ref{lem:estim_x_K}). Hence, $S=\{k_\mu:\mu\in\ZZZ\}$ is separable, because $\ZZZ$ is separable  (according to Proposition~\ref{prop:separable}) Corollary~\ref{cry:1_25} implies that $\{v(k_\mu):\mu\in\ZZZ\}$ is separable.
\end{proof}
\begin{lemma}\label{lem:nem_char} Under assumptions {\bf A1}, {\bf A2} and {\bf A3} of Theorem~\ref{main_hille} the mapping $\mu \mapsto u(k_\mu)(\cdot)$ is of class $\CC^1\left(\ZZZ; \CCA(\RRD;\RRD)\right)$ for $u\in\{v_0,v_1\}$ and of class $\CC^1\left(\ZZZ; \CCA(\RRD;\RR)\right)$ for $u\in\{m_0,m_1\}$.
Moreover, there exist constants $C_{u,K_u},C_{u,K_u}'>0$ such that 
\begin{equation}\label{eq:estim_coeff}
\|u(k_\mu)\|_{\CCA(\RRD)}\leq C_{u,K_u}(1+\|\mu\|_\ZZZ+\|\mu\|_\ZZZ^2)
\end{equation}
\begin{equation*}
\textrm{and} \qquad
\|\partial u(k_\bullet)(\mu)\|_{\mathcal{L}(\ZZZ;\CCA(\RRD))}\leq C_{u,K_u}'(1+\|\mu\|_\ZZZ+\|\mu\|_\ZZZ^{1+\alpha}).
\end{equation*}
\end{lemma}
\begin{remark}
 Proof of the above lemma requires $u\in\CC^{3+\alpha}(\RRD)$ (assumption {\bf A1} and {\bf A2} of Theorem~\ref{main_hille}), and this is the only place in this paper where such a strong assumption on coefficients is needed. The differentiability of this map is necessary further for the differentiability in parameter.
Notice that when $v$ and $m$ are $C^{2+\alpha}$ then Lemma~\ref{lem:lem_2_2502} holds.
\end{remark}
Proof of Lemma~\ref{lem:nem_char} can be found in the Appendix~\ref{apx:lem36}. Now we present a stronger statement.
\begin{lemma}\label{tw:B_C1}
For any weight $\omega(t)>0$,
under the  assumptions {\bf A1}, {\bf A2}, {\bf A3} of Theorem~\ref{main_hille}, the mapping $\mu_\bullet \mapsto u(k_{\mu_\bullet})(\cdot)$ is of class  \[\CC^1\left(\CC_\omega([0,\infty);\ZZZ);\: \CC_\omega([0,\infty); \CCA(\RRD;\RRD))\right)\] for $u\in\{v_0,v_1\}$ and similarly for $u\in\{m_0,m_1\}$.
\end{lemma}
\begin{proof}[Proof of Lemma~\ref{tw:B_C1}]
We need to show that the following norm \begin{equation*}\label{eqn:estim_BC1}
\begin{split}
& \sup_{t\geq 0}\omega(t)\Big\|u(k_{\mu_t})-u(k_{\overline{\mu}_t})- \partial u(k_{\overline{\mu}_t})(\mu_t-\overline{\mu}_t)\Big\|_{\CCA(\RRD)}
\end{split}
\end{equation*}
can be estimated from above by $\|\mu_\bullet-\overline{\mu}_\bullet\|_{\ZZZ_{\omega}}$. We will directly use estimations from proof of Lemma~\ref{lem:nem_char}. Since constants in~\eqref{eq:estim_coeff} do not depend on time, thus immediately one obtains
\begin{align*}
& \sup_{t\geq 0}\omega(t)\Big\|u(k_{\mu_t})-u(k_{\overline{\mu}_t})- \partial u(k_{\overline{\mu}_t})(\mu_t-\overline{\mu}_t)\Big\|_{\CCA(\RRD)} \leq\\
&\leq\sup_{t\geq 0}\omega(t)\left\{
\Big\| u(k_{\mu_t})-u(k_{\overline{\mu}_t})- \partial u(k_{\overline{\mu}_t})(\mu_t-\overline{\mu}_t)\Big\|_{\infty}\right\}\\
&\qquad+\sup_{t\geq 0}\omega(t)\left\{
\Big\| \nabla_x \Big(u(k_{\mu_t})-u(k_{\overline{\mu}_t})- \partial u(k_{\overline{\mu}_t})(\mu_t-\overline{\mu}_t)  \Big)\Big\|_{\infty}\right\}\\
& \qquad+\sup_{t\geq 0}\omega(t)\left\{
\left| \nabla_x \Big( u(k_{\mu_t}) - u(k_{\overline{\mu}_t})- \partial u(k_{\overline{\mu}_t})(\mu_t-\overline{\mu}_t)\Big)\right|_\alpha\right\}\\
&\leq
c_1\|\mu_\bullet-\overline{\mu}_\bullet\|_{\ZZZ_\omega} + c_2 \|\mu_\bullet-\overline{\mu}_\bullet\|_{\ZZZ_\omega}^{1+\alpha}+ c_3 \|\mu_\bullet-\overline{\mu}_\bullet\|_{\ZZZ_\omega}^{2}.
\end{align*}
This finishes the proof.
\end{proof}
Since the parameter $h$ is considered in a bounded interval and since $v^h=v_0+h\: v_1$ and analogously $m^h=m_0+h\: m_1$ (see~\eqref{pert_non}), the regularity of perturbed coefficients $v^h$ and $m^h$ can be easily stated.
\begin{corollary}
Under assumption {\bf A1}, {\bf A2}, {\bf A3} of Theorem~\ref{main_hille} for any fixed value of parameter~$h$ the mapping $\mu_\bullet \mapsto u^h(k_{\mu_\bullet})(\cdot)$ is of class  \[\CC^1\left(\CC_\omega([0,\infty);\ZZZ);\: \CC_\omega([0,\infty); \CCA(\RRD;\RRD))\right)\] for $u\in\{v_0,v_1\}$ and analogously for $u\in\{m_0,m_1\}$. 
\end{corollary}
Notice that in the above lemma nothing was assumed about weight $\omega$.
We now first compute the Fr\'echet derivative of $\mu\mapsto u(k_\mu)$ at $\mu_0$, assuming that it exists. If it does, then it equals the directional derivative (Gateaux derivative):
\[\partial [u(k_\bullet)](\mu)\overline{\mu}=D_{\overline{\mu}}u(k_\mu)=\lim_{h\to 0}\frac{u(k_\mu+hk_{\overline{\mu}})-u(k_\mu)}{h},\]
where the limit is in $\CCA(\RRD)$. If it exists, then the limit must also exists pointwise at $x\in\RRD$.

In the following, we will omit the evaluation at $x$ in $D_{\overline{\mu}}u(k_\mu),k_\mu$ and $k_{\overline{\mu}}$.
Then 
\begin{equation}\label{eq:deriv_direct_N}
\begin{split}
D_{\overline{\mu}} u(k_\mu)& = \lim_{h \to 0} \frac{u\circ(k_\mu+h\: k_{\overline{\mu}}) - u\circ k_\mu}{h}\\
& \mbox{\tiny{(using Taylor expansion)}}\\
 &=\lim_{h \to 0} \frac{u\circ k_\mu + h\: (u'\circ k_\mu)\: k_{\overline{\mu}} + R_1[u,k_\mu](k_\mu+h\: k_{\overline{\mu}})
 - u\circ k_\mu}{h} \\
  &=(u'\circ k_\mu)\: k_{\overline{\mu}} +\lim_{h \to 0} \frac{  R_1[u,k_\mu](k_\mu+h\: k_{\overline{\mu}})}{h} =
(u'\circ k_\mu) \: k_{\overline{\mu}},
\end{split}
\end{equation}
where $R_{n+1}[f,x_0](x)$ is the remainder term when expanding a function  $x\mapsto f(x)$ in Taylor series at~$x_0$ up to order $n$. 
The vanishing of the limit with the remainder term $R_1[u,k_\mu]$ follows from the following proposition.
\begin{proposition}\label{prop:remainder estimate Holder}
Let $\Omega\subset\RRD$ be an open domain and let $f\in \CC^{n+\alpha}(\Omega)$, with $n\in\NN_0, 0<\alpha\leq 1$. Then for any $x,x_0\in\Omega$
\begin{equation*}\label{eq:remainder estimate Holder}
\bigl| R_n[f;x_0](x) \bigr| \leq C_{n,\alpha} \max_{\beta\in\NN_0^d,\ |\beta|=n} |D^\beta f|_\alpha \: \|x-x_0\|^{n+\alpha},
\end{equation*}
where $C_{n,\alpha} := \frac{d^n}{\alpha(\alpha+1)\cdot\,\ldots\,\cdot(\alpha+n-1)}$.
\end{proposition}
The $D^\beta$ is a multi-index notation for vector $\beta$.
Hence, \[\big|R_1[u,k_\mu](k_\mu+h\:k_{\overline{\mu}})(x)\big|\leq C_{1,\alpha}|u'|_\alpha\:|h|^{1+\alpha}\:\|k_{\overline{\mu}}\|_\infty^{1+\alpha}.\]
%
%
%
A weak solution $\mu_\bullet$ to the non-linear problem~\eqref{def:prob_init} (or~\eqref{def:prob_init2}) is a weak solution to the linear problem~\eqref{eq:def_linear}, where the velocity field $b(t,x)=v_0(k_{\mu_t}(x))$ and production rate $w(t,x)=m_0(k_{\mu_t}(x))$. In order to be able to apply the crucial representation results for the solution to the linear equation (Lemma~\ref{lem:repr_form}) we need to verify that $b$ and $w$ defined by $t\mapsto \mu_t$ satisfy the assumptions of this lemma. This is assured by the following. 
\begin{lemma}\label{lem:3_26}
Let $v_0\in\CC^{2+\alpha}(\RR;\RRD)$, $m_0\in\CC^{2+\alpha}(\RR)$. Assume moreover, that $x\mapsto K_{u}(\cdot, x)$ is in $\CCA(\RRD;\CCA(\RRD))$ for $u\in\{v_0, m_0\}$ and that for any $\phi\in\big(\CCA(\RRD)\big)^*$ the mapping 
\[y\mapsto \langle K_u(y,\cdot),\phi\rangle \qquad \textrm{ is continuous and bounded.}\]
Then for any narrowly continuous map $\nu_\bullet:\RR_+\to \mathcal{M}(\RRD)$, the following holds:
\begin{enumerate}[label=(\roman*)]
\item $w(t,x):=m_0(k_{\nu_t}(x))$ (see notation~\eqref{pert_non}) is separately continuous, bounded and for each $t\geq 0$, $w(t,\cdot)\in\CCA(\RRD)$.
\item $b(t,x):=v_0(k_{\nu_t}(x))$ is such that $t\mapsto b(t,\cdot):\RR_+\!\to\CCA(\RRD;\RRD)$ is weakly continuous. It is Bochner measurable and essentially bounded $t\mapsto b(t,\cdot)\in\textrm{L}^\infty(\RR_+;\CCA(\RRD;\RRD))$. 
\end{enumerate}
\end{lemma}
\begin{proof}
For fixed $x\in\RRD$, the mapping $y\mapsto K_u(y,x)$ is continuous and bounded, hence $t\mapsto k_{\nu_t}(x)=\int_\RRD K_u(y,x)\dd\nu_t(y)$ is also continuous and bounded (by assumption that $t\mapsto \nu_t$ is narrowly continuous).

For fixed $t$, the mapping $x\mapsto k_{\nu_t}(x)\in\CCA(\RRD)$. Hence, $x\mapsto w(t,x)$ is $\CCA(\RRD)$, in particular it is Lipschitz. Coefficient $w$ is bounded because $m_0$ is bounded. Thus, $w$ satisfies the conditions of Lemma~\ref{lem:repr_form}. Next, $b(t,\cdot)\in\CCA(\RRD;\RRD)\subset W^{1,\infty}(\RRD;\RRD)$. Moreover, one has
\[\Big\{b(t,\cdot):t\in[0,T]\Big\}\subset\Big\{v_0(k_\mu):\mu\in\mathcal{M}(\RRD)\Big\}.\]
The latter set is a separable subset of $\CCA(\RRD;\RRD)$ according to Corollary~\ref{cry:2_25}. Furthermore, for any $\phi\in(\CCA(\RRD;\RRD))^*$, which can be indicated with $(\phi_1,\dots,\phi_d)\in\left[\CCA(\RRD)^*\right]^d$, we have
\begin{align*}
\langle v_0(k_{\nu_t}),\phi\rangle&=\sum_{i=1}^d \langle v_{0i}(k_{\nu_t}),\phi_i\rangle=\sum_{i-1}^d\langle k_{\nu_t},\phi_i\circ v_{0i}\rangle.
\end{align*}
Each $\psi_i=\phi_i\circ v_{0i}\in\CCA(\RRD)^*$. So,
\[t\mapsto \langle v_0(k_{\nu_t}),\phi\rangle=\sum_{i=1}^d\int_\RRD\langle K_{v_0}(y,\cdot),\psi_i\rangle\:\dd\nu_t(y).\]
The functions $y\mapsto \langle K_{v_0}(y,\cdot),\psi_i\rangle$ are continuous and bounded by assumption $t\mapsto \nu_t$ is narrowly continuous, so $t\mapsto \langle v_0(k_{\nu_t}),\phi\rangle$ is continuous for every $\phi$, hence Borel measurable.

Pettis Measurability Theorem implies that $t\mapsto b(t,\cdot)=v_0(k_{\nu_t})$ is Bochner measurable. As it is bounded,
\[t\mapsto b(t,\cdot)\in L^\infty([0,T];W^{1,\infty}(\RRD;\RRD))\subset L^1([0,T];W^{1,\infty}(\RRD;\RRD)).\]
So, $b$ and $w$ satisfy the conditions of Lemma~\ref{lem:repr_form}.
\end{proof}
\begin{corollary}\label{cry:311}
Under conditions of Lemma~\ref{lem:3_26}, the coefficients $w(t,\cdot)=m_0(k_{\nu_t})$ and $b(t,\cdot)=v_0(k_{\nu_t})$ satisfy the conditions of Lemma~\ref{lem:repr_form}, when restricted to $t\in[0,T], T>0$.
\end{corollary}

\begin{proof}
What is left to prove is that $w$ is (jointly) Borel measurable. This follows from separate continuous in $t$ and $x$, see e.g.~\cite[Lemma 6.4.6, page 16]{bogachev2007measure}.
\end{proof}
\section{Existence of weak solution to non-perturbed non-linear equation}\label{sec:well_00}
In this section we argue that the {\bf non-perturbed}, non-linear problem~\eqref{def:prob_init} has a weak solution. This solution will be denoted by~$\mu_\bullet^0$, to underline that $h=0$. An approximation scheme will be introduced -- let us denote by $\nu_\bullet^{0,n}$ an approximate solution in $n$-th step of approximation.
To obtain the existence of a unique solution we need to show that:
\begin{description}
\item[A] The approximation scheme is convergent in $\ZZZ_{\tilde{\omega}}$ i.e. $\lim_{n\to\infty}\nu_\bullet^{0,n}=:\overline{\nu}_\bullet^0$ exists for every $t\in[0,\infty)$.
\item[B] The limit $\overline{\nu}_t^0$ in $\ZZZ$ is in fact a measure -- convergence is shown in $\ZZZ$ and the space $\MM(\RRD)\subset \ZZZ$ is not complete with the norm $\|\cdot\|_{\ZZZ}$. Hence, an additional argument is needed to establish that $\overline{\nu}_t^0\in\MM(\RRD)$;
\item[C] The mapping $t\mapsto\overline{\nu}_t^0$ is narrowly continuous,
so $\overline{\nu}_\bullet^0$ is a weak solution to~\eqref{def:prob_init}.
\item[D] There is no other solution to~\eqref{main_hille} in the sense of Definition~\ref{def:weak_solu_22}, when $v_0\in\CC^{2+\alpha}(\RRD;\RRD)$ and $m_0\in\CC^{2+\alpha}(\RRD)$, in the space $C_\omega([0,\infty);\ZZZ)$.
\end{description}  
In this section we concentrate on showing {\bf A}. 

Using the notation for the superposition operator (introduced in the previous section), non-linear and non-perturbed transport equation can be written as
 \begin{equation}\label{eqn:pertur0_00}
\left\{ \begin{array}{l}
\partial_t \mu_t^0 +\div_x \Big(v_0(k_{\mu_t^0})\; \mu_t^0\Big)= m_0(k_{\mu_t^0})\;\mu_t^0,\\
\mu_{t=0}^0=\mu_0\in\MM(\RRD).
\end{array} \right.
\end{equation} 
 We claim that a solution $t\mapsto\mu_t^0$ can be obtained as a limit of a sequence $\lim_{n\to\infty}\nu_\bullet^{0,n}$ defined by the following scheme. Recall Lemma~\ref{lem:repr_form} and Corollary~\ref{coro:conti_time}.
\begin{description}
 \item[Step 0] A unique continuous mapping $\nu_\bullet^{0,0}:[0,\infty)\to\ZZZ$ solves the equation
\begin{equation}\label{egn:non-pertur0_0000}
\left\{ \begin{array}{l}
\partial_t \nu_t^{0,0} +\div_x \Big(v_0(k_{\mu_0})\; \nu_t^{0,0}\Big)=m_0(k_{\mu_0})\;\nu_t^{0,0},\\
\nu_0^{0,0}=\mu_0,
\end{array} \right. 
\end{equation} 
and $\nu_t^{0,0}\in\mathcal{M}(\RRD)$ and $t\mapsto \nu_t^{0,0}$ is narrowly continuous.
   \item[Step 1] For $n \geq 1$, a unique continuous mapping $\nu_\bullet^{0,n}:[0,\infty)\to\ZZZ$ solves the equation
 \begin{small}
\begin{equation}\label{egn:non-pertur0_01}
\left\{ \begin{array}{l}
\partial_t \nu_t^{0,n} +\div_x \Big(v_0(k_{\nu_t^{0,n-1}})\; \nu_t^{0,n}\Big)= m_0(k_{\nu_t^{0,n-1}})\nu_t^{0,n},\\
  \nu_0^{0,n}=\mu_0,
\end{array} \right.
\end{equation}
 \end{small}
 and $\nu_t^{0,n}\in\mathcal{M}(\RRD)$ and $t\mapsto \nu_t^{0,0}$ is narrowly continuous.
\end{description}

Note that in every step of approximation $\nu_\bullet^{0,n}$ solves the {\bf linear problem} i.e. the coefficients do not depend on the solution. 
Moreover, notice that if only the mappings $\big(t\mapsto v_0(k_{\nu_t})(\cdot)\big)$, $\big(t\mapsto m_0(k_{\nu_t})(\cdot)\big)$ of $[0,\infty)$ into $\CCA(\RR;\RRD)$ and respectively $\CCA(\RR,\RR)$ are continuous and bounded, then immediately   according to Lemma~\ref{lem:repr_form} (representation formula) one has existence and uniqueness of weak solution $\nu_\bullet$ to the problem 
\begin{equation}\label{eq:T_h_nemy_00}
\left\{ \begin{array}{l}
\partial_ t \nu_t +\textrm{div}_x \Big(v_0(k_{\nu_t})\;\nu_t\Big)= m_0(k_{\nu_t})\;\nu_t,\\
\nu_{t=0}=\nu_0\in\MM(\RRD).
\end{array}\right. 
\end{equation} 
Indeed, according to Lemma~\ref{tw:B_C1} one has that for any $\nu_\bullet\in\CC_\omega([0,\infty);\MM(\RRD))$, coefficients $v_0(k_{\nu_\bullet})$ and $m_0(k_{\nu_\bullet})$ are of the class $\CC_\omega([0,\infty);\CCA(\RR;\RRD))$ and $\CC_\omega([0,\infty);\CCA(\RR;\RR))$ respectively. Hence, \eqref{eq:T_h_nemy_00} has unique weak solution in a class of mappings $\CC_\omega([0,\infty);\MM(\RRD))$.
  
The approximation scheme makes that a sequence $\{\nu_t^{0,n}\}_{n\in\NN}$ can be viewed as realised by repeated application of a suitable operator $\mathcal{T}$ on a subset of $\CC([0,\infty);\ZZZ)$. To make this precise, consider $T>0$ and
\[S:\CC\big([0,\infty);\CCA(\RRD;\RR^{d+1})\big)\to\CC([0,\infty);\ZZZ):\qquad \big(t\mapsto(b(t,\cdot),w(t,\cdot))\big)\mapsto(t\mapsto\widehat{\nu}_t),\]
where $\widehat{\nu}_\bullet$ is the weak solution to the linear problem~\eqref{eq:def_linear} with velocity field $b$ and production rate $w$. According to Corollary~\ref{coro:conti_time}, $t\mapsto\widehat{\nu}_t:[0,\infty)\to\ZZZ$ is continuous. Let us consider also
\[B:\CC([0,T];\ZZZ)\to\CC([0,T];\CCA(\RRD;\RR^{d+1})):\qquad (t\mapsto\nu_t)\mapsto\big(t\mapsto(v_0(k_{\nu_t}),m_0(k_{\nu_t}))\big).\]
{\cred Define $\mathcal{T}:=S\circ B$. The crucial observation concerning existence of weak solution to the non-perturbed linear equation~\eqref{def:prob_init} is the following:
\begin{proposition}\label{prop:san_260220}
Let $v_0\in\CC^{2+\alpha}(\RR;\RRD)$ and $m_0\in\CC^{2+\alpha}(\RR;\RR)$. Assume moreover, that $x\mapsto K_{v_0}(\cdot, x)$ and $x\mapsto K_{m_0}(\cdot,x)$ are in $\CCA(\RRD;\CCA(\RRD))$, while for every $\psi\in(\CCA(\RRD))^*$, the mapping $y\mapsto \langle K_u(y,\cdot),\psi\rangle$ is continuous and bounded. Then a mapping $\nu_\bullet:[0,\infty)\to\MM(\RRD)$ is a~weak solution to the non-linear problem~\eqref{def:prob_init} with initial condition $\mu_0\in\MM(\RRD)$ if and only if $\nu_\bullet\in\CC([0,\infty);\ZZZ)$, $\nu_0=\mu_0$ and $\nu_\bullet$ is a fixed point of $\mathcal{T}\colon\mathcal{T}(\nu_\bullet)=\nu_\bullet$.
\end{proposition}
\begin{proof}
If $\nu_\bullet:[0,\infty)\to\mathcal{M}(\RRD)$ is a weak solution to~\eqref{def:prob_init} with initial condition $\mu_0$, then $\nu_\bullet$ is a weak solution to the linear equation~\eqref{eq:def_linear} with coefficients (see notation~\eqref{pert_non}):
\[b(t,\cdot)=v_0(k_{\nu_t}), \qquad w(t,\cdot)=m_0(k_{\nu_t}).\]

According to Lemma~\ref{lem:3_26} and Corollary~\ref{cry:311}, $b(t,x)$ and $w(t,x)$ satisfy the conditions of Lemma~\ref{lem:repr_form}.
 By uniqueness, $\nu_\bullet$ must be given by expression~\eqref{repr:form}. Corollary~\ref{coro:conti_time} yields, that $\nu_\bullet:[0,\infty)\to\ZZZ$ is continuous, $\nu_0=\mu_0$, $B(b,w)$ and $S(b,w)=\nu_\bullet$. Hence, one has $\nu_\bullet=\mathcal{T}(\nu_\bullet)$.

For the other implication, let $\nu_\bullet\in\CC([0,\infty);\ZZZ)$ be a fixed point for $\mathcal{T}$ with $\nu_0=\mu_0$. By construction, $\nu_\bullet=\mathcal{T}(\nu_\bullet)=S(B(\nu_\bullet))$ then is weak solution to the linear equation with $b(t,\cdot)=v_0(k_{\nu_t})$ and $w(t,\cdot)=m_0(k_{\nu_t})$. Thus it is a weak solution to the non-linear equation~\eqref{def:prob_init}.
\end{proof}

We shall now construct a (unique) fixed point of $\mathcal{T}$, in a suitable subset of $\CC_\omega([0,\infty);\ZZZ)$ of $\CC([0,\infty);\ZZZ)$. 
We cannot exclude existence of fixed points of $\mathcal{T}$ outside $\CC_\omega([0,\infty);\ZZZ)$.
To that end, let us define 
\begin{equation}\label{def:ZZ}
\ZZZ_\omega:=\big(\CC_\omega([0,\infty);\ZZZ),\|\cdot\|_{\ZZZ_\omega}\big)\quad \textrm{and}\quad \ZZZ^T_\omega:=\big(\CC_\omega([0,T];\ZZZ),\|\cdot\|_{\ZZZ_\omega}\big).
\end{equation}
Because of Corollary~\ref{cly:bound_nu}, we find that for any $\nu_\bullet \in\CC([0,T];\ZZZ)$
\begin{equation}\label{eq:T_bound}
\|\mathcal{T}(\nu_\bullet)(t)\|_\ZZZ\leq \exp\big(t\|m_0^+(k_{\nu_t})\|_\infty\big)\|\mu_0\|_{TV}
\leq
\exp(t\|m_0\|_\infty)\|\mu_0\|_{TV}.
\end{equation}
Now let us assume that $\omega(0)\equiv 1$, weight $\omega$ is bounded on any $[0,T]$ and such that for any $t_0\geq 0$, there exists $\sigma>0$ and a closed interval $[t_1,t_2]$ containing $t_0$ with $t_1<t_2$, such that $\omega(t)\geq\sigma$ on $[t_1,t_2]$. Recall the later condition implies that $\ZZZ_\omega^T$ is complete.
Moreover, assume that 
\begin{equation}\label{eq:def_M_T}
\sup_{0\leq t\leq T}\omega(t)\exp(t\|m_0\|_\infty)=:M_T<\infty.
\end{equation}
Define
\begin{equation}\label{set:B}
\mathcal{B}_\omega^{\mu_0,T}:=\big\{f\in\ZZZ_\omega^T:f(0)=\mu_0 \in\MM(\RRD), \|f\|_{\ZZZ_\omega^T}\leq M_T\|\mu_0\|_{TV}\big\}.
\end{equation}
Let us notice that $\mathcal{B}_\omega^{\mu_0,T}$ is a closed bounded subset of $\ZZZ_\omega^T$, hence complete.
It is non-empty as it contains the constant function  $t\mapsto\mu_0$. Because of~\eqref{eq:T_bound}, $\mathcal{T}$ maps $\mathcal{B}_\omega^{\mu_0,T}$ into itself. 
 \begin{proposition}\label{prop:unique_00}
For any finite $T>0$ and any $0\leq c<1$, there exists a constant $g\geq \|m_0\|_\infty>0$, such that the weight $\omega(t)=e^{-g t}$ makes the operator $\mathcal{T}=S \circ B: \mathcal{B}_\omega^{\mu_0,T} \to \mathcal{B}_\omega^{\mu_0,T}$ a contraction, where the Lipschitz constant of contraction of $\mathcal{T}$ is less or equal to $c$. 
\end{proposition}
To prove above lemma we need to first show some properties of operators $S$ and $B$. The proof of Proposition~\ref{prop:unique_00} will be conducted on page~\pageref{proof:prop_unique_00}.
Let us argue that indeed operators $B$ and $S$ act between proper function spaces.

\begin{lemma}\label{cl:B_prop}
Under the assumptions {\bf A1}, {\bf A2} and {\bf A3} of Theorem~\ref{main_hille} for any weight $\omega(t)>0$ that is bounded on $[0,T]$
the operator $B$ maps $\ZZZ_\omega^T$ into $\CC_\omega([0,T];\CCA(\RRD;\RR^{d+1}))$. Moreover, there exists $C>0$ independent of $\omega$ or $T$ such that for every $\mu_\bullet\in\ZZZ_\omega^T$ the following estimate holds
\begin{equation}\label{eq:B_proper_space}
\|B(\mu_\bullet)\|_{\CC_\omega\left([0,T];\:\CCA(\RRD;\RR^{d+1})\right)}\leq C(1+\sup_{0\leq s\leq T}\|\mu_s\|_{\ZZZ})\sup_{0\leq s\leq T}\omega(s).
\end{equation}
\end{lemma}
\begin{proof}
Using Lemma~\ref{lem:nem_char} and estimation~\eqref{eq:estim_coeff} one has
\begin{align*}
&\sup_{0\leq t\leq T}\omega(t)\|B(\mu_t)\|_{\CCA(\RRD;\:\RR^{d+1})}\leq \sup_{0\leq t\leq T}(C_v+C_m)\:\omega(t)(1+\|\mu_t\|_\ZZZ+\|\mu_t\|_\ZZZ^2)\\
&\qquad\leq \sup_{0\leq t\leq T}(C_v+C_m)\:\omega(t)\big(1+\|\mu_t\|_\ZZZ+\|\mu_t\|_\ZZZ\sup_{0\leq s\leq T}\|\mu_s\|_\ZZZ\big)\\
&\qquad\leq \sup_{0\leq t\leq T}(C_v+C_m)(1+\sup_{0\leq s\leq T}\|\mu_s\|_\ZZZ)\sup_{0\leq s\leq T}\omega(s)(1+\|\mu_t\|_\ZZZ).
\end{align*}
\end{proof}
 Let us  characterize the operator $S$. Observe that for a given~$\nu_\bullet^{0,n}$ the operators $v_0(k_{\nu_\bullet^{0,n}})$ and $m_0(k_{\nu_\bullet^{0,n}})$ can be seen as functions $b(t,x)$ and $w(t,x)$ in linear problem. According to Lemma~\ref{tw:B_C1} coefficients $b$ and $w$ satisfy assumptions of Theorem~\ref{tw:ciagla_pocho}. 

\begin{lemma}\label{lem:4_3_new}
Under the assumptions {\bf A1}-{\bf A3} of Theorem~\ref{main_hille} and the conditions on the weight~$\omega$ as in Lemma~\ref{cl:B_prop}, for any $\nu_\bullet\in\CC_\omega([0,T];\ZZZ)$, the operator $B(\nu_\bullet)(t,x)=\big(b(t,x),w(t,x)\big)$ satisfies 
\begin{align*}
\|b(t,\cdot)\|_\infty&\leq \|v_0\|_\infty,  \qquad \|w(t,\cdot)\|_\infty\leq \|m_0\|_\infty\\
\|\nabla_x b(t,\cdot)\|_\infty & \leq \|v_0'\|_\infty\sup_{x\in\RRD}\|\nabla_xK_{v_0}(\cdot,x)\|_{\CCA(\RRD)}\|\nu_t\|_\ZZZ,\\
\|\nabla_x w(t,\cdot)\|_\infty &\leq \|m_0'\|_\infty \sup_{x\in\RRD}\|\nabla_xK_{m_0}(\cdot,x)\|_{\CCA(\RRD)}\|\nu_t\|_\ZZZ.\\
\end{align*}
In particular, the mapping $v_0\mapsto \|B(\nu_\bullet)\|_\infty$ and $v_0\mapsto \|\nabla B(\nu_\bullet)\|_\infty$ are bounded on $\mathcal{B}_\omega^{\mu_0,T}$.
\end{lemma}
\begin{proof}
The estimates have been shown in the proof of Lemma~\ref{lem:estim_x_K}. The statement on boundedness results are as follows.
For $\nu_\bullet\in\CC_\omega([0,T];\ZZZ)$ and for $0\leq t\leq T$ the following estimate holds
\begin{align*}
\|\nu_t\|_\ZZZ\leq \frac{1}{\omega(t)}\omega(t)\|\nu_t\|_\ZZZ\leq \Big( \inf_{0\leq t\leq T}\omega(t)\Big)^{-1}\|\nu_\bullet\|_{\ZZZ_\omega^T}
\leq 
M_T\Big(\inf_{0\leq t\leq T}\omega(t)\Big)^{-1}\|\mu_0\|_{TV},
\end{align*}
where $M_T$ is defined by~\eqref{eq:def_M_T}.
\end{proof}
Since $\CC_\omega$ is a~subspace of continuous function, the operator $S$ maps $\CC_\omega([0,\infty);\CCA(\RRD;\RR^{d+1}))$ into $\ZZZ_\omega$. 
\begin{lemma}\label{lem:S_estim_00}
If $\left( t\mapsto b(t, \cdot)\right), \left( t\mapsto \overline{b}(t, \cdot)\right), (t\mapsto w(t,\cdot)), (t\mapsto \overline{w}(t,\cdot))\in \CC([0, T]; \CCA(\RRD))$ and $\mu_0\in\MM(\RRD)$ then for any $t$, such that $0\leq t\leq T<\infty$ the following estimate holds
\begin{align*}\label{S_estim_00}
&\Vert S(b,w)(t)-S(\overline{b},\ww)(t)\Vert_\ZZZ\leq \\ \leq & e^{A_1(t)} A_2(t) \left(\int_0^t \|b(s,\cdot)-\overline{b}(s,\cdot)\|_\infty \:\dd s + \int_0^t \|w(s,\cdot)-\ww(s,\cdot)\|_\infty \:\dd s \right),
\end{align*}
where
\begin{equation}
\begin{split}
A_1(t)&:=tC_t[b, \overline{b}]+t\max\{\|w(\cdot,\cdot)\|_{\infty},\|\ww(\cdot,\cdot)\|_{\infty}\},\\ A_2(t)&:=\|\mu_0\|_{TV} (1+tC_t[w,\ww]),
\end{split}
\end{equation}
and $C_t$ is defined as $
C_t[b, \overline{b}]:=  \min \left\lbrace \sup_{0\leq s \leq t} \Vert \nabla_x b(s, \cdot) \Vert_\infty, \sup_{0\leq s \leq t} \Vert \nabla_x \overline{b}(s, \cdot) \Vert_\infty \right\rbrace.$
\end{lemma}
\begin{proof}
It will be used $\nu_t^{b,w}$ for shorthand for $S(b,w)(t)$. Take into consideration difference of two solutions $\nu_t^{b,w}$ and~$\nu_t^{\overline{b},\ww}$.
By the representation formula, Lemma \ref{lem:repr_form}, for $f\in \CCA(\RRD)$ one has
\begin{equation*}\begin{split}
&\left|\int_\RRD f\:\dd(\nu_t^{b,w}-\nu_t^{\overline{b},\ww})(x)\right|=\left|\int_\RRD f\:\dd\nu_t^{b,w}(x)-\int_\RRD f\:\dd\nu_t^{\overline{b},\ww}(x)\right|\\
=& \left|\int_{\RRD}f(X_b(t,x)) \exp\! \left(\int_0^t w(s, X_b(s,x))\:\dd s\right)\dd \nu_0(x)\right.\\
 &\quad - \left. \int_{\RRD}f(X_{\overline{b}}(t,x)) \exp\! \left(\int_0^t \ww(s, X_{\overline{b}}(s,x))\:\dd s\right)\dd \nu_0(x)\right|\\
=& \left|\int_{\RRD} \bigg[\Big( f(X_b(t,x))-f(X_{\overline{b}}(t,x))\Big)\exp\! \left(\int_0^t w(s, X_b(s,x))\:\dd s\right) \right.\\
&\quad -  \left( \exp \Big(\int_0^t \overline{w}(s, X_{\overline{b}}(s,x))\:\dd s\Big)- \exp \Big(\int_0^t \overline{w}(s, X_b(s,x))\:\dd s\Big)\right)f(X_{\overline{b}}(t,x)) \\
&\quad -  \left( \exp \Big(\int_0^t \overline{w}(s, X_b(s,x))\:\dd s\Big)- \exp \Big(\int_0^t w(s, X_b(s,x))\:\dd s\Big)\right)f(X_{\overline{b}}(t,x))\bigg] \:\dd \nu_0(x)\Big|
 \end{split}\end{equation*}
 \begin{equation*}\begin{split}
 \leq &\| \nabla_x f\|_{\infty}  \int_{\RRD} \underbrace{\Big| X_b(t,x)-X_{\overline{b}}(t,x)\Big|}_{I_1}\exp \Big(\int_0^t w(s, X_{b}(s,x))\:\dd s\Big)\:\dd \nu_0(x)\\
&  \quad+\|f\|_{\infty}\int_{\RRD} \underbrace{\left| \exp \Big(\int_0^t \overline{w}(s, X_{\overline{b}}(s,x))\:\dd s\Big)- \exp \Big(\int_0^t \overline{w}(s, X_b(s,x))\:\dd s\Big)\right|}_{I_2} \:\dd \nu_0(x)\\
&  \quad+\|f\|_{\infty}\int_{\RRD} \underbrace{\left| \exp \Big(\int_0^t \overline{w}(s, X_b(s,x))\:\dd s\Big)- \exp \Big(\int_0^t w(s, X_b(s,x))\:\dd s\Big)\right|}_{I_3} \:\dd \nu_0(x).
\end{split}
\end{equation*}
{\bf Estimate for $I_1$. } According to Lemma~\ref{lem:X_estim}, one immediately obtains
\begin{equation*}
I_1 \leq  \exp(tC_t[b, \overline{b}]) \int_0^t \|b(s, \cdot)- \overline{b}(s,\cdot)\|_\infty\dd s.
\end{equation*}
{\bf Estimate for $I_2$. }Since $\big(t\mapsto \ww(t,\cdot)\big)\in\CC([0,T];\CCA(\RRD))$, then for $t\leq T$ one has\\$\left|\int_0^t \overline{w}(s, X_{\overline{b}}(s,x))\:\dd s\right| < t \|\overline{w}(\cdot,\cdot)\|_{\infty} $
and moreover  $|\ww(s,y)|\leq \|\ww(s, \cdot)\|_{\infty}\leq \|\ww(\cdot,\cdot)\|_{\infty}$ and \[|\ww(s, y)-\ww(s, \overline{y})|\leq \|\nabla_x \ww(s, \cdot)\|_{\infty} \cdot \|y-\overline{y}\|_{\RRD}.\]
Thus one can estimate 
\begin{equation*}
\begin{split}
I_2\leq& e^{ t\|\ww(\cdot,\cdot)\|_{\infty} }\left| \int_0^t \ww(s, X_{\overline{b}}(s,x))\:\dd s - \int_0^t \ww(s, X_{b}(s,x))\:\dd s \right|\\
\leq&  e^{t \|\ww(\cdot,\cdot)\|_{\infty}}  \| \nabla_x \ww\|_{t, \infty}  
\int_0^t |X_b(s,x)-X_{\overline{b}}(s,x)|\:\dd s
\\
&\textrm{\footnotesize(according to Lemma~\ref{lem:X_estim})}
\\
\leq&\exp\Big(t\|\ww(\cdot,\cdot)\|_{\infty}\Big) \| \nabla_x \ww\|_{t, \infty}
\int_0^t \exp(s  C_s[b, \overline{b}]) \left(\int_0^s \|b(\overline{s}, \cdot)- \overline{b}(\overline{s}, \cdot)\|_\infty\dd \overline{s}\right) \:\dd s.
\end{split}
\end{equation*}
Since $\|b(\overline{s}, \cdot)- \overline{b}(\overline{s}, \cdot)\|_\infty$ is non-decreasing function and  $s\leq t$, we can further estimate\\ $\int_0^s \|b(\overline{s}, \cdot)- \overline{b}(\overline{s}, \cdot)\|_\infty\dd \overline{s}\leq \int_0^t \|b(s, \cdot)- \overline{b}(s, \cdot)\|_\infty\dd s$. Afterwards one obtains
\begin{equation*}
\begin{split}
I_2\leq\exp\Big(t \|\ww(\cdot,\cdot)\|_{\infty}\Big) \| \nabla_x \ww\|_{t, \infty} 
\int_0^t \exp(s  C_s[b, \overline{b}])\:\dd s \int_0^t \|b(s, \cdot)- \overline{b}(s, \cdot)\|_\infty\dd s.
\end{split}
\end{equation*}
Function $C_s[b,\overline{b}]$ is also non-deceasing, thus $C_s[b,\overline{b}]\leq C_t[b,\overline{b}]$, and hence
\begin{equation*}
\begin{split}
I_2\leq\exp\left(t\|\ww(\cdot,\cdot)\|_{\infty}\right)   \| \nabla_x \ww\|_{t, \infty}  
 t\exp(t  C_t[b, \overline{b}]) \int_0^t \|b(s, \cdot)- \overline{b}(s, \cdot)\|_\infty\dd s.
\end{split}
\end{equation*}
%
{\bf Estimate for $I_3$.} One gets
\begin{align*}
I_3&\leq \max\{e^{t\|w(\cdot,\cdot)\|_\infty},e^{t\|\ww(\cdot,\cdot)\|_\infty}\} \int_0^t\Big( \ww(s,X_b(s,x))-w(s,X_b(s,x))\Big)\:\dd s\\
&\leq \max\{e^{t\|w(\cdot,\cdot)\|_\infty},e^{t\|\ww(\cdot,\cdot)\|_\infty}\} \int_0^t\Big\| \ww(s,\cdot)-w(s,\cdot)\Big\|_\infty\dd s.
\end{align*}
Combining estimates for $I_1$, $I_2$ and $I_3$ and taking
supremum over $f\in \CCA(\RRD)$, such that $\|f\|_{\CCA(\RRD)}\leq 1$ one obtains
\begin{equation*}\label{estim:mu_b-b_00}
\begin{split}
\big\|\nu_t^{b,w}-&\nu_t^{\overline{b},\ww}\big\|_\ZZZ \leq  \int_{\RRD}\left(\exp(t  (\|w(\cdot,\cdot)\|_{\infty} +C_t[b, \overline{b}]))  \int_0^t\|b(s,\cdot) - \overline{b}(s,\cdot)\|_\infty\dd s  \right)\dd \nu_0(x)\\
&+ \int_{\RRD} \left( \exp(t  (\|\ww(\cdot,\cdot)\|_{\infty} +C_t[b, \overline{b}]))   t\|\nabla_x \ww\|_{t,\infty} \int_0^t \|b(s, \cdot)-\overline{b}(s, \cdot)\|_\infty\dd s\right)\dd \nu_0(x)\\
&+ \int_{\RRD} \left( \max\{\exp(t\|w(\cdot,\cdot)\|_{\infty}, \exp(t \|\ww(\cdot,\cdot)\|_{\infty})\}
   \int_0^t \|w(s, \cdot)-\ww(s, \cdot)\|_\infty\dd s\right)\dd \nu_0(x)
   \end{split}
\end{equation*}
\begin{equation*}
\begin{split}
\leq&\exp\Big(t  (\max\{\|w\|_{\infty},\|\ww\|_\infty\}+C_t[b, \overline{b}])\Big)(1+tC_t[w,\ww])   \|\nu_0\|_{TV} \\
&\quad \cdot \left(\int_0^t \|b(s, \cdot)-\overline{b}(s, \cdot)\|_\infty\dd s+\int_0^t \|w(s, \cdot)-\ww(s, \cdot)\|_\infty\dd s\right).
\end{split}
\end{equation*}
This finishes the proof.
\end{proof}
Let us go back to operator $B$.

\begin{lemma}\label{lem:estimB_00}
Assume that $M_0:=\max\{ \Vert  K_{v_0}\Vert_{\CC_b(\RRD;\:\CCA(\RRD))},\Vert K_{m_0}\Vert_{\CC_b(\RRD;\CCA(\RRD))}\}<\infty$.
Then  for any  $\nu_\bullet, \overline{\nu}_\bullet\in \CC([0,+\infty); \ZZZ)$ the following estimate holds
\[\Vert B(\nu_\bullet)(t) - B(\overline{\nu}_\bullet)(t)\Vert_{\infty} \leq 2 L_0 M_0 \Vert \nu_t - \overline{\nu}_t\Vert_\ZZZ,\]
where $L_0:=\max\{\| \nabla v_0\|_{\CC_b(\RRD)},\|\nabla m_0\|_{\CC_b(\RRD)}\}$. 
\end{lemma}

\begin{proof} One can estimate
\begin{equation}\label{esim_b_00}
\begin{split}
\Big\Vert \Big(B(\nu_\bullet)-&B(\overline{\nu}_\bullet)\Big)(t)\Big\Vert_{\infty} \leq
\left\Vert v_0\!\left(\int_{\RRD}K_{v_0}(y, \cdot)\:\dd \nu_t(y)\right)- v_0\!\left(\int_{\RRD}K_{v_0}(y, \cdot)\:\dd \overline{\nu}_t(y)\right)\right\Vert_{\infty}\\
&\qquad + \left\Vert m_0\!\left(\int_{\RRD}K_{m_0}(y, \cdot)\:\dd \nu_t(y)\right)- m_0\!\left(\int_{\RRD}K_{m_0}(y, \cdot)\:\dd \overline{\nu}_t(y)\right)\right\Vert_{\infty},
\end{split}
\end{equation}
by definition of $L_0$ one obtains
\begin{align*}\label{estimbh_00}
\Big\Vert \Big(B(\nu_\bullet)-&B(\overline{\nu}_\bullet)\Big)(t)\Big\Vert_{\infty}\leq L_0\left\|\int_{\RRD}K_{v_0}(y,\cdot) \:\dd\nu_t(y)-\int_{\RRD}K_{v_0}(y,\cdot) \:\dd\overline{\nu}_t(y)\right\|_\infty\\
&+L_0\left\|\int_{\RRD}K_{m_0}(y,\cdot) \:\dd\nu_t(y)-\int_{\RRD}K_{m_0}(y,\cdot) \:\dd\overline{\nu}_t(y)\right\|_\infty.
\end{align*}
Finally, according to Lemma~\ref{lem:estim_x_K} for $u\in\{v_0,m_0\}$
\[\left\|\int_\RRD K_u(y,x)\:\dd\nu_t(y)-\int_\RRD K_u(y,x)\overline{\nu}_t(y)\right\|_\infty\leq \Vert  K_u(\cdot, x)\Vert_{\CCA(\RRD)}\Vert \nu_t-\overline{\nu}_t\Vert_\ZZZ,\]
what finishes the proof.\end{proof}

Now we are going back to the operator $\mathcal{T}$ as a composition of $S$ and $B$ and we want to show that it is a contraction on $\mathcal{B}_\omega^{\mu_0,T}$ for a suitable weight $\omega(t)$. 

\begin{proof}[Proof of Proposition~\ref{prop:unique_00}]\label{proof:prop_unique_00}
Let $\nu_\bullet, \overline{\nu}_\bullet \in \mathcal{B}^{\mu_0, T}_\omega$. Then for any $g>0$
\begin{equation*}
\begin{split}
e^{-g t}\|(\mathcal{T}(\nu_\bullet)(t)-\mathcal{T}(\overline{\nu}_\bullet)(t)\|_\ZZZ &= e^{-g t} \|(S\circ B)(\nu_\bullet)-(S\circ B)(\overline{\nu}_\bullet))(t)\|_\ZZZ\\
& \textrm{\footnotesize(according to Lemma~\ref{lem:S_estim_00})}
\\
&
\leq e^{-g t} e^{A_1(t)}  A_2(t) \int_0^t \Big\|(B(\nu_\bullet)-B(\overline{\nu}_\bullet))(s)\Big\|_\infty \dd s.
\end{split}
\end{equation*}
Use Lemma~\ref{lem:estimB_00} and denote $2L_0 M_0=:D$ (occurs in Lemma~\ref{lem:estimB_00})
\begin{equation*}
\begin{split}
e^{-g t}\|(\mathcal{T}(\nu_\bullet)(t)-\mathcal{T}(\overline{\nu}_\bullet)(t)\|_\ZZZ &
\leq e^{-g t+A_1(t)}  A_2(t) D \int_0^t \|\nu_s - \overline{\nu}_s\|_\ZZZ \dd s\\
& = e^{-g t+A_1(t)}  A_2(t) D \int_0^t e^{g s} e^{-g s}\|\nu_s - \overline{\nu}_s\|_\ZZZ \dd s\\
& \leq e^{A_1(t)}  A_2(t) D \int_0^t e^{-g(t- s)} \sup_{0\leq s\leq t}e^{-g s}\|\nu_s-\overline{\nu}_s\|_\ZZZ \dd s\\
& \leq e^{A_1(t)}  A_2(t) D  \|\nu_s-\overline{\nu}_s\|_{\ZZZ_\omega} \int_0^t e^{-g s}  \:\dd s\\
&\leq \frac{A_2(t)D}{g}e^{A_1(t)}\|\nu_t-\overline{\nu}_t\|_{\ZZZ_\omega}.
\end{split}
\end{equation*}
Notice that for any $t\in[0,T]$  functions $A_2(t)$ and $e^{A_1(t)}$ are bounded because of Lemma~\ref{lem:4_3_new}. This bound is independent of~$\nu_\bullet$ and $\overline{\nu}_\bullet$ (but dependent on $\|\mu_0\|_{TV}$). The~constant~$D$ is independent of time. Thus for any finite~$T$ one can establish appropriate value of~$g\geq \|m_0\|_\infty$ providing $\frac{A_2(t)D}{g}e^{A_1(t)}\leq c<1$ for all $t\in[0,T]$. Therefore one obtains
  \[\left\Vert\mathcal{T}(\nu_{\bullet})-\mathcal{T}(\overline{\nu}_{\bullet})\right\Vert_{\ZZZ_\omega^T}\leq c \Vert \nu_\bullet-\overline{\nu}_\bullet\Vert_{\ZZZ_\omega^T}.\]
That is, the operator $\mathcal{T}$ is contraction on $\mathcal{B}_\omega^{\mu_0,T}$ with Lipschitz constant less than $c$.
\end{proof}
\begin{remark}
A conclusion from Proposition~\ref{prop:unique_00} is that for any bounded interval $[t_1,t_2]$ there exists a~constant $g>0$, such that $\omega(t)=e^{-gt}$ and the operator $\mathcal{T}$ is contraction on $\CC_\omega([t_1,t_2];\ZZZ)$.
\end{remark}

Now we would like to extend Proposition~\ref{prop:unique_00} on the interval~$[0,+\infty)$.
It is impossible to state that $\mathcal{T}$ is contraction on $[0,\infty)$ with weight $\omega(t)=e^{-gt}$ for a single fixed $g\geq \|m_0\|_\infty$. Let us take
\begin{equation}\label{eq:final_weight_00}
{\omega_1}(t)=\left\{ \begin{array}{ll}
e^{-g_1 t}& \textrm{for } 0\leq t< 1,\\
e^{-g_2 t}& \textrm{for } 1\leq t< 2,\\
\dots\\
e^{-g_N t}& \textrm{for } N-1\leq t< N,\\
\dots
\end{array} \right.
\end{equation}
where $g_n\geq \|m_0\|_\infty$ is a positive constant for $n\in\NN$ such that for fixed $0<c<1$ and $M_n\equiv 1$ for all $n$, the following estimate holds
\begin{equation}\label{eq:sup_sup}
\sup_{t\in[N-1,N)}e^{-g_nt}\left\|\big(\mathcal{T}\nu_\bullet-\mathcal{T}\overline{\nu}_\bullet\big)(t)\right\|_{\ZZZ}\leq c \sup_{t\in[N-1,N)}e^{-g_nt}\|\nu_t-\overline{\nu}_t\|_{\ZZZ}\qquad \textrm{ on }[N-1,N)
\end{equation}
for all
\begin{equation}\label{eqn:notat_B}
\nu_\bullet,\overline{\nu}_\bullet\in\widehat{\mathcal{B}}_\omega^{\mu_0}:=\left\{f\in\ZZZ_\omega:f(0)=\mu_0\in\MM(\RRD), \|f\|_{\ZZZ_\omega}\leq\|\mu_0\|_{TV}\right\};
\end{equation}
see notation~\eqref{set:B}. Thus, summarizing:
\begin{proposition}\label{prop:T_H_0_infty_00}
Let $0\leq c<1$ and the weight $\omega_1:\RR_+\to\RR_+$ be defined by~\eqref{eq:final_weight_00}. Then the operator $\mathcal{T}=S \circ B: \widehat{\mathcal{B}}_{\omega_1}^{\mu_0} \to \widehat{\mathcal{B}}_{\omega_1}^{\mu_0}$ is a contraction 
\[\|\mathcal{T}(\nu_\bullet)-\mathcal{T}(\overline{\nu}_\bullet)\|_{\ZZZ_{\omega_1}}\leq c\|\nu_\bullet-\overline{\nu}_\bullet\|_{\ZZZ_{\omega_1}}.\]
\end{proposition}
%
%
%
\begin{proposition}
Let coefficients $v_0$, $m_0$ satisfy assumptions of Theorem~\ref{main_hille} and let weight be defined by~\eqref{eq:final_weight_00}. Then according to Proposition~\ref{prop:T_H_0_infty_00} and Banach Fixed-Point Theorem one has, that the approximate scheme described by~\eqref{egn:non-pertur0_0000} and~\eqref{egn:non-pertur0_01} is convergent. It means that there exists $\overline{\nu}_\bullet:[0,\infty)\to\ZZZ$ such that
\[\lim_{n\to \infty}\|\nu_\bullet^{0,n}-\overline{\nu}_\bullet^{0}\|_{\ZZZ_{\omega_1}}=0.\]
\end{proposition}
 Notice that the weight defined in~\eqref{eq:final_weight_00} is the same as in~\eqref{eq:final_weig_final_3}. Recall that the weight used in Theorem~\ref{main_hille} is defined by~\eqref{eq:final_weig_final}. 
 Since $\tilde{\omega}(t)\leq\omega_1(t)$ for all $t\in[0,\infty)$, above argumentation stays true for $\tilde{\omega}$.  


\begin{corollary}\label{cry:measure}
Note that by the fixed point of $\mathcal{T}$, the mapping $t\mapsto\overline{\nu}_t^0$ in $\ZZZ_{\omega_1}$ solves the problem~\eqref{def:prob_init}, where  $t\mapsto\overline{\nu}_t^0$ is the limit of approximation scheme given by~\eqref{egn:non-pertur0_0000} and~\eqref{egn:non-pertur0_01}. 
{\cblue Being a fixed point of $\mathcal{T}$, the~mapping $\overline{\nu}_\bullet^0$ solves the linear equation with coefficients $b(t,x)=v_0(k_{\overline{\nu}_t})$, $w(t,x)=m_0(k_{\overline{\nu}_t})$, hence $\overline{\nu}_t^0\in\mathcal{M}(\RRD)$ for all $t\geq 0$. Moreover $t\mapsto \overline{\nu}_t^0$ is narrowly continuous.}
\end{corollary}

\section{Aproximation scheme for perturbed equation}\label{sec:well}
Recall the perturbed equation~\eqref{def:prob_init2} and its solution $\mu_\bullet^h$ for $h\in(-\tfrac{1}{2},\tfrac{1}{2})$. The existence of a~unique weak solution in the sense of Definition~\ref{def:weak_solu_22} has been established in the previous section for $h=0$ by means of exhibiting $\mu_\bullet^0$ as the limit in $\ZZZ_{\omega_1}$ of a~sequence of approximate solutions $\{\nu_\bullet^{0,n}\}_n$. In a~similar way one can show the existence of a~unique weak solution $\mu_\bullet^h$ for general $h\in(-\tfrac{1}{2},\tfrac{1}{2})$ by changing the velocity fields and production rates. The resulting sequence of approximation to $\mu_\bullet^h$ is not convenient for showing the continuous differentiability of $h\mapsto\mu_\bullet^h$ in $\ZZZ_{\omega_1}$. For reaching this main objective we prefer another approximation scheme that we shall now introduce.


Problem~\eqref{def:prob_init2} can be written as
 \begin{equation}\label{eqn:pertur0}
\left\{ \begin{array}{l}
\partial_t \mu_t^{h} +\div_x \Big(\big(v_0\big(k_{\mu_t^h}\big)+h\:v_1\big(k_{\mu_t^h}\big)\big) \mu_t^{h}\Big)= \Big(m_0\!\left(k_{\mu_t^h}\right)+h\:m_1\!\left(k_{\mu_t^h}\right)\Big) \mu_t^{h}\\
\mu_{t=0}^{h}=\mu_0\in\MM(\RRD).
\end{array} \right.
\end{equation} 
We assume that coefficients satisfy the assumptions of Theorem~\ref{main_hille}.
A solution $\mu_\bullet^{h}$ can be approximated by a sequence $\lbrace\nu_\bullet^{h,n}\rbrace_{n}$ defined by the following scheme. Recall that $\MM(\RRD)$ is equipped  with the topology of $\ZZZ$.

\noindent {\bf Step 0:} A unique continuous mapping $\mu_\bullet^{0}:[0,\infty)\to\MM(\RRD)$ solves the  non-perturbed, {\bf non-linear} equation
\begin{equation}\label{egn:non-pertur0}
\left\{ \begin{array}{l}
\partial_t \mu_t^{0} +\div_x \Big(v_0\!\left(k_{\mu_t^0}\right) \mu_t^{0}\Big)= m_0\!\left(k_{\mu_t^0}\right) \mu_t^{0}\\
\mu_{t=0}^{0}=\mu_0\in\MM(\RRD).
\end{array} \right. 
\end{equation} 
From previous section one has that there exists a unique weak solution to~\eqref{egn:non-pertur0} such that $\mu_t^0$ is a bounded Radon measure; see Corollary~\ref{cry:measure}. Recall that one has $\lim_{n\to\infty}\|\mu_\bullet^{0}-\nu_\bullet^{0,n}\|_{\ZZZ_{\omega_1}}=0$. Define $\nu_t^{h,0}:=\mu_t^0$.\\

\noindent {\bf Step 1:} For $n \geq 1$, a unique continuous mapping $\nu_\bullet^{h,n}:[0,\infty)\to\MM(\RRD)$ solves the {\bf linear} equation
\begin{equation}\label{approx:sch}
\left\{ \begin{array}{ll}
\partial_t \nu_t^{h,n}\! +\!\div_x \Big[\left(v_0\!\left(k_{\nu_t^{h,n-1}}\right)\!+\!h\:v_1\!\left(k_{\nu_t^{h,n-1}}\right)\right) \nu_t^{h,n}\Big]\!=\! \Big[m_0\!\left(k_{\nu_t^{h,n-1}}\right)\!+\!h\:m_1\!\left(k_{\nu_t^{h,n-1}}\right)\!\Big]\nu_t^{h,n}&\\
  \nu_0^{h,n}=\mu_0.
\end{array} \right.
\end{equation}

Notice that for $h\in(-\frac{1}{2},\frac{1}{2})$ and for every solution to the linear equation $\nu_\bullet^{h,n}$,  the coefficients 
\[v_0\left(k_{\nu_\bullet^{h,n}}\right)+h\:v_1\left(k_{\nu_\bullet^{h,n}}\right)
\quad \textrm{ and } \quad 
m_0\left(k_{\nu_\bullet^{h,n}}\right)+h\:m_1\left(k_{\nu_\bullet^{h,n}}\right)\]
satisfy the assumptions of Lemma~\ref{lem:repr_form}; see Corollary~\ref{cry:311}. By Corollary~\eqref{cry:measure} one has that~$\mu_t^0$ is bounded Radon measure. For $n=1$, the problem~\eqref{approx:sch} is the linear equation with measure initial condition $\mu_0\in\MM(\RRD)$. Hence by Lemma~\ref{lem:repr_form} solution $\nu_t^{h,1}$ also is a bounded Radon measure and $t\mapsto \nu_t^{h,1}$ is the unique solution in the class of narrowly continuous mapping. Identically for every~$\nu_t^{h,n}$, where $n\in\NN$.
Also note that $\nu_t^{0,0}=\mu_t^0$ for all $n$ in this approximation scheme, which differs from that in Section~\ref{sec:well_00} for $h=0$.

Let us consider the following problem, where $\mu_\bullet:[0,\infty)\to\MM(\RRD)$ is given:
\begin{equation}\label{eq:T_h_nemy}
\left\{ \begin{array}{l}
\partial_ t \widehat{\mu}_t +\div_x \Big((v_0(k_{\mu_t})+h\:v_1(k_{\mu_t}))\widehat{\mu}_t\Big)= \Big(m_0(k_{\mu_t})+h\:m_1(k_{\mu_t})\Big) \widehat{\mu}_t\\
\widehat{\mu}_{t=0}=\mu_0\in\MM(\RRD).
\end{array}\right. 
\end{equation} 
Define an operator $\mathcal{T}_h(\mu_\bullet):= \widehat{\mu}_\bullet$.  
Similar to Section~\ref{sec:well_00} and Proposition~\ref{prop:san_260220} for the case $h=0$ any fixed point of this operator is a solution to problem~\eqref{eqn:pertur0}.
Similarly to previous reasoning (i.e. as in the case for $\mathcal{T}$) the operator ~$\mathcal{T}_h$ can be regarded as a composition of two operators $\mathcal{T}_h= B_h\circ S$; see page~\pageref{set:B}. Operator $S$ 
is identical to the previously considered operator.
Second operator~$B_h$  describes how the approximated solution influences the velocity field $v_0(k_{\nu_t^{h,m}})$ and the function $m_0(k_{\nu_t^{h,m}})$, which will be used in the $(m{+}1)$-th step of approximation scheme (as before like in the case of operator $B$).
Now, the operator~$B_h$ depends also on parameter~$h$, hence
\[B_h:\Big(t \mapsto \mu_t\Big)\mapsto \Big(t \mapsto\big( v_0(k_{\mu_t})+h\:v_1(k_{\mu_t}), m_0(k_{\mu_t})+h\:m_1(k_{\mu_t})\big)\Big).\] 
Since $h\in(-\frac{1}{2},\frac{1}{2})$, thus the domain and values of operator $B_h$ are identical as for operator $B$
\[B_h:\ZZZ_\omega^T\to \mathcal{C}_{\omega}\left([0,T]; \CCA(\RRD;\RR^{d+1})\right);\]
see page~\pageref{def:ZZ}.
All reasoning concerning operator $B$ described in Section~\ref{sec:well_00} can be conducted also for the operator~$B_h$. In particular, the direct analogy of Proposition~\ref{prop:T_H_0_infty_00} is the following proposition.
\begin{proposition}\label{prop:T_H_0_infty}
Let $0<c<1$ and the weight $\omega_1$ be defined by~\eqref{eq:final_weight_00} with $\mathcal{T}_h=S\circ B_h$ replacing $\mathcal{T}$, such that~\eqref{eq:sup_sup} holds for all $h\in(-\tfrac{1}{2},\tfrac{1}{2})$.
Then $T_h:\widehat{\mathcal{B}}_{\omega_1}^{\mu_0} \to \widehat{\mathcal{B}}_{\omega_1}^{\mu_0}$
 is contraction with contraction constant $c$ for any $h\in(-\frac{1}{2}; \frac{1}{2})$; see notation~\eqref{eqn:notat_B}. 

The unique fixed point of operator $\mathcal{T}_h$ in $\widehat{\mathcal{B}}_{\omega_1}^{\mu_0}$ (which is $\mu_\bullet^h$) has such the property that 
\[
\lim_{n\to\infty}\|\mu_\bullet^{h}-\nu_\bullet^{h,n}\|_{\ZZZ_{\omega_1}}=0.
\]
Moreover $t\mapsto \mu_t^h$ is the unique weak solution to problem~\eqref{def:prob_init2}.
\end{proposition}

\section{Proof of Theorem~\ref{main_hille}}\label{sec:linear_tr}
In this section we finally provide a proof of the main result, Theorem~\ref{main_hille}. Thus, under assumption {\bf A1}-{\bf A4} of that theorem  we now want to show that $\lim_{\lambda\to 0}\frac{\mu_\bullet^{h+\lambda}-\mu_\bullet^{h}}{\lambda}$ exists in  $\ZZZ_{\tilde{\omega}}$ for any fixed $h\in(-\frac{1}{2},\frac{1}{2})$,
where the weight $\tilde{\omega}$ is defined by~\eqref{eq:final_weig_final}.
Moreover, we prove that the mapping $h\mapsto\mu_\bullet^h:(-\frac{1}{2};\frac{1}{2})\to\ZZZ_{\tilde{\omega}}$ is $\CC^1$.
 First, we consider differentiability at $h=0$, hence the limit 
\begin{equation}\label{eq:lim_0}
\lim_{\lambda\to 0}\frac{\mu_\bullet^{\lambda}-\mu_\bullet^0}{\lambda}
\end{equation}
is investigated.
Using Proposition~\ref{prop:T_H_0_infty} we know that for fixed $\lambda\in(-\tfrac{1}{2},\tfrac{1}{2})$ one has $\lim_{n\to\infty}\|\mu_\bullet^{\lambda}-\nu_\bullet^{\lambda,n}\|_{\ZZZ_{\tilde{\omega}}}=0$. Hence, instead of~\eqref{eq:lim_0} the double limit
 \[\lim_{\lambda\to 0}\left(\lim_{n\to\infty}\frac{\nu_\bullet^{\lambda,n}-\mu_\bullet^{0}}{\lambda}\right)\]
is considered. 
 We proceed by arguing that the order of limits can be changed. Then we can make use of results concerning the linear transport equation, in particular Theorem~\ref{tw:ciagla_pocho}.
The following theorem will be involved.
\begin{theorem}\label{lem:rudin}
Let $E$ be a set in  a metric space and $\lambda_0$ be a limit point of $E$. Suppose that
$(Y, d)$ is a complete metric space. Let $f_n, f: E \to Y$ such that $f_n \to f$ uniformly on $E$ and
\[\lim_{\lambda\to \lambda_0} f_n(\lambda)=A_n, \qquad (n=1,2,3, \ldots)\]
exists. Then the sequence $A_n$ converges in $Y$ and 
$
\lim_{\lambda\to \lambda_0} f(\lambda)=\lim_{n\to \infty} A_n.
$
\end{theorem}
Above theorem is an analog of Theorem 7.11 in \cite[page 149]{Rudi:1976}. The proof remains essentially the same as in \cite{Rudi:1976}. 

 In our case $E=(-\frac{1}{2}, \frac{1}{2})\setminus \lbrace 0 \rbrace$. Set $Y=\ZZZ_{\tilde{\omega}}$, let $\lambda_0=0$  and \[f_n(\lambda):=\frac{\nu_{\bullet}^{\lambda,n}-\mu_{\bullet}^{0}}{\lambda}\qquad \textrm{ and }\quad 
f(\lambda):=\lim_{n\to \infty} f_n(\lambda)=\frac{\mu_{\bullet}^{\lambda}-\mu_{\bullet}^{0}}{\lambda}.\] Each $f_n(\lambda)$ is a $\MM(\RRD)$-valued function on $[0,\infty)$ which can be evaluated at $t$.
In other words, to state that 
\[
\lim_{\lambda\to \lambda_0}\lim_{n\to \infty} f_n(\lambda) =\lim_{n\to \infty} \lim_{\lambda\to \lambda_0} f_n(\lambda),
\]
we need to argue that $f_n$ converges uniformly to $f$ in $E$ ({\bf 1st Step} of the proof) and also show that for all $n$ the limit $\lim_{\lambda\to \lambda_0} f_n(\lambda)$ exists in $\ZZZ_{\tilde{\omega}}$ ({\bf 2nd Step}). Both limits are considered in $\ZZZ_{\tilde{\omega}}$.
In~{\bf 3rd Step} we will show that the result remains true for any $h\in(-\tfrac{1}{2},\tfrac{1}{2})$. 



\noindent {\bf Step 1.} Recall that the space $\ZZZ_{\tilde{\omega}}$ is complete. We investigate uniform convergence $\left\{\frac{\nu_{\bullet}^{\lambda,n}-\mu_{\bullet}^{0}}{\lambda}\right\}_n$ in $\ZZZ_{\tilde{\omega}}$ for $\lambda\in E$ to $\frac{\mu_\bullet^\lambda-\mu_\bullet^0}{\lambda}$. 
We want to argue that it is a Cauchy sequence, then as a sequence in complete space~$\ZZZ_{\tilde{\omega}}$ its limit also is in~$\ZZZ_{\tilde{\omega}}$. Using a property of approximating scheme this limit will be then $\frac{\mu_\bullet^\lambda-\mu_\bullet^0}{\lambda}$. To show that the norm $\left\|\frac{\nu_\bullet^{\lambda,n}-\mu_\bullet^0}{\lambda}-\frac{\nu_\bullet^{\lambda,m}-\mu_\bullet^0}{\lambda}\right\|_{\ZZZ_{\tilde{\omega}}}$ can be made arbitrary small when $n,m\to \infty$ uniformly for $\lambda\in E$, we will need a few additional estimates.

\begin{proposition}\label{prop:pochodna_0}
If the coefficients satisfy the assumptions of Theorem~\ref{main_hille}, then
the limit \[\lim_{\lambda\to 0}\frac{\nu_\bullet^{\lambda,1}-\mu_\bullet^{0}}{\lambda}\] 
exists and is an element of $\ZZZ_{\tilde{\omega}}$. Moreover, the following estimate holds
\begin{equation}\label{eq:final_final_B_0}
\left\|\frac{\nu_t^{\lambda,1}-\mu_t^{0}}{\lambda}\right\|_\ZZZ\leq C(t) \big\|\partial_\lambda\nu^{\lambda,1}_{t}\big|_{\lambda=0}\big\|_\ZZZ,
\end{equation}
where $C(t)=\mathcal{O}(te^{gt})$ and $g>0$ can be chosen to hold for all $\lambda\in E$.
\end{proposition}
Notice that in estimation~\eqref{eq:final_final_B_0} on the one hand the weight  $\tilde{\omega}$ needs to reduce $C(t)$ but also it reduces the growth of $\partial_\lambda\nu_t^{\lambda,1}$ in time. That is why $\omega_2$ occurs in the construction of weight $\tilde{\omega}$ in~\eqref{eq:final_weig_final}.
\begin{proof}
The mapping $\mu_\bullet^{0}$ is a solution to
\begin{equation}\label{egn:non-pertur0_000}
\left\{ \begin{array}{l}
\partial_t \mu_t^{0} +\div_x \Big(v_0\!\left(k_{\mu_t^0}\right) \mu_t^{0}\Big)= m_0\!\left(k_{\mu_t^0}\right) \mu_t^{0}\\
\mu_0^{0}=\mu_0,
\end{array} \right. 
\end{equation}
and $\nu_\bullet^{\lambda,1}$ solves
\begin{equation}\label{approx:sch_000}
\left\{ \begin{array}{ll}
\partial_t \nu_t^{\lambda,1} +\div_x \Big[\left(v_0\!\left(k_{\mu_t^0}\right)+\lambda \:v_1\!\left(k_{\mu_t^0}\right)\right) \nu_t^{\lambda,1}\Big]= \Big[m_0\!\left(k_{\mu_t^0}\right)+\lambda\: m_1\!\left(k_{\mu_t^0}\right)\Big]\nu_t^{\lambda,1}&\\
 \nu_0^{\lambda,1}=\mu_0.
\end{array} \right.
\end{equation}
Solution $\mu_\bullet^{0}$ can be obtained using approximating scheme. 
The coefficients $v_0\!\left(k_{\mu_t^0}\right)$, $v_1\!\left(k_{\mu_t^0}\right)$, $m_0\!\left(k_{\mu_t^0}\right)$, $m_1\!\left(k_{\mu_t^0}\right)$ in the above problems are fixed i.e. coefficients does not depend on solution. It means that the problem~\eqref{egn:non-pertur0_000} can be seen as the linear one
\begin{equation}\label{egn:non-pertur0_00_prim}
\left\{ \begin{array}{l}
\partial_t \mu_t^{0} +\div_x \Big(\overline{b}_0(t,x)\; \mu_t^{0}\Big)= \overline{w}_0(t,x)\; \mu_t^{0}\\
\mu_0^{0}=\mu_0,
\end{array} \right. \tag{\ref{egn:non-pertur0_000}'}
\end{equation}
and similarly~\eqref{approx:sch_000} can be rewritten into a form
\begin{equation*}\label{approx:sch_00_prim}
\left\{ \begin{array}{ll}
\partial_t \nu_t^{\lambda,1} +\div_x \Big((\overline{b}_0(t,x)+\lambda \overline{b}_1(t,x))\; \nu_t^{\lambda,1}\Big)= \Big(\overline{w}_0(t,x)+\lambda \overline{w}_1(t,x)\Big)\;\nu_t^{\lambda,1}&\\
 \nu_0^{\lambda,1}=\mu_0,
\end{array} \right.\tag{\ref{approx:sch_000}'}
\end{equation*}
where $\overline{b}_i(t,x):=v_i\!\left(k_{\mu_t^0}\right)$ and $\overline{w}_i(t,x):=m_i\!\left(k_{\mu_t^0}\right)$ for $i=0,1$. By Lemma~\ref{lem:nem_char} we know that these coefficients satisfy assumptions of Theorem~\ref{th:main}.
Thus, the limit $\lim_{\lambda\to 0}\frac{\nu_t^{\lambda,1}-\mu_t^{0}}{\lambda}$ is the derivative of the solution to linear equation~\eqref{def:prob_init2_lin} with perturbation~\eqref{def:pert}.
Theorem~\ref{tw:ciagla_pocho} states that this derivative exists and is an element of~$\ZZZ$. Hence estimate~\eqref{eq:final_final_B_0} is true for fixed~$t$. By Theorem~\ref{tw:ciagla_pocho} one has $\lim_{\lambda\to 0}\frac{\nu_\bullet^{\lambda,1}-\mu_\bullet^0}{\lambda}\in\ZZZ_{\widehat{\omega}}$. Since $\widehat{\omega}(t)<\tilde{\omega}(t)$ for all $t\in[0,\infty)$, hence this limit is an element of $\ZZZ_{\tilde{\omega}}$.
\end{proof}

\begin{proposition}\label{lem:estim_z_teleskop}
For any fixed~$t$ and any $n$ the following estimates holds on $(-\tfrac{1}{2},\tfrac{1}{2})\setminus\{0\}$
\[
\left\|\frac{\nu_t^{\lambda,n+1}-\nu_t^{\lambda,n}}{\lambda}\right\|_\ZZZ\leq C(t)\: c^n\Big\|\partial_\lambda\nu_t^{\lambda,1}|_{\lambda=0}\Big\|_\ZZZ,\]
where constant $c<1$, and $C(t)=\mathcal{O}(te^{gt})$ are independent of $n$ and can be chosen to hold uniformly for $\lambda\in E$. 
\end{proposition}

\begin{proof}
First, let us recall that 
the operator $\mathcal{T}_h$ is a contraction in $\ZZZ_{{\omega_1}}$ with a contraction constant $c<1$ (cf. Proposition~\ref{prop:T_H_0_infty}). Hence it is contraction also in $\ZZZ_{\tilde{\omega}}$ with the same contraction constant. 
One obtains
\begin{equation}\label{eq:estim_final_A}
\begin{split}
&\left\|\frac{\nu_\bullet^{\lambda,n+1}-\nu_\bullet^{\lambda,n}}{\lambda}\right\|_{\ZZZ_{\tilde{\omega}}}\leq\left\Vert \frac{\mathcal{T}_h(\nu_\bullet^{\lambda,n})-\mathcal{T}_h(\nu_\bullet^{\lambda,n-1})}{\lambda}\right\Vert_{\ZZZ_{\tilde{\omega}}} 
\leq
c\left\Vert \frac{\nu_\bullet^{\lambda,n}-\nu_\bullet^{\lambda,n-1}}{\lambda}\right\Vert_{\ZZZ_{\tilde{\omega}}}\\
&=c\left\Vert \frac{\mathcal{T}_h(\nu_\bullet^{\lambda,n-1})-\mathcal{T}_h(\nu_\bullet^{\lambda,n-2})}{\lambda}\right\Vert_{\ZZZ_{\tilde{\omega}}}
 \!\!\!\!\leq 
c^2\left\Vert \frac{\nu_\bullet^{\lambda,n-1}\!-\nu_\bullet^{\lambda,n-2}}{\lambda}\right\Vert_{\ZZZ_{\tilde{\omega}}}
 \!\!\!\leq\dots \leq
 c^n \left\Vert \frac{\nu_\bullet^{\lambda,1}\!-\mu_\bullet^{0}}{\lambda}\right\Vert_{\ZZZ_{\tilde{\omega}}} 
,
\end{split}
\end{equation}
where $c<1$.
Making use of Proposition~\ref{prop:pochodna_0} and estimates~\eqref{eq:estim_final_A} we have
\[\left\|\frac{\nu_t^{\lambda,n+1}-\nu_t^{\lambda,n}}{\lambda}\right\|_{\ZZZ}\leq
C(t) \: c^n\big\|\partial_\lambda\nu^{\lambda,1}_{t}|_{\lambda=0}\big\|_\ZZZ.\]
Note that $C(t)$ is independent of $n$ and can be established to hold uniformly for $\lambda\in E$. 
The growth of $C(t)$ is described by Theorem~\ref{tw:ciagla_pocho} i.e. it is reduced by~$\omega_2(t)$. Hence, it is reduced also by weight $\tilde{\omega}(t)\leq\omega_2(t)$.
\end{proof}
Now, it will be argued that $\left\{\frac{\nu_{\bullet}^{\lambda,n}-\mu_{\bullet}^{0}}{\lambda}\right\}_n$ is a Cauchy sequence.
\begin{lemma}
The sequence $\left\{\frac{\nu_{\bullet}^{\lambda,n}-\mu_{\bullet}^{0}}{\lambda}\right\}_n$ converges uniformly for $\lambda \in(-\tfrac{1}{2},\tfrac{1}{2})$ to $\frac{\mu_\bullet^\lambda-\mu_\bullet^0}{\lambda}$ in the space~$\ZZZ_{\tilde{\omega}}$, where $\tilde{\omega}(t)$ is defined by~\eqref{eq:final_weig_final}. 
\end{lemma}
\begin{proof}
It will be first shown that $\left\{\frac{\nu_{\bullet}^{\lambda,n}-\mu_{\bullet}^{0}}{\lambda}\right\}_n$ is Cauchy sequence. Let us assume that $n\geq m$. By the definition of the norm $\|\cdot\|_{\ZZZ_{\tilde{\omega}}}$ one obtains
\begin{equation}\label{eq:final_final}
\begin{split}
&\left\| \frac{\nu_{\bullet}^{\lambda,n}-\mu_{\bullet}^{0}}{\lambda}
-
\frac{\nu_{\bullet}^{\lambda,m}-\mu_{\bullet}^{0}}{\lambda}
\right\|_{\ZZZ_{\tilde{\omega}}}=\sup_{t\geq 0} {\tilde{\omega}}(t)\left\Vert \frac{\nu_t^{\lambda,n}-\mu_t^{0}}{\lambda} -\frac{\nu_t^{\lambda,m}-\mu_t^{0}}{\lambda} \right\Vert_\ZZZ\\
&= 
\sup_{t\geq 0} {\tilde{\omega}}(t)\left\Vert \sum_{\ell=m+1}^n\frac{\nu_t^{\lambda,\ell}-\nu_t^{\lambda,\ell-1}}{\lambda} \right\Vert_\ZZZ
\leq
\sup_{t\geq 0} {\tilde{\omega}}(t) \sum_{\ell=m+1}^n\underbrace{\left\Vert\frac{\nu_t^{\lambda,\ell}-\nu_t^{\lambda,\ell-1}}{\lambda} \right\Vert_\ZZZ}_{\textrm{(using Proposition~\ref{lem:estim_z_teleskop})}}
\\
&\leq
\sup_{t\geq 0}{\tilde{\omega}}(t)\sum_{\ell=m+1}^n C(t) c^{\ell-1} \Big\|\partial_\lambda\nu^{\lambda,1}_{t}|_{\lambda=0}\Big\|_\ZZZ =\sup_{t\geq 0}{\tilde{\omega}}(t)C(t)\Big\|\partial_\lambda\nu^{\lambda,1}_{t}|_{\lambda=0}\Big\|_\ZZZ \sum_{\ell=m+1}^n c^{\ell-1}\\
&=
  \sup_{t\geq 0}{\tilde{\omega}}(t)C(t)\Big\|\partial_\lambda\nu^{\lambda,1}_{t}|_{\lambda=0}\Big\|_\ZZZ c^{m} \left(\frac{1-c^{n-m}}{1-c}\right).
\end{split}
\end{equation}
Theorem~\ref{tw:ciagla_pocho} states that the growth of $\partial_\lambda\nu^{\lambda,1}_{t}$ in $\ZZZ$ is controlled in time by a function $\mathcal{O}(te^{gt})$ uniformly for $\lambda\in(-\tfrac{1}{2};\tfrac{1}{2})$. It means that the weight $\tilde{\omega}(t)$ needs to be of order $\mathcal{O}(\tfrac{1}{t^2}e^{-2gt})$. By Proposition~\ref{lem:estim_z_teleskop} we know that $c<1$ and that the constant  $C(t)$ is independent on parameter~$\lambda$. Hence, one has
\begin{equation*}
\begin{split}
&\sup_{t\geq 0}\tilde{\omega}(t)C(t)\Big\|\partial_\lambda\nu^{\lambda,1}_{t}|_{\lambda=0}\Big\|_\ZZZ c^{m} \left(\frac{1-c^{n-m}}{1-c}\right)\xrightarrow{n,m\to\infty}0,
\end{split}
\end{equation*}
since function $\tilde{\omega}(t) C(t)$ is bounded in time and reduce the growth of $\partial_\lambda\mu_t^{\lambda,1}$.
It means that $\left\{\frac{\nu_{\bullet}^{\lambda,n}-\mu_{\bullet}^{0}}{\lambda}\right\}_n$ is a Cauchy sequence in space $\ZZZ_{\tilde{\omega}}$.  Moreover, the convergence is uniform on $\lambda\in(-\tfrac{1}{2},\tfrac{1}{2})\setminus\{0\}$. Since the space~$\ZZZ_{\tilde{\omega}}$ is complete, the limit of this sequence also is in the same space. 
Moreover, the estimate~\eqref{eq:final_final} is independent of $\lambda\in E$, thus
 one obtains
\[\left\{\frac{\nu_\bullet^{\lambda,n}-\mu_\bullet^0}{\lambda}\right\}_n \quad \textrm{converges to}\quad \frac{\mu_\bullet^{\lambda}-\mu_\bullet^0}{\lambda}\quad \textrm{as $n\to\infty$ in space $\ZZZ_{\tilde{\omega}}$ uniformly on $\lambda\in E$ },\]
what finishes the {\bf 1st Step}.\end{proof}
\noindent{\bf 2nd Step.} Now, $\lim_{\lambda\to 0} \frac{\nu_\bullet^{\lambda,n}-\mu_\bullet^{0}}{\lambda}$ will be investigated.  
\begin{lemma}\label{lem:granica}
For any $n\in\NN$, the limit $\lim_{\lambda\to 0} \frac{\nu_\bullet^{\lambda,n}-\mu_\bullet^{0}}{\lambda}$ exists in~$\ZZZ_{\tilde{\omega}}$.
\end{lemma}
A short proof of lemma will appear on page~\pageref{proof:short}.
For this proof, we need the following property.
\begin{proposition}\label{prop:lim_final}
For any $\ell\geq 2$ the limit $\lim_{\lambda\to 0}\frac{\nu_\bullet^{\lambda,\ell}-\nu_\bullet^{\lambda,\ell-1}}{\lambda}$
exists in~$\ZZZ_{\tilde{\omega}}$.
\end{proposition}

\begin{proof}[Proof of Proposition~\ref{prop:lim_final}]
It will be shown by induction. Introduce the notation $\overline{u}^\lambda(\cdot):=u_0(\cdot)+\lambda u_1(\cdot)$  for $u\in\{v,m\}$.
The mapping $\nu_\bullet^{\lambda,\ell}$ is the unique solution to the linear equation
\begin{equation}\label{eq:pert_non}
\left\{ \begin{array}{ll}
\partial_t\nu_t^{\lambda,\ell}+\div_x\left(\overline{v}^\lambda(k_{\nu_t^{\lambda,\ell-1}})\; \nu_t^{\lambda,\ell}\right)=\overline{m}^\lambda(k_{\nu_t^{\lambda,\ell-1}})\;\nu_t^{\lambda,\ell}\\
\nu_{t=0}^{\lambda,\ell}=\mu_0,
\end{array}\right.
\end{equation}
where for $\nu_t^{\lambda,0}$ we take $\nu_t^0$. Similarly for $\nu_t^{\lambda,\ell-1}$.
In the problem~\eqref{eq:pert_non} expand coefficients $\overline{u}^\lambda\left(k_{\nu_t^{\lambda,\ell-1}}\right)$ into Taylor series at $k_{\nu_t^{\lambda,\ell-2}}$ for $u\in\{v,m\}$ -- from Lemma~\ref{lem:nem_char}, we know that coefficients can be expanded. Then problem~\eqref{eq:pert_non} has the form
\begin{small}
\begin{equation}\label{eq:ind_2}
\left\{ \begin{array}{ll}
\partial_t\nu_t^{\lambda,\ell}+\div_x\!\left[\left(\overline{v}^\lambda(k_{\nu_t^{\lambda,\ell-2}})\!+\!\nabla_\nu \overline{v}^\lambda(k_{\nu_t^{\lambda,\ell-2}})\!\cdot\!\big(k_{\nu_t^{\lambda,\ell-1}}\!-\!k_{\nu_t^{\lambda,\ell-2}}\big)\!+\!o(\|k_{\nu_t^{\lambda,\ell-1}}\!-\!k_{\nu_t^{\lambda,\ell-2}}\|_\infty)\right)  \nu_t^{\lambda,\ell}\right]=\\
\qquad=\left[\overline{m}^\lambda(k_{\nu_t^{\lambda,\ell-2}})+\nabla_\nu \overline{m}^\lambda(k_{\nu_t^{\lambda,\ell-2}})\cdot\big(k_{\nu_t^{\lambda,\ell-1}}-k_{\nu_t^{\lambda,\ell-2}}\big) + o(\|k_{\nu_t^{\lambda,\ell-1}}-k_{\nu_t^{\lambda,\ell-2}}\|_\infty)\right]\nu_t^{\lambda,\ell}\\
\nu_{t=0}^{\lambda,\ell}=\mu_0.
\end{array}
\right.\tag{\ref{eq:pert_non}'}
\end{equation}
\end{small}
\noindent {\bf Base case}. For $\ell=2$ the problem~\eqref{eq:ind_2} can be written as
\begin{equation}\label{eq:ind_2_1}
\left\{ \begin{array}{ll}
\partial_t\nu_t^{\lambda,2}+\div_x\left[\left(\overline{v}^\lambda(k_{\mu_t^{0}})+\nabla_\nu \overline{v}^\lambda(k_{\mu_t^{0}})\cdot\big(k_{\nu_t^{\lambda,1}}-k_{\mu_t^{0}}\big) + o(\|k_{\nu_t^{\lambda,1}}-k_{\mu_t^{0}}\|_\infty)\right)  \nu_t^{\lambda,2}\right]=\\
\qquad \qquad \qquad \qquad =\left[\overline{m}^\lambda(k_{\mu_t^{0}})+\nabla_\nu \overline{m}^\lambda(k_{\mu_t^{0}})\cdot\big(k_{\nu_t^{\lambda,1}}-k_{\mu_t^{0}}\big) + o(\|k_{\nu_t^{\lambda,1}}-k_{\mu_t^{0}}\|_\infty)\right]\nu_t^{\lambda,2}\\
\nu_{t=0}^{\lambda,2}=\mu_0.
\end{array}
\right.
\end{equation}
The value $\|k_{\nu_t^{\lambda,1}}-k_{\mu_t^{0}}\|_\infty$ can be estimated by
\begin{equation}\label{eq:estim_diff}
\begin{split}
\left\|k_{\nu_t^{\lambda,1}}-k_{\mu_t^{0}}\right\|_\infty&=\left\|\int_\RRD K(y,x)\:\dd(\nu_t^{\lambda,1}-\mu_t^0)\right\|_\infty\leq \|K\|_{\CCA}\|\nu_t^{\lambda,1}-\mu_t^0\|_\ZZZ\\
&\leq \|K\|_{\CCA}\lambda\: \underbrace{\left\|\frac{\nu_t^{\lambda,1}-\mu_t^0}{\lambda}\right\|_\ZZZ}_{\textrm{(using~\eqref{eq:final_final_B_0})}}
\leq C(t)\|K\|_{\CCA}\|\partial_\lambda\nu_t^{\lambda,1}|_{\lambda=0}\|_\ZZZ \lambda.
\end{split}
\end{equation}
By Theorem~\ref{lem:main_2} the growth of $\big\|\partial_\lambda\nu_t^{\lambda,1}|_{\lambda=0}\big\|_\ZZZ$ in time is controlled. By assumption~{\bf A3} of Theorem~\ref{main_hille} one has, that value $\|K\|_{\CCA}$ is bounded. Hence, one has \[o\left(\left\|k_{\nu_t^{\lambda,2}}-k_{\nu_t^{\lambda,1}}\right\|_\infty\right)=C(t)\mathcal{O}(|\lambda|).\]
The problem~\eqref{eq:ind_2_1} can be written as
\begin{equation}\label{eq:ind_2_2}
\left\{ \begin{array}{ll}
\partial_t\nu_t^{\lambda,2}+\div_x\left[\left(\overline{v}^\lambda(k_{\mu_t^{0}})+\lambda\nabla_\nu \overline{v}^\lambda(k_{\mu_t^{0}})\frac{k_{\nu_t^{\lambda,1}}-k_{\mu_t^{0}}}{\lambda} + C(t)\mathcal{O}(|\lambda|)\right)  \nu_t^{\lambda,2}\right]=\\
\qquad \qquad \qquad \qquad =\left[\overline{m}^\lambda(k_{\mu_t^{0}})+\lambda\nabla_\nu \overline{m}^\lambda(k_{\mu_t^{0}})\frac{k_{\nu_t^{\lambda,1}}-k_{\mu_t^{0}}}{\lambda} +C(t) \mathcal{O}(|\lambda|)\right]\nu_t^{\lambda,2}\\
\nu_{t=0}^{\lambda,2}=\mu_0.
\end{array}
\right.\tag{\ref{eq:ind_2_1}'}
\end{equation}
The problem~\eqref{eq:pert_non} for $\ell=1$ has the form
\begin{equation}\label{eq:pert_non_3}
\left\{ \begin{array}{ll}
\partial_t\nu_t^{\lambda,1}+\div_x\left(\overline{v}^\lambda(k_{\mu_t^{0}})\; \nu_t^{\lambda,1}\right)=\overline{m}^\lambda(k_{\mu_t^{0}})\;\nu_t^{\lambda,1}\\
\nu_{t=0}^{\lambda,1}=\mu_0.
\end{array}\right.
\end{equation}
Introduce the notation
\begin{align*}
b_{0,2}(t,x)&:=\overline{v}^\lambda(k_{\mu_t^{0}})&\quad b_{1,2}(t,x)&:=\nabla_\nu \overline{v}^\lambda(k_{\mu_t^{0}})\:\frac{k_{\nu_t^{\lambda,1}}-k_{\mu_t^{0}}}{\lambda},\\
w_{0,2}(t,x)&:=\overline{m}^\lambda(k_{\mu_t^{0}})&\quad w_{1,2}(t,x)&:=\nabla_\nu \overline{m}^\lambda(k_{\mu_t^{0}})\:\frac{k_{\nu_t^{\lambda,1}}-k_{\mu_t^{0}}}{\lambda}.
\end{align*}
The second number in lower index in $b$ and $w$ informs about the current value of  $\ell$.
Under assumptions of Theorem~\ref{main_hille} we can make use of Theorem~\ref{lem:main_2}. We know that for $i=0,1$ coefficients $b_{i,2}$ and $w_{i,2}$ satisfy assumptions of Theorem~\ref{lem:main_2}.
Using above notation, the problem~\eqref{eq:pert_non_3} can be written as
\begin{equation}\label{eq:11}
\left\{ \begin{array}{ll}
\partial_t\nu_t^{\lambda,1}+\div_x\left(b_{0,2}(t,x)\; \nu_t^{\lambda,1}\right)&=w_{0,2}(t,x)\;\nu_t^{\lambda,1}\\
\nu_{t=0}^{\lambda,1}=\mu_0,
\end{array}\right.\tag{\ref{eq:pert_non_3}'}
\end{equation}
and the problem~\eqref{eq:ind_2_2} has the form
\begin{small}
\begin{equation}\label{eq:22}
\left\{ \begin{array}{ll}
&\partial_t\nu_t^{\lambda,2}+\div_x\!\left[\left(b_{0,2}(t,x)\!+\!\lambda b_{1,2}(t,x)\! + \!C(t)\mathcal{O}(|\lambda|)\right)  \nu_t^{\lambda,2}\right]=\\
&\qquad\qquad\qquad=\left[w_{0,2}(t,x)\!+\!\lambda w_{1,2}(t,x)\! + \!C(t)\mathcal{O}(|\lambda|)\right]\nu_t^{\lambda,2}\\
&\nu_{t=0}^{\lambda,2}=\mu_0.
\end{array}
\right.\tag{\ref{eq:ind_2_1}''}
\end{equation}
\end{small}
It means that~\eqref{eq:22} is nothing else than~\eqref{eq:11} with perturbation
\[b_{0,2}(t,x)+h\:b_{1,2}(t,x)+C(t)\mathcal{O}(|h|)\quad \textrm{ and }\quad w_{0,2}(t,x)+h\: w_{1,2}(t,x)+C(t)\mathcal{O}(|h|).\]
%
By Theorem~\ref{lem:main_2} we know that $\nu_t^{\lambda,2}$ is differentiable with this perturbation, it means that the limit
\begin{equation}\label{eq:lim_2}
\lim_{\lambda\to 0} \frac{\nu_\bullet^{\lambda,2}-\nu_\bullet^{\lambda,1}}{\lambda}=\partial_\lambda\nu_\bullet^{\lambda,2}\big|_{\lambda=0}\quad \textrm{ exists and is an element of }\ZZZ_{\tilde{\omega}}.
\end{equation}
\medskip

\noindent{\bf Induction step.} Let us assume that for some $\ell=l$ it holds that 
\begin{equation}\label{eq:zal_induk}
\lim_{\lambda\to 0}\frac{\nu_\bullet^{\lambda,l}-\nu_\bullet^{\lambda, l-1}}{\lambda}\quad\textrm{exists in}\quad \ZZZ_{\tilde{\omega}}.\end{equation}
We will show, that it implies $
\lim_{\lambda\to 0}\frac{\nu_\bullet^{\lambda,l+1}-\nu_\bullet^{\lambda, l}}{\lambda}$ exists in $\ZZZ_{\tilde{\omega}}$.\\

\noindent For $\ell=l+1$ a problem~\eqref{eq:ind_2} on page~\pageref{eq:ind_2} has the form\begin{small}
\begin{equation}\label{eq:ind_l_1}
\left\{ \begin{array}{ll}
\partial_t\nu_t^{\lambda,l+1}+\div_x\!\left[\left(\overline{v}^\lambda(k_{\nu_t^{\lambda,l-1}})+\nabla_\nu \overline{v}^\lambda(k_{\nu_t^{\lambda,l-1}})(k_{\nu_t^{\lambda,l}}-k_{\nu_t^{\lambda,l-1}}) + o(\|k_{\nu_t^{\lambda,l}}-k_{\nu_t^{\lambda,l-1}}\|_\infty)\right)  \nu_t^{\lambda,l+1}\right]=\\
\qquad\qquad \qquad =\left[\overline{m}^\lambda(k_{\nu_t^{\lambda,l-1}})+\nabla_\nu \overline{m}^\lambda(k_{\nu_t^{\lambda,l-1}})(k_{\nu_t^{\lambda,l}}-k_{\nu_t^{\lambda,l-1}}) + o(\|k_{\nu_t^{\lambda,l}}-k_{\nu_t^{\lambda,l-1}}\|_\infty)\right]\nu_t^{\lambda,l+1}\\
\nu_{t=0}^{\lambda,l+1}=\mu_0.
\end{array}
\right.
\end{equation}
\end{small}

\noindent Analogously to~\eqref{eq:estim_diff}, by assumption~\eqref{eq:zal_induk} one has $o\left(\left\|k_{\nu_t^{\lambda,l}}-k_{\nu_t^{\lambda,l-1}}\right\|_\infty\right)=C(t)\mathcal{O}(|\lambda|)$.
The problem~\eqref{eq:ind_l_1} can be written as
\begin{equation}\label{eq:ind_l_2}
\left\{ \begin{array}{ll}
\partial_t\nu_t^{\lambda,l+1}+\div_x\left[\left(\overline{v}^\lambda(k_{\nu_t^{\lambda,l-1}})+\lambda\nabla_\nu \overline{v}^\lambda(k_{\nu_t^{\lambda,l-1}})\frac{k_{\nu_t^{\lambda,l}}-k_{\nu_t^{\lambda,l-1}}}{\lambda} + C(t)\mathcal{O}(|\lambda|)\right)  \nu_t^{\lambda,l+1}\right]=\\
\qquad \qquad \qquad \qquad =\left[\overline{m}^\lambda(k_{\nu_t^{\lambda,l-1}})+\lambda\nabla_\nu \overline{m}^\lambda(k_{\nu_t^{\lambda,l-1}})\frac{k_{\nu_t^{\lambda,l}}-k_{\nu_t^{\lambda,l-1}}}{\lambda} + C(t)\mathcal{O}(|\lambda|)\right]\nu_t^{\lambda,l+1}\\
\nu_{t=0}^{\lambda,l+1}=\mu_0.
\end{array}
\right.\tag{\ref{eq:ind_l_1}'}
\end{equation}
For $\ell=l$ the problem~\eqref{eq:pert_non} has the form
\begin{equation}\label{eq:pert_non_l}
\left\{ \begin{array}{ll}
\partial_t\nu_t^{\lambda,l}+\div_x\left(\overline{v}^\lambda(k_{\nu_t^{\lambda,l-1}}) \;\nu_t^{\lambda,l}\right)=\overline{m}^\lambda(k_{\nu_t^{\lambda,l-1}})\;\nu_t^{\lambda,l}\\
\nu_{t=0}^{\lambda,l}=\mu_0.
\end{array}\right.
\end{equation}
Let us use the notation
\begin{align*}
b_{0,l+1}(t,x)&:=\overline{v}^\lambda\!\left(k_{\nu_t^{\lambda,l-1}}\right)&\quad b_{1,l+1}(t,x)&:=\nabla_\nu \overline{v}^\lambda\!\left(k_{\nu_t^{\lambda,l-1}}\right)\:\frac{k_{\nu_t^{\lambda,l}}-k_{\nu_t^{\lambda,l-1}}}{\lambda},\\
w_{0,l+1}(t,x)&:=\overline{m}^\lambda\!\left(k_{\nu_t^{\lambda,l-1}}\right)&\quad w_{1,l+1}(t,x)&:=\nabla_\nu \overline{m}^\lambda\!\left(k_{\nu_t^{\lambda,l-1}}\right)\:\frac{k_{\nu_t^{\lambda,l}}-k_{\nu_t^{\lambda,l-1}}}{\lambda}.
\end{align*}
Analogously as in  {\bf base case}, by induction assumption~\eqref{eq:zal_induk} one has  $\left\|\frac{k_{\nu_\bullet^{\lambda,l}}-k_{\nu_\bullet^{\lambda,l-1}}}{\lambda}\right\|_\infty\leq C(t)\lambda \left\|\partial_\lambda\nu_\bullet^{\lambda,l}|_{\lambda=0}\right\|_{\ZZZ_{\tilde{\omega}}}$, where $C(t)$ and $\tilde{\omega}(t)$ are such that this value is bounded. Hence, one has that for $i=0,1$ coefficients $b_{i,l+1}$ and $w_{i,l+1}$ satisfy assumption of Theorem~\ref{lem:main_2}.
Using above notation, the problem~\eqref{eq:pert_non_l} can be written in the following form
\begin{equation}\label{eq:1_l}
\left\{ \begin{array}{ll}
\partial_t\nu_t^{\lambda,l}+\div_x\left(b_{0,l+1}(t,x) \;\nu_t^{\lambda,l}\right)&=w_{0,l+1}(t,x)\;\nu_t^{\lambda,l}\\
\nu_{t=0}^{\lambda,l}=\mu_0,
\end{array}\right.\tag{\ref{eq:pert_non_l}'}
\end{equation}
and the problem~\eqref{eq:ind_l_2} can be written as
\begin{equation}\label{eq:l_2}
\left\{ \begin{array}{ll}
\partial_t\nu_t^{\lambda,l+1}+\div_x\left[\left(b_{0,l+1}(t,x)+\lambda b_{1,l+1}(t,x) + C(t)\mathcal{O}(|\lambda|)\right)  \nu_t^{\lambda,l+1}\right]=\\
\qquad\qquad\qquad \qquad=\Big[w_{0,l+1}(t,x)+\lambda w_{1,l+1}(t,x) + C(t)\mathcal{O}(|\lambda|)\Big]\nu_t^{\lambda,l+1}\\
\nu_{t=0}^{\lambda,l+1}=\mu_0.
\end{array}
\right.\tag{\ref{eq:ind_l_1}''}
\end{equation}
It means, that the problem~\eqref{eq:l_2} is nothing else than~\eqref{eq:1_l} with the following perturbation in coefficients
\[b_{0,l+1}(t,x)+\lambda\: b_{1,l+1}(t,x)+C(t)\mathcal{O}(|\lambda|)\quad \textrm{and}\quad w_{0,l+1}(t,x)+\lambda\: w_{1,l+1}(t,x)+C(t)\mathcal{O}(|\lambda|).\]
By Theorem~\ref{lem:main_2} one has that $\nu_\bullet^{\lambda,l+1}$ is differentiable with respect to such perturbation, i.e. the limit $\lim_{\lambda\to 0} \frac{\nu_\bullet^{\lambda,l+1}-\nu_\bullet^{\lambda,l}}{\lambda}$ exists in $\ZZZ_{\tilde{\omega}}$.
Finally, one has that $\lim_{\lambda\to 0} \frac{\nu_\bullet^{\lambda,\ell}-\nu_\bullet^{\lambda,\ell-1}}{\lambda}$ exists in $\ZZZ_{\tilde{\omega}}$ for any $\ell\in\NN$, such that $\ell\geq 2$.
\end{proof}
Now we can show that $\lim_{\lambda\to 0}\frac{\nu_\bullet^{\lambda,n}-\mu_\bullet^0}{\lambda}$ exists and is an element of $\ZZZ_{\tilde{\omega}}$ for any $n\in\NN$.
\begin{proof}[Proof of Lemma~\ref{lem:granica}]\label{proof:short}
Let us notice that \begin{equation}\label{eq:lim_sum}
\begin{split}
\lim_{\lambda\to 0} \frac{\nu_\bullet^{\lambda,n}-\mu_\bullet^{0}}{\lambda}
=&
\lim_{\lambda\to 0}\left( \frac{\nu_\bullet^{\lambda,1}-\mu_\bullet^{0}}{\lambda}+\sum_{\ell=2}^n \frac{\nu_\bullet^{\lambda,\ell}-\nu_\bullet^{\lambda,\ell-1}}{\lambda}\right)\\
=&
\lim_{\lambda\to 0}\frac{\nu_\bullet^{\lambda,1}-\mu_\bullet^{0}}{\lambda}+\sum_{\ell=2}^n \lim_{\lambda\to 0} \frac{\nu_\bullet^{\lambda,\ell}-\nu_\bullet^{\lambda,\ell-1}}{\lambda}.
\end{split}
\end{equation}
By Proposition~\ref{prop:pochodna_0} we know that the limit $\lim_{\lambda\to 0}\frac{\nu_\bullet^{\lambda,1}-\mu_\bullet^{0}}{\lambda}$ is an element of $\ZZZ_{\tilde{\omega}}$. 
By Proposition~\ref{prop:lim_final} one has that $\lim_{\lambda\to 0} \frac{\nu_\bullet^{\lambda,\ell}-\nu_\bullet^{\lambda,\ell-1}}{\lambda}\in\ZZZ_{\tilde{\omega}}$ for any $\ell\geq 2$.
Hence~\eqref{eq:lim_sum} also is an element of~$\ZZZ_{\tilde{\omega}}$, as a finite sum of elements from space $\ZZZ_{\tilde{\omega}}$.
\end{proof}

This finishes the {\bf 2nd Step} of the proof of Theorem~\ref{main_hille}, which together with {\bf 1st Step} guarantees, that assumptions of Theorem~\ref{lem:rudin} are satisfied and the order of limits can be changed. 
Thus, all limits below exist in $\ZZZ_{\tilde{\omega}}$
\[\lim_{\lambda\to 0}  \frac{\nu_\bullet^{\lambda}-\mu_\bullet^{0}}{\lambda}
=
\lim_{\lambda\to 0}\left( \lim_{n\to \infty} \frac{\mu_\bullet^{\lambda,n}-\mu_\bullet^{0}}{\lambda}\right)
=
\lim_{n\to \infty}\left(\lim_{\lambda\to 0}  \frac{\nu_\bullet^{\lambda,n}-\mu_\bullet^{0}}{\lambda}\right)
=
\lim_{n\to\infty} \partial_\lambda \nu_\bullet^{\lambda,n}|_{\lambda=0}.
\]

\noindent{\bf 3rd Step.} The same argumentation can be applied for $h\neq 0$. 
\begin{theorem}
The limit $\lim_{\lambda\to 0}\frac{\mu_\bullet^{h+\lambda}-\mu_\bullet^{h}}{\lambda}$
exists in $\ZZZ_{\tilde{\omega}}$.
\end{theorem}

\begin{proof}
For fixed $h$ it can be written
\begin{equation}\label{eqn:last}
\begin{split}
u_0\!\left(k_{\mu_t^{h+\lambda}}\right)+(h+\lambda)\: u_1\!\left(k_{\mu_t^{h+\lambda}}\right)&= \Big(u_0\!\left(k_{\mu_t^{h+\lambda}}\right)+h u_1\!\left(k_{\mu_t^{h+\lambda}}\right)\Big)+\lambda u_1\!\left(k_{\mu_t^{h+\lambda}}\right)\\
 &=: \overline{u}_0^h\!\left(k_{\mu_t^{h+\lambda}}\right)+\lambda u_1\!\left(k_{\mu_t^{h+\lambda}}\right)
 \end{split}
\end{equation}
for $u\in\{v,m\}$.
%
Hence, instead of solving the problem
\begin{equation*}\label{eq:final}
\left\{ \begin{array}{l}
 \partial_t \mu_t^{h+\lambda} +\div_x\left[\left(v_0\!\left(k_{\mu_t^{h+\lambda}}\right)+(h+\lambda)\: v_1\!\left(k_{\mu_t^{h+\lambda}}\right)\right) \mu_t^{h+\lambda}\right]=\\
\qquad  \qquad \qquad \qquad \qquad =\left[m_0\!\left(k_{\mu_t^{h+\lambda}}\right)+(h+\lambda)\: m_1\!\left(k_{\mu_t^{h+\lambda}}\right)\right]\mu_t^{h+\lambda}\\ 
 \mu_0^{h+\lambda}=\mu_0,
\end{array} \right. 
\end{equation*}
the following can be considered
\begin{equation*}
\left\{ \begin{array}{l}
 \partial_t \overline{\mu}_t^{\lambda} +\div_x \left[\left(\overline{v}_0^h\!\left(k_{\overline{\mu}_t^\lambda}\right)+\lambda v_1\!\left(k_{\overline{\mu}_t^\lambda}\right)\right) \overline{\mu}_t^{\lambda}\right]=\left[\overline{m}_0^h\!\left(k_{\overline{\mu}_t^\lambda}\right)+\lambda m_1\!\left(k_{\overline{\mu}_t^\lambda}\right)\right] \overline{\mu}_t^{\lambda}\\
 \overline{\mu}_0^{\lambda}=\mu_0,
\end{array} \right. 
\end{equation*}
where $\overline{\mu}_\bullet^{\lambda}:=\mu_\bullet^{h+\lambda}$.
Thus, the problem of differentiability at $h\neq 0$ is reduced to differentiability at $h=0$, i.e
\[\lim_{\lambda\to 0}\frac{\mu_\bullet^{h+\lambda}-\mu_\bullet^h}{\lambda}=\lim_{\lambda\to 0}\frac{\overline{\mu}_\bullet^{\lambda}-\overline{\mu}_\bullet^0}{\lambda}.\]
Thus having shown existence of the derivative of $h\mapsto \mu_\bullet^h$ for any $h\in(-\tfrac{1}{2};\tfrac{1}{2})$, now note that the estimates for the steps 1 and 2 involve bounds on $u_0$. The mapping $h\mapsto\overline{u}_0^h$ is bounded on $(-\frac{1}{2};\tfrac{1}{2})$. So convergence is uniform for $h\in(-\tfrac{1}{2};\tfrac{1}{2})$. Hence we have shown the continuity of the derivative with respect to $h$, i.e $h\mapsto\partial_h\mu_t^h\in\CC((-\frac{1}{2},\frac{1}{2});\ZZZ)$.
It means that $h\mapsto\mu_\bullet^h$ is of a~class $\CC^1\big((-\frac{1}{2},\frac{1}{2});\CC_{\tilde{\omega}}\big([0,\infty);\ZZZ\big)\big)$.

This finally finishes the proof of Theorem~\ref{main_hille}.
\end{proof}
\newpage
\appendix

\section{Proof of Lemma~\ref{lem:nem_char}}\label{apx:lem36}
The map $\mu\mapsto u(k_\mu)$ is of a class $\CC^1(\ZZZ;\CCA(\RRD))$ if the Fr\'echet derivative exists (Step 1 of the proof) and both $\mu\mapsto u(k_\mu)$  and $\mu\mapsto\partial u(k_\mu)$ are bounded (Step 2 and Step~3).

{\bf Step 1.} For differentiability of the map $\mu\mapsto u(k_\mu)$  at $\overline{\mu}$ we need to prove that there exists an operator $\partial u(k_{\overline{\mu}})(\cdot) \in \mathcal{L}( \ZZZ; \CCA(\RRD))$, such that
\[u(k_\mu)=u(k_{\overline{\mu}})+ \partial u(k_{\overline{\mu}})(\mu-\overline{\mu})+o\left(\Vert \mu-\overline{\mu} \Vert_{\ZZZ}\right).\]
If $u$ is differentiable at $\overline{\mu}$, then the operator $\partial u(k_{\overline{\mu}})$ will be given by the directional derivative. Let us check the operator~\eqref{eq:deriv_direct_N} i.e. $\partial u(k_{\overline{\mu}})(\mu-\overline{\mu})=(u' \circ k_{\overline{\mu}})\: k_{\mu-\overline{\mu}}$.
Sufficient condition for  $D_{\overline{\mu}} u(k_\mu)$ to be an element of space $\CCA(\RRD)$ is that $u' \in \CCA$, thus $u\in\CC^{2+\alpha}$. This provides the mapping
 $x \mapsto \left(u' \circ k_{\overline{\mu}}\right)\!(x)\in \CCA(\RRD)$.
Let us estimate
\begin{equation}\label{eqn:estim_III_3}
\begin{split}
& \Big\|u(k_\mu)-u(k_{\overline{\mu}})- \partial u(k_{\overline{\mu}})(\mu-\overline{\mu})\Big\|_{\CCA(\RRD)}=\\
=&
\underbrace{\Big\| u(k_\mu)-u(k_{\overline{\mu}})- \partial u(k_{\overline{\mu}})(\mu-\overline{\mu})\Big\|_{\infty}}_{=:II.1}+
\underbrace{\Big\| \nabla_x \Big(u(k_\mu)-u(k_{\overline{\mu}})- \partial u(k_{\overline{\mu}})(\mu-\overline{\mu})  \Big)\Big\|_{\infty}}_{=:II.2}\\
& \qquad+
\underbrace{ \left| \nabla_x \Big(u(k_\mu) - u(k_{\overline{\mu}})- \partial u(k_{\overline{\mu}})(\mu-\overline{\mu})\Big)\right|_\alpha}_{=:II.3}.
\end{split}
\end{equation}
 We estimate $II.1, II.2$ and $II.3$ separately.
Expanding $u(k_{\mu}(x))$ at $k_{\overline{\mu}}(x)$ up to the first order one obtains
\begin{equation*}
\begin{split}
II.1&=\Big\| u(k_\mu)-u(k_{\overline{\mu}})- \partial u(k_{\overline{\mu}})(\mu-\overline{\mu})\Big\|_{\infty}\\
&=
\Big\| u( k_{\overline{\mu}}) + u'( k_{\overline{\mu}})\: k_{\mu-\overline{\mu}} + R_1[u,k_{\overline{\mu}}](k_\mu)-
u (k_{\overline{\mu}})
- u'( k_{\overline{\mu}})\: k_{\mu-\overline{\mu}}\Big\|_{\infty}\\
&= \left\Vert R_1[u,k_{\overline{\mu}}](k_\mu)\right\Vert_{\infty}.
\end{split}
\end{equation*}
Let us first use Proposition~\ref{prop:remainder estimate Holder}, and then estimate~\eqref{estim:bound1}
\begin{equation}\label{eqn:II.1}
II.1=\sup_{x\in\RRD}\big|R_1\big[u,k_{\overline{\mu}}(x)\big](k_{\mu}(x))\big|\leq C_{1,\alpha} |u'|_\alpha \sup_{x\in\RRD}|k_{\mu-\overline{\mu}}(x)|^{1+\alpha}\leq C\: |u'|_\alpha \: \|\mu-\overline{\mu}\|_\ZZZ^{1+\alpha},\end{equation}
where constant $C$ does not depend on time, value $|u'|_\alpha$ neither.

Notice that $\frac{\partial}{\partial x_j}[u\circ k_\mu]=(u'\circ k_\mu(x))\: \frac{\partial k_\mu(x)}{\partial x_j}$ (by chain rule) and let us estimate $II.2$
\begin{equation*}
\begin{split}
&II.2=d\max_j\left\Vert\frac{\partial}{\partial x_j}\left[ u\circ k_\mu \right]
-\frac{\partial}{\partial x_j}\left[u\circ  k_{\overline{\mu}}\right]
- \frac{\partial}{\partial x_j}\left[(u' \circ k_{\overline{\mu}})\: k_{\mu-\overline{\mu}}\right]\right\Vert_{\infty}\\
&=d\max_j\Big\Vert
(u'\circ k_\mu) \frac{\partial k_\mu}{\partial x_j}
- (u'\circ k_{\overline{\mu}}) \frac{\partial k_{\overline{\mu}}}{\partial x_j}
-  (u''\circ k_{\overline{\mu}}) \frac{\partial k_{\overline{\mu}}}{\partial x_j} \: k_{\mu-\overline{\mu}}
- (u' \circ k_{\overline{\mu}}) \frac{\partial k_{\mu-\overline{\mu}}}{\partial x_j}
\Big\Vert_{\infty}
\\
&=d\max_j\Big\Vert
(u'\circ k_\mu) \frac{\partial k_\mu}{\partial x_j}
- (u'\circ k_{\overline{\mu}}) \frac{\partial k_{\mu}}{\partial x_j}
-  (u''\circ k_{\overline{\mu}}) \frac{\partial k_{\overline{\mu}}}{\partial x_j} \: k_{\mu-\overline{\mu}}
\Big\Vert_{\infty}
\\
&\textrm{\footnotesize{(expanding $u'\circ k_\mu(x)$ into Taylor series at $k_{\overline{\mu}}(x)$)}}
\end{split}
\end{equation*}
\begin{equation*}
\begin{split}
&=d\max_j\Big\|\big[(u'\circ k_{\overline{\mu}})
+
(u''\circ k_{\overline{\mu}})\: k_{\mu-\overline{\mu}}+R_1[u',k_{\overline{\mu}}](k_\mu)
-
(u'\circ k_{\overline{\mu}})\big]\: \frac{\partial k_{\mu}}{\partial x_j}\\
&\qquad\qquad\qquad\qquad\qquad\qquad \qquad\qquad\qquad\qquad \qquad\qquad
 - (u''\circ k_{\overline{\mu}})\: \frac{\partial k_{\overline{\mu}}}{\partial x_j}  \: k_{\mu-\overline{\mu}}
\Big\Vert_{\infty}
\\
&=
d\max_j\Big\| (u''\circ k_{\overline{\mu}})\:k_{\mu-\overline{\mu}} \: \frac{\partial k_{\mu-\overline{\mu}}}{\partial x_j}+
R_1[u', k_{\overline{\mu}}]\! \left(k_{\mu}\right)\left.\:
\frac{\partial k_\mu}{\partial x_j} \right\Vert_{\infty}
\\
&\leq
d\max_j\Big\{\left\Vert
u''\right\Vert_\infty \left\|k_{\mu-\overline{\mu}}\right\|_\infty \left\Vert \frac{\partial k_{\mu-\overline{\mu}}}{\partial x_j}\right\Vert_{\infty}
+ 
\left\|R_1[u', k_{\overline{\mu}}]\! \left(k_\mu\right)\right\|_\infty\Big\|
\frac{\partial k_\mu}{\partial x_j} \Big\|_{\infty}\Big\}.
\end{split}
\end{equation*}
Since $k_\mu\in\CCA$, we can use~Proposition~\ref{eq:remainder estimate Holder} for $R_1[u',k_{\overline{\mu}}](k_\mu)$, and estimate the rest in the exactly the same way as in~\eqref{eqn:II.1}
\begin{equation}\label{eq:R_estim_1}
\sup_{x\in\RRD}\Big|R_1[u',k_{\overline{\mu}}(x)](k_\mu(x))\Big|\leq C\: |u''|_\alpha \: \|\mu-\overline{\mu}\|_\ZZZ^{1+\alpha}.
\end{equation}
Additionally, there exists such a constant $C$ that the following estimations hold 
\begin{equation}\label{eq:estim_II.2}
\|u''\circ k_{\overline{\mu}}\|_\infty \leq C \|u''\|_\infty,\;\;\; \underbrace{\|k_{\mu-\overline{\mu}}\|_\infty}_{\textrm{\footnotesize(using~\eqref{estim:bound1})}}\!\!\!\leq C\|\mu-\overline{\mu}\|_\ZZZ\quad \textrm{and}\;\;\; \underbrace{\left\|\frac{\partial k_{\mu-\overline{\mu}}}{\partial x_j}\right\|_\infty}_{\textrm{\footnotesize(using~\eqref{estim:bound11})}}\!\!\! \leq C\|\mu-\overline{\mu}\|_\ZZZ,
\end{equation}
hence one obtains
\begin{equation}\label{eqn:II.2}
\begin{split}
II.2&\leq
C_1 \|u''\|_\infty  \Vert \mu-\overline{\mu}\Vert_\ZZZ^2 + C_2|u''|_\alpha  \: \|\mu-\overline{\mu}\|_\ZZZ^{1+\alpha},
\end{split}
\end{equation}
where constants do not depend on time.
Now, let us take the last component of equation~\eqref{eqn:estim_III_3}
{\cred
\begin{equation*}
\begin{split}
&II.3\leq
d\max_j\left\{\Big| (u''\circ k_{\overline{\mu}})\:k_{\mu-\overline{\mu}} \: \frac{\partial k_{\mu-\overline{\mu}}}{\partial x_j}\Big|_\alpha+
\Big|R_1[u', k_{\overline{\mu}}]\! \left(k_{\mu}\right)\:
\frac{\partial k_\mu}{\partial x_j} \Big|_{\alpha}\right\}
\\
&\leq
d\max_j\left\{|u''\circ k_{\overline{\mu}}|_\alpha\:\|k_{\mu-\overline{\mu}}\|_\infty \: \left\|\frac{\partial k_{\mu-\overline{\mu}}}{\partial x_j}\right\|_\infty
+
\|u''\circ k_{\overline{\mu}}\|_\infty\:|k_{\mu-\overline{\mu}}|_\alpha \: \left\|\frac{\partial k_{\mu-\overline{\mu}}}{\partial x_j}\right\|_\infty
\right.
\\
&\quad +
\|u''\circ k_{\overline{\mu}}\|_\infty\:\|k_{\mu-\overline{\mu}}\|_\infty \left|\frac{\partial k_{\mu-\overline{\mu}}}{\partial x_j}\right|_\alpha\!\!\!
+
\left|R_1[u', k_{\overline{\mu}}]\! \left(k_{\mu}\right)\right|_\alpha
\left\|\frac{\partial k_\mu}{\partial x_j} \right\|_{\infty}
\!\!\!\!+\!\left.
\left\|R_1[u', k_{\overline{\mu}}]\! \left(k_{\mu}\right)\right\|_\infty\:
\left|\frac{\partial k_\mu}{\partial x_j} \right|_{\alpha}\right\}.
\end{split}
\end{equation*}
Using estimations~\eqref{eq:estim_II.2} one obtains
\begin{equation*}
\begin{split}
&II.3\leq\\
&\leq d\max_j\Big\{
|u''\circ k_{\overline{\mu}}|_\alpha \cdot C^2\|\mu-\overline{\mu}\|_\ZZZ^2
+
\|u''\|_\infty\cdot C\|\mu-\overline{\mu}\|_\ZZZ|k_{\mu-\overline{\mu}}|_\alpha 
\\
&
\quad +
\|u''\|_\infty C\|\mu-\overline{\mu}\|_\ZZZ\left|\frac{\partial k_{\mu-\overline{\mu}}}{\partial x_j}\right|_\alpha\!\!+\big|R_1[u',k_{\overline{\mu}}](k_\mu)\big|_\alpha\left\|\frac{\partial k_\mu}{\partial x_j}\right\|_\infty
\!\!+C|u''|_\alpha \|\mu-\overline{\mu}\|_\ZZZ^{1+\alpha}\left|\frac{\partial k_\mu}{\partial x_j}\right|_\alpha\Big\}.
\end{split}
\end{equation*}
}

{\cred
Let us estimate the term $|R_1[u',k_{\overline{\mu}}](k_\mu)|_\alpha$.
In general, if $f\in\CC^{n+1}$, then for some $\xi\in[x_0;x]$ holds
\begin{equation*}\label{taylor1}
R_n[f,x_0](x)=\tfrac{1}{(n+1)!}\Big(f^{(n+1)}(\xi)\Big)(x-x_0)^{n+1}. \end{equation*}
Proof of this fact can be found in~\cite[Theorem
   16, page 139]{Adams:2003}.
In our case \[R_1[u', k_{\overline{\mu}}(x)](k_{\mu}(x))=\frac{1}{2}(u'''\circ \xi(x))
\: \left(k_{\mu-\overline{\mu}}(x)\right)^2,\] where $k_{\overline{\mu}}(x) \leq \xi(x) \leq k_{\mu}(x)$. Since  $u\in\CC^{3+\alpha}$, thus one obtains
\begin{equation}\label{eqn:II.3_3}
\begin{split}
&\Big| R_1[u',k_{\overline{\mu}}](k_\mu)\Big|_{\alpha}= \frac{1}{2}\underbrace{\Big|  (u'''\circ \xi)
\left(k_{\mu-\overline{\mu}}\right)^2 \Big|_{\alpha}}_{\textrm{\footnotesize{(using~\eqref{eqn:alpha_infty})}}}\\
&\leq\tfrac{1}{2}\!\!\underbrace{|u'''\circ \xi|_\alpha}_{\textrm{\footnotesize(using~\eqref{eqn:alpha})}}
\underbrace{\|\left(k_{\mu-\overline{\mu}}\right)^2 \|_{\infty}}_{\textrm{\footnotesize(using~\eqref{estim:bound1})}}\: +\:\:\tfrac{1}{2}\left\|u'''\circ \xi\right\|_\infty \!\!\!\!\!\!\!\!\!\!
\underbrace{\left|\left(k_{\mu-\overline{\mu}}\right)^2 \right|_{\alpha}}_{\textrm{\footnotesize(analogously as in~\eqref{estim:bound3})}}\!\!\!
\leq c\|\mu-\overline{\mu}\|_\ZZZ^2.
\end{split}
\end{equation}
We have
\[|k_\mu|_\alpha\leq \max\{2\|k_\mu\|_\infty,\|\nabla k_\mu\|_\infty\}\leq C_k\|\mu\|_\ZZZ,\qquad \textrm{(see Lemma~\ref{lem:estim_x_K} equations~\eqref{estim:bound1} and~\eqref{estim:bound11})}\]
while
\[\left|\frac{\partial k_\mu}{\partial x_j}\right|_\alpha\leq C_k'\|\mu\|_\ZZZ \qquad \textrm{(see~\eqref{estim:bound3}).}\]
Thus, one obtains
\begin{equation}\label{eqn:IIfinal}
II.3\leq \widehat{C}_1\|\mu-\overline{\mu}\|_\ZZZ^2+\widehat{C}_2\|\mu-\overline{\mu}\|_\ZZZ^{1+\alpha}.
\end{equation}
This concludes Step 1,}
 finishes the proof that $\mu\mapsto u(k_\mu)$ is of class $\CC^1\left(\ZZZ; \CCA(\RRD)\right)$. Notice that in estimate~\eqref{eqn:II.1}, \eqref{eqn:II.2}, and~\eqref{eqn:IIfinal} all constants are independent on time.

{\bf Step 2. }Now we are going to show boundedness of the mapping $\mu\mapsto u(k_\mu)$.
By definition of the norm $\|\cdot\|_{\CCA(\RRD)}$ one has
\begin{align}\label{eq:boun_partial_u}
\|u(k_\mu)\|_{\CCA(\RRD)}=\left\|u(k_\mu)\right\|_\infty+
\left\|\nabla_x\left[u(k_\mu)\right]\right\|_\infty+
\left|\nabla_x\left[u(k_\mu)\right]\right|_\alpha.
\end{align}
Let us first recall that $u\in\CC^{3+\alpha}(\RRD)$ (by assumptions {\bf A1} and {\bf A2}).
For the first term one can observe that
\[
\left\|u\left(\int_\RRD K_u(y,x)\:\dd\mu(y)\right)\right\|_\infty\leq \|u\|_\infty=:  C_1.\]
For the second term on the right-hand side of~\eqref{eq:boun_partial_u} let us observe that according to chain rule one can estimate
\begin{align*}
\left\|\nabla_x\left[u(k_\mu)\right]\right\|_\infty
&\leq 
\left\| u'\!\left(\int_\RRD K_u(y,x)\:\dd\mu(y)\right)\right\|_\infty 
\left\|\nabla_x\int_\RRD K_u(y,x)\:\dd\mu(y)\right\|_\infty\\
&\leq 
\left\| u'\right\|_\infty 
\left\|\int_\RRD \nabla_x K_u(y,x)\:\dd\mu(y)\right\|_\infty\\
&\leq
\big\| u'\big\|_\infty\big\|\nabla_x K\big\|_{\CC^{0+\alpha}\left(\RRD;\:\CCA(\RRD)\right)}\big\|\mu\big\|_\ZZZ \leq C_2\big\|\mu\big\|_\ZZZ.
\end{align*} 
The last term can be estimated by
\begin{align*}
&\big|\nabla_x[u(k_\mu)]\big|_\alpha=\\
=&\sup_{\substack{x_1\neq x_2\\ x_1, x_2 \in \RRD }}\frac{\left|\nabla_x\left[u\left(\int_\RRD K_u(y,x_1)\:\dd\mu(y)\right)\right]-\nabla_x\left[u\left(\int_\RRD K_u(y,x_2)\:\dd\mu(y)\right)\right]\right|}{|x_1-x_2|^\alpha}
\end{align*}
\begin{align*}
\leq &
\sup_{\substack{x_1\neq x_2\\ x_1, x_2 \in \RRD }}\Bigg\{\frac{\big| u'\left(\int_\RRD K_u(y,x_1)\:\dd\mu(y)\right)  \int_\RRD \nabla_xK_u(y,x_1)\:\dd\mu(y)}{|x_1-x_2|^\alpha}\\
&\qquad-\frac{ u'\left(\int_\RRD K_u(y,x_2)\:\dd\mu(y)\right)\big| \int_\RRD \nabla_xK_u(y,x_2)\:\dd\mu(y)}{|x_1-x_2|^\alpha}\Bigg\}\\
= &
\sup_{\substack{x_1\neq x_2\\ x_1, x_2 \in \RRD }}
\Bigg\{\frac{\big|\big[ u'\left(\int_\RRD K_u(y,x_1)\:\dd\mu(y)\right)- u'\left(\int_\RRD K_u(y,x_2)\:\dd\mu(y)\right)\big] \int_\RRD \nabla_xK_u(y,x_1)\:\dd\mu(y)}{|x_1-x_2|^\alpha}\\
&\qquad-\frac{\big[\int_\RRD\nabla_xK_u(y,x_2)-\nabla_xK_u(y,x_1)\:\dd\mu(y)\big] u'\left(\int_\RRD K_u(y,x_2)\:\dd\mu(y)\right)\big|}{|x_1-x_2|^\alpha}\Bigg\}
\\
\leq &
\sup_{\substack{x_1\neq x_2\\ x_1, x_2 \in \RRD }}
\Bigg\{\frac{\| u''\|_\infty\left|\int_\RRD K(y,x_1)-K(y,x_2) \:\dd\mu(y)\right|\|\nabla_x K\|_{\CC^{0+\alpha}(\RRD;\:\CCA(\RRD))}\|\mu\|_\ZZZ}{|x_1-x_2|^\alpha}\\
&\qquad
+\frac{\|\nabla_x\big(K(\cdot,x_1)-K(\cdot,x_2)\big)\|_{\CCA(\RRD)}\|\mu\|_\ZZZ\| u'\|_\infty}{|x_1-x_2|^\alpha}\Bigg\}\\
\leq & 
\sup_{\substack{x_1\neq x_2\\ x_1, x_2 \in \RRD }}
\Bigg\{\frac{\| u''\|_\infty\|K(\cdot,x_1)-K(\cdot,x_2)\|_{\CCA(\RRD)}\|\mu\|_\ZZZ\|\nabla_x K\|_{\CC^{0+\alpha}(\RRD;\:\CCA(\RRD))}\|\mu\|_\ZZZ}{|x_1-x_2|^\alpha}\\
&\qquad
+\frac{\|\nabla_x\big(K(\cdot,x_1)-K(\cdot,x_2)\big)\|_{\CCA(\RRD)}\|\mu\|_\ZZZ\| u'\|_\infty}{|x_1-x_2|^\alpha}\Bigg\}\leq  C_3\Big(\big\|\mu \big\|^2_\ZZZ+\big\|\mu \big\|_\ZZZ\Big). 
\end{align*} 
Combining estimates one has
\[\|u(k_\mu)\|_{\CCA(\RRD)}\leq C_1+C_2\|\mu\|_\ZZZ+C_3(\|\mu\|_\ZZZ^2+\|\mu\|_\ZZZ)\leq C_{u,K_u}(1+\|\mu\|_\ZZZ+\|\mu\|_\ZZZ^2).\]
{\cred
{\bf Step 3.} Now we are going to show the boundedness of the mapping $\mu\mapsto \partial u(k_\mu)$. For $\mu,\overline{\mu}\in\ZZZ$ one has
\begin{align*}
\|\partial u(k_{\mu})\:\overline{\mu}\|_{\CCA(\RRD)}&=\|(u'\circ k_\mu)\:k_{\overline{\mu}}\|_{\CCA(\RRD)}
\leq\underbrace{\|u'\circ k_\mu\|_{\CCA(\RRD)}\|k_{\overline{\mu}}\|_{\CCA(\RRD)}}_{\textrm{\footnotesize (according to~\eqref{eqn:CCA_fg})}}\\
&\leq\|u'\circ k_\mu\|_{\CCA(\RRD)}\underbrace{\|K\|_{\CCA( \RRD;\:\CCA(\RRD))} \|\overline{\mu}\|_\ZZZ}_{\textrm{\footnotesize (according to~\eqref{eq:k_mu_in_CCA})}}
\end{align*}
and hence
\[\|\partial u(k_\mu)\|_{\mathcal{L}(\ZZZ;\:\CCA(\RRD))}\leq\|K\|_{\CCA( \RRD;\:\CCA(\RRD))} \|u'\circ k_\mu\|_{\CCA(\RRD)}.\]
The last factor can be estimated as follows
\begin{align*}
\|u'\circ k_\mu\|_{\CCA(\RRD)}&=\|u'\circ k_\mu\|_\infty+\|\nabla(u'\circ k_\mu)\|_\infty+|\nabla(u'\circ k_\mu)|_\alpha\\
&\leq \|u'\|_\infty+\|(u''\circ k_\mu) \nabla k_\mu\|_\infty+\underbrace{|(u''\circ k_\mu)\: \nabla k_\mu|_\alpha}_{\textrm{\footnotesize(use~\eqref{eqn:alpha_infty})}}
\end{align*}
\begin{align*}
&\leq\|u'\|_\infty+\|u''\|_\infty\|\nabla k_\mu\|_\infty+\|u''\circ k_\mu\|_\infty |\nabla k_\mu|_\alpha+\underbrace{|u''\circ k_\mu|_\alpha}_{\textrm{\footnotesize(use~\eqref{eqn:alpha})}} \|\nabla k_\mu\|_\infty
\\
&\leq
\|u'\|_\infty+\|u''\|_\infty\underbrace{\|\nabla k_\mu\|_\infty}_{\textrm{\footnotesize(use~\eqref{estim:bound11})}}+\|u''\|_\infty \underbrace{|\nabla k_\mu|_\alpha}_{\textrm{\footnotesize(use~\eqref{estim:bound3})}}+|u''|_\alpha \underbrace{\| \nabla k_\mu\|_\infty^\alpha}_{\textrm{\footnotesize(use~\eqref{estim:bound11})}} \underbrace{\|\nabla k_\mu\|_\infty}_{\textrm{\footnotesize(use~\eqref{estim:bound11})}}
\\
&\leq \|u'\|_\infty+\|u''\|_\infty\sup_{x\in\RRD}\|\nabla_xK(\cdot,x)\|_{\CCA(\RRD)}\|\mu\|_\ZZZ
\\
&\qquad
+
\|u''\|_\infty\sup_{\substack{x_1 \neq x_2\\ x_1, x_2 \in \RRD}}\left(\frac{\Vert  \nabla_x \left(K(\cdot,x_1) - K(\cdot, x_2)\right)\Vert_{\CCA(\RRD)}}{|x_1 - x_2|^{\alpha}}\right)  \Vert \mu \Vert_\ZZZ\\
&\qquad
+ |u''|_\alpha\sup_{x\in\RRD}\|\nabla_xK(\cdot,x)\|_{\CCA(\RRD)}^\alpha\|\mu\|_\ZZZ^\alpha
\sup_{x\in \RRD}\Vert \nabla_xK(\cdot, x)\Vert_{\CCA(\RRD)} \Vert \mu \Vert_\ZZZ\\
&\leq C'_{\mu,K_u}(1+\|\mu\|_\ZZZ+\|\mu\|_\ZZZ^{1+\alpha}).
\end{align*}
This finishes the proof of Lemma~\ref{lem:nem_char}.
}

\section{Proof of Theorem~\ref{th:main}}\label{apdx:lin}
 Before giving proofs of above theorems, let us notice a few observations, which will be used further.
The characteristic system for the perturbed problem~\eqref{def:prob_init2_lin} has the form
\begin{equation}\label{def:diff_pert}
\left\{ \begin{array}{ll}
{\dot X}_{h}(t,x)&=\left(b_0+h b_1\right)\left(t, X_{h}(t,x)\right)\\
X_{h}(t_0, x)&=x \in \RRD.
\end{array} \right.
\end{equation}
As before: if $\left(t \mapsto b_i(t, \cdot)\right)\in \CC_b\left([ 0, +\infty); \CCA(\RRD)\right)$ for $i=0,1$, then $x \mapsto X_{h}(t, x)$ is a~diffeomorphism. To underline the dependence of the flow on parameter $h$, we will use notation $X(t, x; h)$ instead of $X_{h}(t, x)$. 
\begin{proposition}\label{claim:esti_X}
Let $\left(t \mapsto b_i(t, \cdot)\right) \in \CC_b([ 0, +\infty); \CCA(\RRD))$ for $i=0,1$. Then for any $(t, x)\in [0,+ \infty) \times \RRD$ the mapping $\left(h \mapsto X (t, x; h)\right)$ is of a class $\CCA\big((-\frac{1}{2}, \frac{1}{2})\big)$.
Moreover, for any $t$ the following estimation holds
\begin{equation}\label{eq:estim_X_time_X}
\left|\partial_h X( t,x;h)\right|\leq t\Vert b_1 \Vert_{\infty} \exp\left(C_h\:t\right),\end{equation}
where $C_h:=\Vert \nabla_xb_0\Vert_{\infty}+ |h|\Vert \nabla_x b_1 \Vert_{\infty}$.
\end{proposition}
The proof of above proposition can be found in Appendix~\ref{apdx} on page~\pageref{apdx}.
\begin{proof}[Proof of Theorem~\ref{th:main}] 
In the following proof we will point out a~dependence of constants on time, to prepare the background for the proof of Theorem~\ref{tw:ciagla_pocho} (see page~\pageref{apdx:ciagla_pocho}).
Knowing that the space $\ZZZ$ is complete, it is enough to show that a proper sequence of quotients is a Cauchy sequence. Let as consider a weak solutions to~\eqref{def:prob_init2_lin} for different values of parameter $h$. By $\nu_t^0$ let us denote the solution for $h=0$ and by $\nu_t^{\lambda_1}$, $\nu_t^{\lambda_2}$ the solution for parameters $\lambda_1\neq \lambda_2$ and such that $\lambda_{1}\neq 0$ and $\lambda_{2}\neq 0$. Solutions $t\mapsto\nu_t^0$, $t\mapsto\nu_t^{\lambda_1}$, $t\mapsto\nu_t^{\lambda_2}$ are unique; see representation formula~\eqref{repr:form}.

Let us notice that for any fixed $h,\lambda\in\mathbb{R}$ and any $t\in[0,+\infty)$ quotient $\frac{\nu_t^{h+\lambda}-\nu_t^h}{ \lambda}$ is an element of space $\mathcal{M}(\RRD) \subseteq \ZZZ$.
First we show differentiability at $h=0$. Differentiability at other $h$ follows from this result; see~\eqref{eq:lin_diff_h_non_0} on page~\pageref{eq:lin_diff_h_non_0}.
 
For $h=0$ it suffices to show that value
 \[I_{{\lambda_1},{\lambda_2}} :=
\left\Vert  \frac{\nu_t^{\lambda_1}-\nu_t^0}{\lambda_1}\right. - \left.\frac{\nu_t^{\lambda_2}-\nu_t^0}{\lambda_2}\right\Vert_{\ZZZ}\]
can be arbitrary small --  when $\lambda_1$ and $\lambda_2$ are sufficiently close to 0. This condition is equivalent to such that: for all $\lambda_n\to 0$, sequence of quotients $\left\{\frac{\nu_t^{\lambda_n}-\nu_t^0}{\lambda_n}\right\}_{n\in\NN}$ is a Cauchy sequence in $\ZZZ$. Hence, converges to a limit that is the same for each sequence $\{\lambda_n\}_{n\in \NN}$ such that $\lambda_n\to 0$. Let us consider
\begin{equation}\label{eqn:final_Cauchy}
\begin{split}
I_{{\lambda_1},{\lambda_2}}& =
\sup_{\Vert \varphi \Vert_{\CCA} \leq 1} \left|\int_{\RRD} \varphi \;\dd\left( \frac{\nu_t^{\lambda_1}-\nu_t^0}{\lambda_1} - \frac{\nu_t^{\lambda_2}-\nu_t^0}{\lambda_2} \right)(x)\right|
\\
&=
\sup_{\Vert \varphi \Vert_{\CCA} \leq 1} \left| \int_{\RRD} \varphi \frac{\dd \nu_t^{\lambda_1}(x)}{\lambda_1} - \int_{\RRD} \varphi \frac{\dd \nu_t^0(x)}{\lambda_1} - \int_{\RRD} \varphi \frac{\dd \nu_t^{\lambda_2}(x)}{\lambda_2} + \int_{\RRD} \varphi \frac{\dd \nu_t^0(x)}{\lambda_2}\right|.
\end{split}
\end{equation}
First we use representation formula (Lemma~\ref{lem:repr_form}, page~\pageref{lem:repr_form}) and the fact that $x \mapsto X(t, x; h)$ is a~diffeomorphism
\begin{equation*}
\begin{split}
&I_{{\lambda_1},{\lambda_2}}=\\
&=\sup_{\Vert \varphi \Vert_{\CCA}\leq 1}
\Bigg|
\int_{\RRD} \varphi(X(t,x; \lambda_1)) \exp\left(\int_0^t w_0(s, X(s, x; \lambda_1))+\lambda_1w_1(s, X(s, x; \lambda_1))\dd s\right) \frac{\dd \nu_0(x)}{\lambda_1}\\
&\quad-
\int_{\RRD} \varphi(X(t,x; 0))\exp\left(\int_0^t w_0(s, X(s, x; 0))\dd s\right)\frac{\dd \nu_0(x)}{\lambda_1}\\
&\quad-
\int_{\RRD} \varphi(X(t,x;\lambda_2))\exp\left(\int_0^t w_0(s, X(s, x; \lambda_2))+\lambda_2w_1(s, X(s, x; \lambda_2))\dd s\right)\frac{\dd \nu_0(x)}{\lambda_2}\\
&\quad+
\int_{\RRD} \varphi(X(t, x; 0))\exp\left(\int_0^t w_0(s, X(s, x; 0))\dd s\right)\frac{\dd \nu_0(x)}{\lambda_2}
\Bigg|.
\end{split}
\end{equation*}
Let us introduce the following notation for convenience \begin{equation}\label{eq:w_defi_linear}
 \overline{w}(s, x; h) :=w_0(s, X(s, x; h))+hw_1(s, X(s, x; h))
\end{equation}
and continue
\begin{equation*}
\begin{split}
&I_{{\lambda_1},{\lambda_2}}=\\
&=\sup_{\Vert \varphi \Vert_{\CCA}\leq 1}
\Bigg|
\int_{\RRD} \varphi(X(t,x; \lambda_1)) e^{\int_0^t \overline{w}(s, x; \lambda_1)\dd s} \frac{\dd \nu_0(x)}{\lambda_1} -
\int_{\RRD} \varphi(X(t,x; 0))e^{\int_0^t \overline{w}(s, x; 0)\dd s}\frac{\dd \nu_0(x)}{\lambda_1}\\
& \qquad \qquad\ -
\int_{\RRD} \varphi(X(t,x;\lambda_2))e^{\int_0^t \overline{w}(s, x; \lambda_2)\dd s}\frac{\dd \nu_0(x)}{\lambda_2}
+
\int_{\RRD} \varphi(X(t, x; 0))e^{\int_0^t \overline{w}(s, x; 0)\dd s}\frac{\dd \nu_0(x)}{\lambda_2}
\Bigg|
\end{split}
\end{equation*}
\begin{equation}\label{eq:estim_jednos}
\begin{split}
&= \sup_{\Vert \varphi \Vert_{\CCA}\leq 1} \underbrace{\Bigg|\int_{\RRD} \Big(\varphi(X(t,x; \lambda_1))-\varphi(X(t,x; 0))\Big)e^{\int_0^t \overline{w}(s, x; \lambda_1)\dd s} \frac{\dd \nu_0(x)}{\lambda_1}}_{ I_{\lambda_1}^{(1)}}\\
&\qquad \qquad\ - \underbrace{\int_{\RRD} \Big(\varphi(X(t,x; \lambda_2))-\varphi(X(t,x; 0))\Big)e^{\int_0^t \overline{w}(s, x; \lambda_2)\dd s} \frac{\dd \nu_0(x)}{\lambda_2}}_{ I_{\lambda_2}^{(1)}}\\
&\qquad \qquad\ - \underbrace{\int_{\RRD} \Big( e^{\int_0^t \overline{w}(s, x; 0)\dd s} -e^{\int_0^t \overline{w}(s, x; \lambda_1)\dd s}\Big)\varphi(X(t,x; 0)) \frac{\dd \nu_0(x)}{\lambda_1}}_{ I_{\lambda_1}^{(2)}}
\\
&\qquad \qquad\ + \underbrace{\int_{\RRD} \left( e^{\int_0^t \overline{w}(s, x; 0)\dd s} -e^{\int_0^t \overline{w}(s, x; \lambda_2)\dd s}\right)\varphi(X(t,x; 0)) \frac{\dd \nu_0(x)}{\lambda_2}\Bigg|}_{I_{\lambda_2}^{(2)}}\\
& \leq \sup_{\Vert \varphi \Vert_{\CCA}\leq 1}\left( | I_{\lambda_1}^{(1)}- I_{\lambda_2}^{(1)}|\right)+ \sup_{\Vert \varphi \Vert_{\CCA}\leq 1}\left(|I_{\lambda_1}^{(2)}- I_{\lambda_2}^{(2)}| \right).
\end{split}
\end{equation}
\noindent Let us consider $| I_{\lambda_1}^{(1)}- I_{\lambda_2}^{(1)}|$ and $|I_{\lambda_1}^{(2)}- I_{\lambda_2}^{(2)}|$ separately.
In $I_{\lambda_1}^{(2)}- I_{\lambda_2}^{(2)}$ expand $e^{\int_0^t \overline{w}(s, x; \lambda_1)\dd s}$ and $e^{\int_0^t \overline{w}(s, x; \lambda_2)\dd s}$ into Taylor series around $h=0$
\begin{equation}\label{eqn:holder_need}
\begin{split}
&| I_{\lambda_1}^{(2)}-I_{\lambda_2}^{(2)}| = \\
&=\Bigg|\int_{\RRD} \varphi(X(t,x; 0))\Bigg[e^{\int_0^t \overline{w}(s, x; 0)\dd s}-e^{\int_0^t \overline{w}(s, x; 0)\dd s} -\lambda_1\Eehz \partial_h\Big(\int_0^t \overline{w}(s, x;h) \dd s\Big)\Big|_{h=0}\\
&\qquad - \mathcal{O}\left(\left|\int_0^t\overline{w}(s,x,\lambda_1)\dd s-\int_0^t\overline{w}(s,x,0)\dd s\right|^{1+\alpha}\right)\Bigg]\frac{\dd \nu_0(x)}{\lambda_1}
\\
& -\int_{\RRD} \varphi(X(t,x; 0))\Bigg[\Eehz-\Eehz
 - \lambda_2\Eehz \partial_h\Big(\int_0^t \overline{w}(s, x;h) \dd s\Big)\Big|_{h=0}\\
 &\qquad - \mathcal{O}\left(\left|\int_0^t\overline{w}(s,x,\lambda_2)\dd s-\int_0^t\overline{w}(s,x,0)\dd s\right|^{1+\alpha}\right)\Bigg]\frac{\dd \nu_0(x)}{\lambda_2}\Bigg|.
\end{split}
\end{equation}
For any $h$ the following holds
\begin{equation*}
\begin{split}
&\mathcal{O}\left(\left|\int_0^t\overline{w}(s,x,h)\dd s-\int_0^t\overline{w}(s,x,0)\dd s\right|^{1+\alpha}\right)=\\
&=\mathcal{O}\left(\left|\int_0^t w_0(s,X(s,x;h))+hw_1(s,X(s,x;h))- w_0(s,X(s,x;0))\dd s\right|^{1+\alpha}\right).
\end{split}
\end{equation*}
Expanding $w_0(s,X(s,x;h))$ into Taylor series at $X(s,h;0)$, and knowing that $w(t,\cdot)\in\CCA(\RRD)$ one can continue estimation
\begin{equation*}
\begin{split}
& \mathcal{O}\left(\left|\int_0^t w_0(s,X(s,x;0))
		+h\nabla_x w_0(s,X(s,x;0))\:\partial_hX(s,x;h)|_{h=0}\right.\right.\\
	&\qquad+\mathcal{O}(|X(s,x;h)-X(s,x;0)|^{1+\alpha})+hw_1(s,X(s,x;h))- w_0(s,X(s,x;0))\dd s\Big|^{1+\alpha}\Big)\\
& =\mathcal{O}\left(\left|\int_0^t h\nabla_x w_0(s,X(s,x;0))\:
\partial_hX(s,x;h)|_{h=0}\right.\right.\\
&\qquad +\mathcal{O}(|X(s,x;h)-X(s,x;0)|^{1+\alpha})+hw_1(s,X(s,x;h))\dd s\Big|^{1+\alpha}\Big)\\
&\leq
\mathcal{O}\Bigg(|h|^{1+\alpha}\underbrace{\big(\int_0^t\|\nabla_x w_0\|_\infty \:s\Vert b_1 \Vert_{\infty} e^{C s}+\|w_1\|_\infty\dd s\big)}_{I_{\mathcal{O}}^{(1)}}\\
&\qquad\qquad\qquad\qquad\qquad\qquad\qquad+\underbrace{\int_0^t\mathcal{O}(|X(s,x;h)-X(s,x;0)|^{1+\alpha})\dd s}_{I_{\mathcal{O}}^{(2)}}\Bigg).
	\end{split}
\end{equation*}
The last inequality follows Proposition~\ref{claim:esti_X} and assumption of Theorem~\ref{th:main}, that $w_0(s,X(s,x;0))$ and $w_1(s,X(s,x;h))$ are globally bounded in time.
Constant $C$ in the above inequality is independent on time and it is defined in Proposition~\ref{claim:esti_X}.
Components $I_{\mathcal{O}}^{(1)}$ and $I_{\mathcal{O}}^{(2)}$ can be estimated as follows
\begin{align*}
I_{\mathcal{O}}^{(1)}&=\|\nabla_x w_0\|_\infty\|b_1\|_\infty\left(\frac{t}{C}\exp(Ct)-\frac{1}{C^2}\exp(Ct)+\frac{1}{C^2}\right)+t\|w_1\|_\infty,\\
I_{\mathcal{O}}^{(2)}&=\int_0^t\mathcal{O}\Big(\Big|X(s;x;0)+h\partial_hX(s;x;h)|_{h=0}+\mathcal{O}(|h|^{1+\alpha})-X(s;x;0)\Big|^{1+\alpha}\Big)\dd s\\
&\leq \int_0^t\mathcal{O}\Big(\Big|h\partial_hX(s;x;h)|_{h=0}+\mathcal{O}(|h|^{1+\alpha})\Big|^{1+\alpha}\Big)\dd s\\
&\leq\int_0^t \mathcal{O}\big((s\Vert b_1 \Vert_{\infty} e^{C_h s})|h|^{1+\alpha}\big)\dd s + t\mathcal{O}(|h|^{1+\alpha})\\
&\leq \|b_1\|_\infty\left(\frac{t}{C}\exp(C_ht)-\frac{1}{C^2}\exp(C_ht)+\frac{1}{C^2}\right)\mathcal{O}(|h|^{1+\alpha})+t\mathcal{O}(|h|^{1+\alpha}).
\end{align*}
Let us denote \begin{equation}\label{eq:holder_XXX}
\overline{C}(t):=\big(\|\nabla_x w_0\|_\infty+1\big)\|b_1\|_\infty\left(\frac{t}{C}\exp(Ct)-\frac{1}{C^2}\exp(Ct)+\frac{1}{C^2}\right),\qquad C_1:=\|w_1\|_\infty.
\end{equation}
Using above notation we have \begin{align*}\label{eq:star_1}
&\mathcal{O}\left(\left|\int_0^t\overline{w}(s,x,h)\dd s-\int_0^t\overline{w}(s,x,0)\dd s\right|^{1+\alpha}\right)\leq \big(2\overline{C}(t)+C_1t\big)\mathcal{O}(|h|^{1+\alpha}).
\end{align*}
Now, we can go back to estimation~\eqref{eqn:holder_need}
\begin{small}
\begin{equation}\label{eq:I2_final}
\begin{split}
&| I_{\lambda_1}^{(2)}-I_{\lambda_2}^{(2)}| \leq\\
&\leq\;\Bigg|
\int_{\RRD} \varphi(X(t,x; 0))\Bigg[
\Eehz \partial_h\left(\int_0^t \overline{w}(s, x;h) \dd s\right)\Big|_{h=0}-
\Eehz \partial_h \left(\int_0^t \overline{w}(s, x;h) \dd s\right)\Big|_{h=0}\\
&\qquad\quad +\big(2\overline{C}(t)+C_1t\big)\left(\frac{\mathcal{O}(|\lambda_1|^{1+\alpha})}{\lambda_1}+ \frac{\mathcal{O}(|\lambda_2|^{1+\alpha})}{\lambda_2}\right)\Bigg]\dd \nu_0(x)\Bigg|
\\
&\leq\; (2\overline{C}(t)+C_1t\big)\Bigg|\int_{\RRD}\varphi(X(t,x; 0)) \left(\frac{\mathcal{O}(|\lambda_1|^{1+\alpha})}{\lambda_1}+ \frac{\mathcal{O}(|\lambda_2|^{1+\alpha})}{\lambda_2}\right)\dd \nu_0(x)\Bigg|\\
& \leq\;  (2\overline{C}(t)+C_1t\big)\left|\int_{\RRD} \left(\frac{\mathcal{O}(|\lambda_1|^{1+\alpha})}{\lambda_1}+ \frac{\mathcal{O}(|\lambda_2|^{1+\alpha})}{\lambda_2}\right)\dd \nu_0(x)\right|,
\end{split}
\end{equation}
\end{small}

\noindent where $\overline{C}(t)$ and $C_1$ are given by~\eqref{eq:holder_XXX}.
We have shown, that for any fixed moment of time $t$, the term
 $\sup_{\Vert \varphi \Vert_{\CCA}\leq 1}\left(|I_{\lambda_1}^{(2)}- I_{\lambda_2}^{(2)}| \right)$ can be arbitrary small. Above can be seen why do we need $w_0$ and $w_1$ to be of a class $\CCA$ with respect to~$x$. 

We now take into consideration $| I_{\lambda_1}^{(1)}- I_{\lambda_2}^{(1)}|$. Because $\varphi\in\CCA(\RRD)$, one has
\begin{equation}\label{taylor_ex}
\varphi(x)=\varphi(x_0)+\nabla_x \varphi(x_0)(x-x_0) + R_1[\varphi;x_0](x),
\end{equation}
with $|R_1[\varphi;x_0](x)| \leq C_2|\nabla_x\varphi|_{\alpha} | x-x_0|^{1+\alpha}$, where $|\nabla_x\varphi|_\alpha$ is an $\alpha$-H\"older constant. Since $\varphi\in\CCA(\RRD)$, the function $\varphi(X(t,x;h))$ is bounded globally in time, hence constant $C_2$ is independent of time.

Thus, expand $\varphi(X(t,x; \lambda_1))
$ and $\varphi(X(t,x; \lambda_2))$ into Taylor series around $X(t,x; 0)$
\begin{equation*}
\begin{split}
|I_{\lambda_1}^{(1)}-I_{\lambda_2}^{(1)}|&\leq
\Big|\int_{\RRD} \left[ \varphi(X(t,x; 0))
+
\nabla_x \varphi(X(t, x; h))|_{h=0} \cdot (X(t, x; \lambda_1) - X(t, x; 0))\right. \\
& \qquad   + \left. \mathcal{O}\left(\left|X(t, x; \lambda_1)-X(t, x; 0)\right|^{1+\alpha}\right)
- \varphi(X(t, x; 0))\right] e^{\int_0^t \overline{w}(s, x; \lambda_1)\dd s}\frac{\dd \nu_0(x)}{\lambda_1}
\\
& \ -\, 
\int_{\RRD}  \left[ \varphi(X(t, x;0))
+
\nabla_x\varphi(X(t, x; h))|_{h=0}(X(t, x; \lambda_2) - X(t, x; 0))\right.\\
& \qquad  + \left. \mathcal{O}\left(\left|X(t, x; \lambda_2)-X(t, x; 0)\right|^{1+\alpha}\right)
- \varphi(X(t, x; 0))\right]e^{\int_0^t \overline{w}(s, x; \lambda_2)\dd s} \frac{\dd \nu_0(x)}{\lambda_2}\Big|.
\end{split}
\end{equation*}
Expanding $X(t, x; \lambda_2)$ and $X(t, x; \lambda_1)$ around $h=0$, by Proposition~\ref{claim:esti_X}  one obtains
\begin{equation}\label{eq:kwadrat}
\begin{split}
| I_{\lambda_1}^{(1)}- I_{\lambda_2}^{(1)}|&=\Bigg|
\int_{\RRD} \Big[
\nabla_x\varphi(X(t,x; h))|_{h=0}\Big(\lambda_1\partial_h X(t, x;h)\big|_{h=0} + \mathcal{O}(|\lambda_1|^{1+\alpha})\Big)\\
&\qquad \qquad
 +
 \mathcal{O}\Big(\left|X(t,x; \lambda_1)-X(t,x; 0)\right|^{1+\alpha}\Big)\Big] e^{\int_0^t \overline{w}(s, x; \lambda_1)\dd s}\frac{\dd \nu_0(x)}{\lambda_1}\\
&\quad-\int_{\RRD} \Big[
\nabla_x\varphi(X(t,x; h))|_{h=0}\Big(\lambda_2\partial_h X(t, x;h)\big|_{h=0} + \mathcal{O}(|\lambda_2|^{1+\alpha})\Big)\\
&\qquad \qquad
+
 \mathcal{O}\Big(\left|X(t,x; \lambda_2)-X(t,x; 0)\right|^{1+\alpha}\Big)\Big] e^{\int_0^t \overline{w}(s, x; \lambda_2)\dd s}\frac{\dd \nu_0(x)}{\lambda_2}\Bigg|.
\end{split}
\end{equation}
The remainder $\mathcal{O} \Big(\left|X(t, x; h)-X(t, x; 0)\right|^{1+\alpha}\Big)$ was estimated at $I_{\mathcal{O}}^{(2)}$; see page~\pageref{eq:holder_XXX}. For each $h\in (-\frac{1}{2}, \frac{1}{2})$ and any fixed $t$ one has 
\[
\mathcal{O} \Big(\left|X(t, x; h)-X(t, x; 0)\right|^{1+\alpha}\Big) \leq (t\Vert b_1 \Vert_{\infty} e^{C_h\: t})\mathcal{O}(|h|^{1+\alpha}),
\] 
where $C_h$ is independent of time; see again Proposition~\ref{claim:esti_X}. The expression for $C_h$ shows that it is bounded for $h\in(-\tfrac{1}{2},\tfrac{1}{2})$. Putting $C:=\sup_{h\in(-\tfrac{1}{2},\tfrac{1}{2})}C_h$ 
we obtain
\begin{equation*}\label{str_28}
\begin{split}
&| I_{\lambda_1}^{(1)}- I_{\lambda_2}^{(1)}|\leq\Big|\int_{\RRD} \left[
\nabla_x \varphi(X(t,x; h))|_{h=0}\Big(\partial_h X(t, x;h)\big|_{h=0}+\mathcal{O}(|\lambda_1|^\alpha)\Big) \right.\\
&\qquad\qquad\qquad \qquad\qquad  + (t\Vert b_1 \Vert_{\infty} e^{C t})\mathcal{O}(|\lambda_1|^\alpha)
\Big]e^{\int_0^t \overline{w}(s, x; \lambda_1)\dd s} \dd \nu_0(x)\\
&\qquad \qquad \qquad - \int_{\RRD} \left[
\nabla_x \varphi(X(t,x; h))|_{h=0}\Big(\partial_h X(t, x;h)\big|_{h=0}+\mathcal{O}(|\lambda_2|^\alpha)\Big)\right. \\
&\qquad \qquad \qquad \qquad \qquad +(t\Vert b_1 \Vert_{\infty} e^{C t})\mathcal{O}(|\lambda_2|^\alpha)
\Big]e^{\int_0^t \overline{w}(s, x; \lambda_2)\dd s} \dd \nu_0(x)\Big|.
\end{split}
\end{equation*}
We consider function $\varphi\in\CCA(\RRD)$ with $\|\varphi\|_{\CCA}\leq 1$. Hence we can further estimate $\nabla_x\varphi(X(t,y;h))|_{h=0}\leq 1$. This yields
\begin{equation}\label{eq:estim_WWW}
\begin{split}
| I_{\lambda_1}^{(1)}&- I_{\lambda_2}^{(1)}| \leq \Big|\int_{\RRD}\Big[\underbrace{\partial_h X(t, x;h)\big|_{h=0}}_{{\bf L1}}\underbrace{\left( e^{\int_0^t \overline{w}(s, x; \lambda_1)\dd s}-e^{\int_0^t \overline{w}(s, x; \lambda_2)\dd s} \right)}_{{\bf L2}} \\
& + (1+t\Vert b_1 \Vert_{\infty} e^{C t})\underbrace{\left( \mathcal{O}(|\lambda_1|^\alpha)e^{\int_0^t \overline{w}(s, x; \lambda_1)\dd s} -\mathcal{O}(|\lambda_2|^\alpha)e^{\int_0^t \overline{w}(s, x; \lambda_2)\dd s} \right)}_{{\bf L3}}
\Big] \dd \nu_0(x)\Big|.
\end{split}
\end{equation}
To summarize estimations:
\begin{description}
\item[L1] $\partial_h X(t, x;h)\big|_{h=0}<\infty$; see Proposition~\ref{claim:esti_X};
\item[L2] $\left( e^{\int_0^t \overline{w}(s, x; \lambda_1)\dd s}-e^{\int_0^t \overline{w}(s, x; \lambda_2)\dd s} \right)\xrightarrow{\lambda_1,\lambda_2\to 0} 0$ argumentation is similar as in $|I_{\lambda_1}^{(2)}- I_{\lambda_2}^{(2)}|$. The function $e^{\int_0^t \overline{w}(s, x; \lambda_1)\dd s}$ and $e^{\int_0^t \overline{w}(s, x; \lambda_2)\dd s}$ needs to be expanded into Taylor series at $h=0$; see~\eqref{eqn:holder_need}.
\item[L3] $\left( \mathcal{O}(|\lambda_1|^\alpha)e^{\int_0^t \overline{w}(s, x; \lambda_1)\dd s} -\mathcal{O}(|\lambda_2|^\alpha)e^{\int_0^t \overline{w}(s, x; \lambda_2)\dd s} \right)\xrightarrow{\lambda_1,\lambda_2\to 0} 0$.
\end{description}
By assumptions of Theorem~\ref{th:main} we have $\nu_0\in\MM(\RRD)$, what means that $\nu_0(\RRD)<\infty$.

Thus for any fixed $t$ values $\sup_{\Vert \varphi \Vert_{\CCA}\leq 1}\left(|I_{\lambda_1}^{(1)}- I_{\lambda_2}^{(1)}| \right)$ and $\sup_{\Vert \varphi \Vert_{\CCA}\leq 1}\left(|I_{\lambda_1}^{(2)}- I_{\lambda_2}^{(2)}| \right)$ can be arbitrary close to $0$.
Therefore we have shown that $\left\{\frac{\nu_t^{h+\lambda_n}-\nu_t^h}{\lambda_n}\right\}_{n\in\NN}$ is a Cauchy sequence for every $\lambda_n\to 0$ in $\ZZZ$ for $h=0$, with the same limit. Hence~$\nu_t^h$ which is the solution to~\eqref{def:prob_init2_lin}
is differentiable with respect $h$ at $h=0$. 

The same argumentation works for $h\neq 0$. Let us consider a sequence $\left\{\frac{\nu_t^{h+\lambda_n}-\nu_t^h}{\lambda_n}\right\}_{n\in\NN}$, where $\lambda_n \to 0$ and $h\neq 0$. 
Let $\nu_t^{h+\lambda_n}$ be a solution to~\eqref{def:prob_init2_lin} with coefficients
\begin{align}\label{eq:lin_diff_h_non_0}
b^{h+\lambda_n}&=b_0+hb_1+\lambda_nb_1 =: \widehat{b}_0 +\lambda_n b_1,\\
w^{h+\lambda_n}&=w_0+hw_1+\lambda_nw_1 =: \widehat{w}_0 +\lambda_n w_1
\end{align}
and with initial condition $\nu_0^{h+\lambda_n}=\nu_0$. Then $\nu_t^{h+\lambda_n}$ is equal (by Lemma~\ref{lem:repr_form}) to the solution $\overline{\nu}_t^{\lambda_n}$ with velocity field $ \widehat{b}_0+\lambda_n b_1$  and scalar function $ \widehat{w}_0+\lambda_n w_1$  and initial condition $\nu_0$. 
A similar statement holds for the $\nu_t^h$ and the solution $\overline{\nu}_t^0$ with velocity field~$\widehat{b}_0$ and scalar function $\widehat{w}_0$. Thus
\[\frac{\nu_t^{h+\lambda_n}-\nu_t^h}{\lambda_n}=\frac{\overline{\nu}_t^{\lambda_n}- \overline{\nu}^0_t}{\lambda_n}\]
and the latter sequence converges in $\ZZZ$ as $h\to\infty$, by the first part of the proof.
\end{proof}

\section{Proof of Theorem~\ref{tw:ciagla_pocho}}\label{apdx:ciagla_pocho}
\begin{proof}[Proof of Theorem~\ref{tw:ciagla_pocho}]
Let us denote the mapping $t\mapsto\nu_t$ by $\nu_\bullet$. The goal is to show that the mapping $(h\mapsto \nu_\bullet^h)$ is of a class \[\CC^1((-\tfrac{1}{2},\tfrac{1}{2});\;\CC_{\widehat{\omega}}([0,\infty);\ZZZ)),\] where $\widehat{\omega}$ is of order $\mathcal{O}\big(|\tfrac{1}{t}\exp(-Ct)|\big)$.

By Theorem~\ref{th:main} we know that for any fixed $h\in(-\tfrac{1}{2},\tfrac{1}{2})$ the limit $\lim_{\lambda\to 0} \frac{\nu_t^{h+\lambda}-\nu_t^h}{\lambda}$ exists and is an element of $\ZZZ$. It means that the derivative $\partial_h\nu_t^h$ exists in all points $h\in(-\tfrac{1}{2},\tfrac{1}{2})$. 

First, it will be shown that $\big(h\mapsto\nu_\bullet^h\big)\in\CC\big((-\tfrac{1}{2},\tfrac{1}{2});\CC_{\widehat{\omega}}([0,\infty);\ZZZ)\big)$, it is the {\bf 0th Step} of the proof. In {\bf 1st Step} we will show that $\lim_{\lambda\to 0} \frac{\nu_\bullet^{h+\lambda}-\nu_\bullet^h}{\lambda}$ converges uniformly in $t$ for fixed $h$ as $\lambda\to 0$, so the mapping $t\mapsto\partial_h \nu_t^h$ is continuous and is bounded in norm $\|\cdot\|_{\ZZZ_{\widehat{\omega}}}$. 
In {\bf 2nd Step} we will show that the limit $\lim_{\lambda\to 0} \frac{\nu_\bullet^{h+\lambda}-\nu_\bullet^h}{\lambda}$ is continuous with respect to~$h$.

We will use a few different weights, proper for a current estimation. 
Based on those weights we will establish the final one; see~\eqref{final_weight_oljee} on page~\pageref{final_weight_oljee}.
Notice that if the term $\sup_{t\geq 0}\omega_1(t)f(t)$ is bounded, then for any weight $\omega_2(t)$, such that $\omega_2(t)\leq\omega_1(t)$ for all $t$, the following estimate is true
\[\sup_{t\geq 0}\omega_2(t)f(t)\leq \sup_{t\geq 0}\omega_1(t)f(t).\]
Thus the final weight $\widehat{\omega}(t)$ will be the minimum of weights used in intermediate estimates.\\

\noindent{\bf 0th Step.} 
Take $h_1,h_2\in(-\tfrac{1}{2},\tfrac{1}{2})$ and $\varphi\in\CCA(\RRD)$. Then
\begin{equation}
\begin{split}
& \sup_{\|\varphi\|_{\CCA}\leq 1}|\langle \nu_t^{h_1},\varphi \rangle-\langle \nu_t^{h_2},\varphi \rangle|=\\
&=\sup_{\|\varphi\|_{\CCA}\leq 1}\left|\int_\RRD\varphi(X(t,x;h_1)) \exp\left(\int_0^{t}w_0(s,X(s,x;h_1))+h_1\:w_1(s,X(s,x;h_1))\:\dd s\right)\dd\nu_0(x)\right.\\
&\qquad -
\left.\int_\RRD\varphi(X(t,x;h_2)) \exp\left(\int_0^{t}w(s,X(s,x;h_2))+h_2\:w_2(s,X(s,x;h_2))\:\dd s\right)\dd\nu_0(x)
\right|.
\end{split}
\end{equation}
This can be estimated as in~\eqref{eq:estim_jednos}. Hence, one has that $\nu_\bullet^h$ is continuous with respect to parameter $h$. By Corollary~\ref{prop:exp_omega_1} one has that for all $t$ holds
\begin{equation*}\label{step_0}
\begin{split}
\|\nu_t^h\|_\ZZZ&\leq \exp\Big(t\:(\underbrace{\|w_0^+(\cdot,\cdot)\|+|h|\|w_1^+(\cdot,\cdot)\|_\infty}_{\textrm{(\footnotesize using fact that $h\in(-\tfrac{1}{2},\tfrac{1}{2})$)}})\Big)\|\nu_0\|_{TV}\\
&\leq \exp\big(t\:(\|w_0^+(\cdot,\cdot)\|+\|w_1^+(\cdot,\cdot)\|_\infty)\big)\|\nu_0\|_{TV} .
\end{split}
\end{equation*}
Hence $\nu_\bullet^h\in\CC_{\widehat{\omega}}([0,\infty);\ZZZ)$,  if only $\widehat{\omega}(t)\leq \exp\big(-t(\|w_0^+\|_\infty+\|w_1^+\|_\infty)\big)$.

\noindent {\bf 1st Step.} Notice that for any fixed $h$ and fixed $\lambda$ mapping $\nu_\bullet^{h+\lambda}$ and $\nu_\bullet^h$ are locally Lipschitz continuous with respect to time. It implies that those mappings are continuous with respect to time. Hence, also $\frac{\nu_\bullet^{h+\lambda}-\nu_\bullet^h}{\lambda}$ is continuous with respect to $t$. We are going to show that for any $h$, the convergence of $\lim_{\lambda\to 0} \frac{\nu_\bullet^{h+\lambda}-\nu_\bullet^h}{\lambda}$ as $\lambda\to 0$ is a uniform limit in $t$. 
Then this limit would be continuous with respect to~$t$.

First, we will show this for $h=0$. Let $\{\lambda_n\}_n$ be a sequence in $\RR$ such that $\lambda_n\to 0$. We show that the sequence of quotients $\left\{\frac{\nu_\bullet^{\lambda_n}-\nu_\bullet^0}{\lambda_n}\right\}_{n\in\NN}$ is a Cauchy sequence in $\CC_{\widehat{\omega}}([0,\infty);\ZZZ)$ for a~suitable weight function $\widehat{\omega}(t)>0$. Let us consider
\[\left\Vert  \frac{\nu_\bullet^{\lambda_1}-\nu_\bullet^0}{\lambda_1}\right. - \left.\frac{\nu_\bullet^{\lambda_2}-\nu_\bullet^0}{\lambda_2}\right\Vert_{\ZZZ_{\widehat{\omega}}}.\]
Use the same notation which was introduced in proof of Theorem~\ref{th:main}; see~\eqref{eqn:final_Cauchy} and~\eqref{eq:estim_jednos}. Hence, we consider
 \[\sup_{t\geq 0}\widehat{\omega}(t)I_{\lambda_1,\lambda_2}\leq\sup_{t\geq 0}\widehat{\omega}(t)\left\{\sup_{\|\varphi\|_{\CCA}\leq 1}|I_{\lambda_1}^{(1)}-I_{\lambda_2}^{(1)}|\right\}  + \sup_{t\geq 0}\widehat{\omega}(t)\left\{ \sup_{\|\varphi\|_{\CCA}\leq 1} |I_{\lambda_1}^{(2)}-I_{\lambda_2}^{(2)}|\right\}.\]
Let us notice that in estimation~\eqref{eq:I2_final} on $|I_{\lambda_1}^{(2)}-I_{\lambda_2}^{(2)}|$ there is given explicit dependence on time. It means that for a~weight $\omega_1(t):=1\slash(\overline{C}(t)+C_1t)$ the following holds
\[\sup_{t\geq 0}\omega_1(t)\left\{\sup_{\|\varphi\|_{\CCA}\leq 1}|I_{\lambda_1}^{(2)}-I_{\lambda_2}^{(2)}|\right\} \leq \Big|\int_{\RRD} \!\!\left(-\frac{\mathcal{O}(|\lambda_1|^{1+\alpha})}{\lambda_1}+ \frac{\mathcal{O}(|\lambda_2|^{1+\alpha})}{\lambda_2}\right)\dd \nu_0(x)\Big|\xrightarrow{\lambda_1,\lambda_2\to 0} 0.\]
We need to estimate $\sup_{t\geq 0}\widehat{\omega}(t)\left\{\sup_{\|\varphi\|_{\CCA}\leq 1} |I_{\lambda_1}^{(1)}- I_{\lambda_2}^{(1)}|\right\}$. From~\eqref{eq:estim_WWW} we have
\begin{equation*}\label{eq:estim_WWW_2}
\begin{split}
&\sup_{t\geq 0}\widehat{\omega}(t)\left\{\sup_{\|\varphi\|_{\CCA}\leq 1} |I_{\lambda_1}^{(1)}- I_{\lambda_2}^{(1)}|\right\}\leq\\
& \leq \sup_{t\geq 0}\widehat{\omega}(t) \Big|\int_{\RRD}\Big[\partial_h X(t, x;h)\big|_{h=0}\left( e^{\int_0^t \overline{w}(s, x; \lambda_1)\dd s}-e^{\int_0^t \overline{w}(s, x; \lambda_2)\dd s} \right) \\
& \quad + (1+t\Vert b_1 \Vert_{\infty} e^{C t})\left( \mathcal{O}_1(|\lambda_1|^\alpha)e^{\int_0^t \overline{w}(s, x; \lambda_1)\dd s} -\mathcal{O}_1(|\lambda_2|^\alpha)e^{\int_0^t \overline{w}(s, x; \lambda_2)\dd s} \right)
\Big] \dd \nu_0(x)\Big|\\
& \leq\sup_{t\geq 0}\widehat{\omega}(t) \Big|\int_{\RRD}\underbrace{\partial_h X(t, x;h)\big|_{h=0}}_{{\bf L1'}}\underbrace{\left( e^{\int_0^t \overline{w}(s, x; \lambda_1)\dd s}-e^{\int_0^t \overline{w}(s, x; \lambda_2)\dd s} \right)}_{{\bf L2'}}\dd\nu_0(x)\Big| \\
& \quad+  \underbrace{\sup_{t\geq 0}\widehat{\omega}(t) \left(1+t\Vert b_1 \Vert_{\infty} e^{C t}\right)\Big|\int_{\RRD}\left( \mathcal{O}_1(|\lambda_1|^\alpha)e^{\int_0^t \overline{w}(s, x; \lambda_1)\dd s} -\mathcal{O}_1(|\lambda_2|^\alpha)e^{\int_0^t \overline{w}(s, x; \lambda_2)\dd s} \right)
 \dd \nu_0(x)\Big|}_{{\bf L3'}}.
\end{split}
\end{equation*}
Analogous arguments  to {\bf L1}, {\bf L2} and {\bf L3 } on page~\pageref{eq:estim_WWW} will also work here.
\begin{description}
\item[\bf L1'] For $\omega_2(t):=1\slash(t\Vert b_1 \Vert_{\infty}  \exp(Ct))$ value $\sup_{t\geq 0}\omega_2(t)\partial_h X(t, x;h)|_{h=0}$ is bounded -- see Proposition~\ref{claim:esti_X}. We stress that the constant $C$ that occurs in $\omega_2$ does not depend on time nor parameter $h$.

\item[\bf L2'] The following estimation holds \[\left( e^{\int_0^t \overline{w}(s, x; \lambda_1)\dd s}-e^{\int_0^t \overline{w}(s, x; \lambda_2)\dd s}\right)\leq (\overline{C}(t)+C_1t\big) \left(\frac{\mathcal{O}(|\lambda_1|^{1+\alpha})}{\lambda_1}+ \frac{\mathcal{O}(|\lambda_2|^{1+\alpha})}{\lambda_2}\right).\] What was obtained by expanding $e^{\int_0^t \overline{w}(s, x; \lambda_1)\dd s}$ and $e^{\int_0^t \overline{w}(s, x; \lambda_2)\dd s}$ into Taylor series at $h=0$; exactly as in estimation for $|I_{\lambda_1}^{(2)}- I_{\lambda_2}^{(2)}|$; see~\eqref{eq:I2_final}. This implies that
\[\sup_{t\geq 0}\omega_3(t)\left( e^{\int_0^t \overline{w}(s, x; \lambda_1)\dd s}-e^{\int_0^t \overline{w}(s, x; \lambda_2)\dd s}\right)\leq  \left(\frac{\mathcal{O}(|\lambda_1|^{1+\alpha})}{\lambda_1}+ \frac{\mathcal{O}(|\lambda_2|^{1+\alpha})}{\lambda_2}\right),\]
when as $\omega_3(t)$ we take $1\slash(2\overline{C}(t)+C_1t)$.

\item[\bf L3'] Recall that $\overline{w}(s,x;h)=w_0(s,X(s,x;h))+hw_1(s,X(s,x;h))$; see~\eqref{eq:w_defi_linear}. Both functions $w_0,$ $w_1$ are bounded and $h\in(-\tfrac{1}{2},\tfrac{1}{2})$. Let us substitute $C_2:=\|w_0\|_\infty+\|w_1\|_\infty$, where the supremum is taken with respect to $t$ and $x$. Constant $C_2$ does not dependent on~$t$ nor parameter~$h$. One has
\begin{equation}\label{eq:final_weight_lin}
\begin{split}
&\sup_{t\geq 0}\omega_4(t)(1+t\Vert b_1 \Vert_{\infty} e^{C t})\left( \mathcal{O}(|\lambda_1|^\alpha)e^{\int_0^t \overline{w}(s, x; \lambda_1)\dd s} -\mathcal{O}(|\lambda_2|^\alpha)e^{\int_0^t \overline{w}(s, x; \lambda_2)\dd s} \right)\leq\\
&\leq\sup_{t\geq 0}\omega_4(t)(1+t\Vert b_1 \Vert_{\infty} e^{C t})\left( \mathcal{O}(|\lambda_1|^\alpha)e^{C_2t} +\mathcal{O}(|\lambda_2|^\alpha)e^{C_2t} \right).\\
&\textrm{\footnotesize{(let $\omega_4(t):=1\slash\big[(1+t\Vert b_1 \Vert_{\infty} e^{C t})e^{C_2 t}\big]$\;)}
}\\
&\leq\sup_{t\geq 0}\Big( \mathcal{O}(|\lambda_1|^\alpha) +\mathcal{O}(|\lambda_2|^\alpha) \Big)=\Big( \mathcal{O}(|\lambda_1|^\alpha) +\mathcal{O}(|\lambda_2|^\alpha) \Big).
\end{split}
\end{equation}
Of course holds $\left(\mathcal{O}(|\lambda_1|^\alpha) +\mathcal{O}(|\lambda_2|^\alpha) \right)\xrightarrow{\lambda_1,\lambda_2\to 0} 0$.
\end{description}
Let us write down all weights used till this moment:
\begin{align*}\label{crazy_wagi}
\begin{split}
\omega_1(t)&=1\slash\Big[\big(\|\nabla_x w_0\|_\infty+1\big)\|b_1\|_\infty\left(\frac{t}{C}\exp(Ct)-\frac{1}{C^2}\exp(Ct)+\frac{1}{C^2}\right)+\|w_1\|_\infty t\Big]\\
&\qquad\qquad\qquad \textrm{see~\eqref{eq:holder_XXX} and~\eqref{eq:I2_final}, weights $\omega_1$ and $\omega_3$ are identical};\\
\omega_2(t)&=1\slash \Big[t\|b_1\|_\infty\exp(C t)\Big], \;\textrm{see Proposition~\ref{claim:esti_X}};\\
\omega_4(t)&=1\slash\Big[(1+t\Vert b_1 \Vert_{\infty} \exp(C t))\cdot\exp\left(\|w_0\|_\infty+|h|\cdot\|w_1\|_\infty t\right)\Big],\; \textrm{see estimation~\eqref{eq:final_weight_lin}};
\end{split}
\end{align*}
where constant $C:=\Vert \nabla_x b_0\Vert_{\infty}+ \Vert \nabla_x  b_1 \Vert_{\infty}$; see Proposition~\ref{claim:esti_X}. It means that for a~final weight we need to take function 
\begin{equation}\label{final_weight_oljee}
\widehat{\omega}(t):=\min\big(\omega_1;\;\omega_2\cdot\omega_3;\;\omega_4\big)(t)=\min\big(\omega_1;\;\omega_2\cdot\omega_1;\;\omega_4\big)(t).
\end{equation}
This weight guarantees that sequence of quotients $\left\{\frac{\nu_\bullet^{\lambda_n}-\nu_\bullet^0}{\lambda_n}\right\}_{n\in\NN}$ is a Cauchy sequence in the complete space ~$\ZZZ_{\widehat{\omega}}$. In other words: for any fixed $h_0$ the mapping $t\mapsto \partial_h\nu_t^{h}\big|_{h=h_0}$ is an element of $\CC_{\widehat{\omega}}([0,\infty);\ZZZ)$. This finishes the {\bf 1st Step} of the proof.\\
\medskip

\noindent{\bf 2nd Step. } Now, we show that $\partial_h\nu_\bullet^h=\lim_{\lambda\to 0} \frac{\nu_\bullet^{h+\lambda}-\nu_\bullet^h}{\lambda}$ is continuous with respect to~$h$. Notice that in every estimation in {\bf 1st Step} we choose constants in such a way, that they are independent of $h$. That was possible because we consider $h$ only in a bounded interval. That is why $\lim_{\lambda\to 0} \frac{\nu_\bullet^{h+\lambda}-\nu_\bullet^h}{\lambda}$ exists not only for any $h\in(-\tfrac{1}{2},\tfrac{1}{2})$, but also the convergence is uniform with respect to~$h$. Since $h\mapsto \nu_\bullet^h$ is continuous ({\bf 0th Step}) (with values in $\CC_{\widehat{\omega}}([0,\infty);\ZZZ)$) for each $\lambda$ we have shown that $h\mapsto\partial_h\nu_\bullet^h$ is continuous function, what finish the proof of Theorem~\ref{tw:ciagla_pocho}.
\end{proof}

\section{Proof of Proposition~\ref{claim:esti_X}}\label{apdx}
\hypertarget{foo}{Proof of Lemma}~\ref{claim:esti_X} from page~\pageref{claim:esti_X}, used in proof of Theorem \ref{th:main}, is presented below. We want to show that $\partial_h X$ is H\"older continuous.

\begin{proof}We analyse the difference $\partial_t (\partial_h X(t,y; h))\big|_{h=h_1} - \partial_t (\partial_h X (t,y; h))\big|_{h=h_2}$, which can be written in the following form using \cite[Theorem 3.1, formula (3.3)]{ODE:1982}
\begin{equation*}
\begin{split}
\partial_t& \left(\partial_h X(t,y; h)\right)\Big|_{h=h_1} - \partial_t \left(\partial_h X (t,y; h)\right)\Big|_{h=h_2}=\\
=&\left[\nabla_x\Big(b(t, X(t,y;h_1)) + h_1 b_1(t, X(t,y;h_1))\Big) \cdot \partial_h X(t,y;h)|_{h=h_1} + b_1 \left(t, X(t,y;h_1)\right)\right]\\
&\qquad - \left[\nabla_x\Big(b(t, X(t,y;h_2)) + h_2 b_1(t, X(t,y;h_2))\Big) \cdot \partial_h X(t,y;h)|_{h=h_2} + b_1 (t, X(t,y;h_2))\right]
\\
=&\nabla_x\Big(b(t, X(t,y;h_1)) + h_1 b_1(t, X(t,y;h_1))\Big) \cdot \partial_h X (t,y;h)|_{h=h_1}\\
&\qquad-\nabla_x \Big(b(t, X(t,y;h_2)) + h_2 b_1(t, X(t,y;h_2))\Big) \cdot \partial_h X(t,y;h)|_{h=h_2}\\
& \qquad +b_1 (t, X(t,y;h_1)) - b_1(t, X(t,y;h_2))
 \end{split}
\end{equation*}
\begin{equation*}
\begin{split}
=&\underbrace{\nabla_x b(t, X(t,y;h_1))\cdot \partial_h X(t,y;h)|_{h=h_1} - 
\nabla_x b(t, X(t,y;h_2))\cdot \partial_h X(t,y;h)|_{h=h_2}}_{(\textrm{A})}
\\
& 
\qquad +\underbrace{h_1 \nabla_x b_1(t, X(t,y;h_1)) \cdot \partial_h X(t,y;h)|_{h=h_1} - 
h_2 \nabla_x b_1(t, X(t,y;h_2))  \cdot \partial_h X(t,y;h)|_{h=h_2}}_{(\textrm{B})}\\
& \qquad +\underbrace{b_1 (t, X(t,y;h_1)) - b_1(t, X(t,y;h_2))}_{(\textrm{C})} 
.
\end{split}
\end{equation*} 

Now we estimate (A), (\textrm{B}) and (C) separately.

An ingredient (A) may be written as
\begin{equation*}
\begin{split}
(\textrm{A}) =  &\partial_h X( t,y;h)|_{h=h_1} \cdot \nabla_x \Big(b(t, X(t,y; h_1))- b(t, X(t,y; h_2))\Big) \\
 & \qquad+ \nabla_x  b(t, X(t,y; h_2))  \cdot \left(\partial_h X( t,y;h)|_{h=h_1} - \partial_h X( t,y;h)|_{h=h_2}\right).
  \end{split}
\end{equation*}

Let us first observe that 
\[\Big|\nabla_x \Big(b(t, X(t,y; h_1))- b(t, X(t,y; h_2))\Big) \Big| \leq c\left|X(t,y; h_1)- X(t,y; h_2)\right|^{\alpha}\] and   \[\Big|\nabla_xb(t, X(t,y; h_2)) \Big| \leq \sup_{y \in \RRD} \left|\nabla_x  b  (t, y)\right|.\] We proceed to show that  $\partial_h X( t,y;h)|_{h=h_1}$ is bounded. 
Since
\[\left| \partial_t\left(\partial_h X(t,y; h)\right)\right|=\left|\nabla_x  \Big(b(t, X(t,y; h))+ h b_1(t, X(t,y; h))\Big) \partial_h X(t,y;h) + b_1(t, X(t,y; h))\right|,
\]
thus
\begin{equation*}
\begin{split}
&\left| \int_{t_0}^t \partial_s\left(\partial_h X(s,y;h)\right) \dd s\right|\\
&\qquad \qquad =\left| \int_{t_0}^t\left(\nabla_x  \Big(b(s, X(s,y; h))+ h b_1(s, X(s,y; h))\Big) \partial_h X(s,y; h) + b_1(s, X(s,y; h)) \right)\dd s\right|.
\end{split}
\end{equation*}

Note that $\partial_h X(t_0,y; h)=0$ (initial condition in system (\ref{def:diff_pert}) does not depend on $h$), therefore
\begin{equation*}
\begin{split}\label{eqn:estim}
\left|\partial_h X( t,y;h)\right|\leq& \int_{t_0}^t \Big|\nabla_x  \Big(b(s, X(s,y; h))+ h b_1(s, X(s,y; h))\Big) \Big| \cdot \left|\partial_h X(s,y; h)\right| \dd s \\
&+ \int_{t_0}^t \big|b_1(s, X(s,y; h))\big| \dd s \\
\leq &\left(\Vert\nabla_x  b \Vert_{\infty}+ |h|\cdot \Vert\nabla_x  b_1 \Vert_{\infty}\right) \cdot \int \limits_{t_0}^{t} \left|\partial_h X(s,y; h)\right|\dd s +(t-t_0)\Vert b_1\Vert_{\infty}.
\end{split}
\end{equation*}
Observe that the supremum norm in the above expression is taken over space and time. 

 Since $(t-t_0)\Vert b_1\Vert_{\infty}$ in non-decreasing on $[t_0, \infty)$, thus by  Gronwall's Inequality we obtain
\[\left|\partial_h X( t,y;h)\right|\leq (t-t_0)\Vert b_1 \Vert_{\infty}e^{\int_{t_0}^t C \dd s} =(t-t_0)\Vert b_1 \Vert_{\infty} e^{C(t-t_0)},\]
where $C=\Vert \nabla_xb\Vert_{\infty}+ |h|\Vert \nabla_x b_1 \Vert_{\infty}.$

We are ready to estimate
\begin{equation*}\label{eqn:A}
\begin{split}
|(\textrm{A})| \leq  &\left|\partial_h X( t,y;h)|_{h=h_1}\right| \cdot \Big|\nabla_x \Big(b(t, X(t,y; h_1))- b(t, X(t,y; h_2))\Big) \big|\\
 & \qquad+ \left|\nabla_x  b(t, X(t,y; h_2)) \right| \cdot \left|\partial_h X( t,y;h)|_{h=h_1} - \partial_h X( t,y;h)|_{h=h_2}\right|
  \end{split}
\end{equation*}
\begin{equation*}
\begin{split}
\leq & c_1 |X(t,y; h_1)- X(t,y; h_2)|^{\alpha} + c_2 \left|\partial_h X( t,y;h)|_{h=h_1} - \partial_h X( t,y;h)|_{h=h_2}\right|,
\end{split}
\end{equation*}
where the constants $c_1, c_2$ can be taken independent of $y$ (but possibly dependent on $t$).

Next we estimate (\textrm{B})
\begin{equation*}
\begin{split}
|(\textrm{B})|= & \Big|h_1 \nabla_x  b_1(t, X(t,y; h_1))  \cdot \left(\partial_h X(t,y; h)|_{h=h_1}  - \partial_h X(t,y; h)|_{h=h_2}\right)
\\
&\qquad+ h_2 \partial_h X(t,y; h)|_{h=h_2}  \cdot \Big(\nabla_x  b_1(t, X(t,y; h_1))  - \nabla_x  b_1(t, X(t,y; h_2)) \Big)\\
&\qquad + \nabla_x  b_1(t, X(t,y; h_1))  \cdot \partial_h X(t,y; h)|_{h=h_2} (h_1 - h_2)\Big|\\
\leq & |h_1|\cdot \Vert\nabla_x  b_1  \Vert_{\infty} \cdot \left|\partial_h X(t,y; h)|_{h=h_1} - \partial_h X(t,y; h)|_{h=h_2}\right| \\
& \qquad + |h_2| \cdot c_3 \cdot \left|X(t,y; h_1) - X(t,y; h_2)\right|^{\alpha} + |h_1-h_2| \cdot \left\Vert\nabla_x  b_1 \right\Vert_{\infty} \cdot c_4 \\
\leq & c_5 \left|\partial_h X(t,y; h)|_{h=h_1} - \partial_h X(t,y; h)|_{h=h_2}\right|+c_6 |X(t,y; h_1) - X(t,y; h_2)|^{\alpha} + c_7 |h_1 - h_2|
\end{split}
\end{equation*}

and finally $(\textrm{C})$
\[
(\textrm{C}) \leq \big|  b_1 (t, X(t,y; h_1)) - b_1(t, X(t,y; h_2)) \big| \leq c_8|X(t,y; h_1) - X(t,y; h_2)|.
\]

Sum (\textrm{A}) + (\textrm{B}) + (\textrm{C}) can be written as
\begin{align*}
\left|(\textrm{A}) + (\textrm{B}) + (\textrm{C})\right| \leq & C_1|X(t,y; h_1) -  X(t,y; h_2)|^{\alpha} + C_2|X(t,y; h_1) - X(t,y; h_2)| \\
&+ C_3\left|\partial_h X( t,y;h)|_{h=h_1} - \partial_h X( t,y;h)|_{h=h_2}\right|+  C_4|h_1-h_2|.
\end{align*}

The following integral can be estimated as below
\[
 \int_{t_0}^t \left(\partial_s \partial_h X(s, y; h)|_{h=h_1} - \partial_s\partial_h X(s, y; h)|_{h=h_1}\right) \dd s = \int_{t_0}^t (\textrm{A})+ (\textrm{B})+ (\textrm{C}) \dd s
 \]
\begin{align*} 
 \partial_h X(t, y; h)|_{h=h_1} - \partial_h X(t, y,;h)|_{h=h_2} - \underbrace{\partial_h X(t_0,y;h)|_{h=h_1}}_{0} &+ \underbrace{\partial_h X(t_0, y;h)|_{h=h_1}}_{0} \\
 &= \int_{t_0}^t (\textrm{A})+ (\textrm{B})+ (\textrm{C}) \dd s.
\end{align*}

Hence
\begin{align*}\label{claim:esti_Xation2}
\left|\partial_h\right.& X(t, y; h)|_{h=h_1}- \left.\partial_h X(t, y; h)|_{h=h_2} \right|  \leq   C_1\int_{t_0}^t|X(s,y;h_1) - X(s,y;h_2)|^{\alpha}\dd s\\
&+  C_2\int_{t_0}^t|X(s,y;h_1)\! - X(s,y;h_2)|\dd s+  C_3\int_{t_0}^t\left|\partial_h X(s,y;h)|_{h=h_1} - \partial_h X(s,y;h)|_{h=h_2}\right|\dd s\\
&+  C_4|h_1-h_2|(t-t_0).
\end{align*}
Our goal is to conclude that $\partial_h X(t, y; h_1) - \partial_h X (t, y; h_2)$ is of order $|h_1 - h_2|^\alpha$. To close these estimations (using again Gronwall's Inequality) we need to estimate $|X(t,y; h_1) - X(t,y; h_2)|$. From (\ref{def:diff_pert})
\begin{align*}
\partial_t X(t, y; h_1) - \partial_t X (t, y; h_2) & = b (t, X(t,y; h_1)) - b(t, X(t,y; h_2)) \\
& + h_1 b_1 (t, X(t,y; h_1)) - h_2 b_1(t, X(t,y; h_2)).
\end{align*}
after integration, we obtain 
\begin{align*}
\int_{t_0}^t\left( \partial_s X (s, y; h_1) - \partial_s X (s, y; h_2)\right) \dd s & = 
\int_{t_0}^t\Big(  b  (s, X(s,y;h_1)) - b(s, X(s,y;h_2))\\
& + h_1 b_1 (s, X(s,y;h_1)) - h_2 b_1(s, X(s,y;h_2)) \Big) \dd s
\\
|X(t,y; h_1) - X(t,y; h_2)| & \leq \int_{t_0}^t  \Big|b  (s, X(s,y;h_1)) - b(s, X(s,y;h_2))\\
& + h_1 b_1 (s, X(s,y;h_1)) - h_2 b_1(s, X(s,y;h_2)) \Big| \dd s 
\\
& \leq \int_{t_0}^t \underbrace{| b  (s, X(s,y;h_1)) - b(s, X(s,y;h_2))|}_{(\textrm{D})} \dd s \\
& + \int_{t_0}^t \underbrace{|h_1 b_1 (s,  X(s,y;h_1)) - h_2 b_1(s, X(s,y;h_2))|}_{(\textrm{E})}  \dd s.
\end{align*}
Now, we need to estimate $(\textrm{D})$ and $(\textrm{E})$ separately
\begin{align*}
(\textrm{D}) &  \leq c_1'| X(s,y;h_1) - X(s,y;h_2)|,\\
(\textrm{E}) & = |h_1\left[ b_1 (s,  X(s,y;h_1)) - b_1(s, X(s,y;h_2))\right]  + b_1(s, X(s,y;h_2))(h_1-h_2)| \\
 &\leq |h_1| \cdot | b_1 (s,  X(s,y;h_1)) - b_1(s, X(s,y;h_2))|  +  |b_1(s, X(s,y;h_2)| \cdot |h_1-h_2|\\
 &\leq c_2' \cdot |X(s,y;h_1) -  X(s,y;h_2)|  +  c_3' \cdot |h_1-h_2| .
\end{align*}
Estimations $(\textrm{D})$ and $(\textrm{E})$  combined give
\begin{align*}|X(t,y; h_1) - X(t,y; h_2)| & \leq   \int_{t_0}^t c_3' |h_1-h_2| \dd s +\int_{t_0}^t c_4'| X(s,y;h_1) - X(s,y;h_2)| \dd s  =\\
& = c_3'|h_1 - h_2| (t-t_0) +  \int_{t_0}^t c_4'| X(s,y;h_1) - X(s,y;h_2)| \dd s.
\end{align*}
Using Gronwall's Inequality we obtain
 \[|X(t,y; h_1) - X(t,y; h_2)| \leq c_3'|h_1 - h_2|(t-t_0)e^{c_4'(t-t_0)}.\]
Let us go back to $\partial_t \left(\partial_h X( t,y;h)|_{h=h_1}\right) - \partial_t \left(\partial_h  X( t,y;h)|_{h=h_2}\right)$ 
\[
\left|\int_{t_0}^t \partial_s \left(\partial_h X (s,y;h)|_{h=h_1} - \partial_h X (s,y;h)|_{h=h_2}\right) \dd s\right| \leq \int_{t_0}^t |(\textrm{A}) + (\textrm{B}) + (\textrm{C})| \dd s\]
and thus
\begin{small}
\begin{equation*}
\begin{split}
&\left|\partial_h X( t,y;h)|_{h=h_1} -  \partial_h X( t,y;h)|_{h=h_2}\right|
\leq \int_{t_0}^t \bigg[ C_1| X(s,y;h_1) - X(s,y;h_2)|^{\alpha} \\
&\qquad+ C_2| X(s,y;h_1) - X(s,y;h_2)|+ C_3\left|\partial_h X(s,y;h)|_{h=h_1} - \partial_h X (s,y;h)|_{h=h_2}\right|+  C_4|h_1-h_2|\Big] \dd s \\
&\leq C_1\int_{t_0}^t |h_1 - h_2|^{\alpha} |s-t_0|^{\alpha} e^{\alpha c_4'(s-t_0)} \dd s + C_2\int_{t_0}^t  |h_1 - h_2||s-t_0| e^{c_4'(s-t_0)}\dd s
\\
& \qquad+ C_3\int_{t_0}^t \left|\partial_h X(s,y;h)|_{h=h_1} - \partial_h X(s,y;h)|_{h=h_1} \right| \dd s + C_4\int_{t_0}^t |h_1 - h_2|\dd s\\
&\leq C_1|h_1 - h_2|^{\alpha}\gamma_\alpha(t-t_0)+ C_2|h_1 - h_2| \gamma_1(t-t_0) \\
&\qquad + C_3 \int_{t_0}^t \left| \partial_h X (s,y;h)|_{h=h_1} - \partial_h X(s,y;h)|_{h=h_2}\right| \dd s + C_4|h_1-h_2|(t-t_0),
\end{split}
\end{equation*}
\end{small}
where $\gamma_\alpha(t)\colonequals \int_{t_0}^t s^\alpha e^{c_4' \alpha s}\dd s$. Gronwall's Inequality then shows that $\partial_h X$ is H\"older continuous of order $\alpha$.
\end{proof}

\end{document}